\documentclass{article}

\usepackage[margin = 1.3in]{geometry}

\usepackage{tikz}
\usepackage{tikz-cd}
\usepackage{url, hyperref}
\usepackage{float}

\tikzset{->-/.style={decoration={
  markings,
  mark=at position .45 with {\arrow{>}}},postaction={decorate}}}
\usepackage{amsmath}
\usepackage{amssymb}
\usepackage{enumitem}
\usepackage{graphicx}
\usepackage{mathdots}
\usepackage{color}
\usepackage{diagbox}
\usepackage{array, makecell}
\usepackage{rotating}
\usepackage{amsthm}
\usepackage{dsfont}
\usepackage{adjustbox}
\usetikzlibrary{arrows}

\theoremstyle{definition}
\newtheorem{definition}{Definition}[section]
\newtheorem{example}[definition]{Example}
\newtheorem{remark}[definition]{Remark}
\newtheorem{notation}[definition]{Notation}

\theoremstyle{theorem}
\newtheorem{theorem}[definition]{Theorem}
\newtheorem{proposition}[definition]{Proposition}
\newtheorem{corollary}[definition]{Corollary}
\newtheorem{lemma}[definition]{Lemma}

\def\O{\mathcal{O}}
\def\Q{\mathbb{Q}}
\def\Z{\mathbb{Z}}

\def\R{\mathbb{R}}
\def\H{\mathsf{H}}

\def\P{\mathbb{P}}
\def\M{\mathcal{M}}
\def\R{\mathcal{R}}
\def\CH{\mathrm{CH}}
\def\Sym{\mathrm{Sym}}
\def\RH{\mathcal{RH}}
\def\Sym{\mathrm{Sym}}
\def\det{\mathrm{det}}
\def\Gm{\mathbb{G}_m}
\def\GL{\mathrm{GL}}
\def\GI{(\Gm \times \Gm) \rtimes \mu_2}
\def\GII{\Gm \times (\Gm \rtimes \mu_2)}
\def\PGL{\mathrm{PGL}}
\def\SL{\mathrm{SL}}
\def\cO{\mathcal{O}}
\def\uD{\underline{\Delta}}
\def\Pic{\mathrm{Pic}}
\def\cH{\mathcal{H}}
\def\cD{\mathcal{D}}
\def\cX{\mathcal{X}}
\def\cY{\mathcal{Y}}
\def\cL{\mathcal{L}}
\def\cV{\mathcal{V}}

\title{The Integral Chow Rings of the Moduli Stacks of Hyperelliptic Prym Pairs I}
\author{Alessio Cela and Alberto Landi}
\date{\vspace{-5ex}}

\begin{document}

\maketitle

\begin{abstract}

This paper is the first in a series dedicated to computing the integral Chow rings of the moduli stacks of Prym pairs. In this work, we compute the Chow ring for Prym pairs arising from a single pair of Weierstrass points and from at most $(g-1)/2 $ pairs when the genus $g$ of the curve is odd.

\end{abstract}

\tableofcontents

\section{Introduction}

Prym pairs have been extensively studied due to their deep connections with the theory of abelian varieties, curves, and admissible covers \cite{Mumford, Farkas, Bea, Beabis, DonagiI}. Roughly speaking, Prym pairs correspond to étale double covers of a smooth, projective curve, or equivalently, to non-trivial line bundles on the base curve that are torsion of order 2 in the Picard group.

The study of the Chow rings of the moduli stack of stable genus $g$ curves, $\mathcal{M}_g$, began with Mumford \cite{Mumford1983} and quickly attracted significant attention. Subsequent work, including \cite{Faber, Izadi, PenevVakil, SamHannah}, has computed these Chow rings up to genus 9.

On the other hand, the integral Chow ring of these moduli stacks are much less understood. Edidin and Graham introduced in \cite{EG98} the intersection theory of global quotient stacks with integer coefficients. It is a more refined invariant but as expected, the computations for the
Chow ring with integral coefficients of the moduli stack of curves are much harder than the ones
with rational coefficients. Vistoli \cite{Vis98} computed the Chow ring for $\mathcal{M}_2$, and only more recently has E. Larson \cite{Lar19}, and then Di Lorenzo and Vistoli \cite{DLVis-M2}, computed the Chow ring of the compactification $\overline{\mathcal{M}}_2$. The Chow ring of $\mathcal{M}_3$ is currently known with coefficients in $\mathbb{Z}[1/6]$ \cite{Per23}.

In contrast, the integral Chow ring of the stack $\cH_g$ of hyperelliptic curves is fully known \cite{EF09,FV11,DL18}, while in the pointed case, it has been partially computed \cite{Per22,Lan23,Lan24}. Similarly, the Chow rings of hyperelliptic curves with Weierstrass sections are also known \cite{EH22}.

Finally, the Chow ring of Prym pairs in genus 2 was computed in \cite{CIL24}. In this paper, we extend their work by computing the Chow ring for Prym pairs arising from a single pair of Weierstrass points, as well as for those arising from at most $(g-1)/2$ pairs when the genus $g$ is odd. More specifically, the moduli stack $\RH_g$ of hyperelliptic Prym pairs of genus $g$ is a disjoint union of moduli stacks $\RH_g^n$, indexed by $n \in {1, \ldots, \lfloor (g+1)/2 \rfloor}$. In this work, we focus on the cases $n=1$ for any $g$, and for odd $g$, we address the case where $n \leq (g-1)/2$.

\subsection*{Convention}
We work over an algebraically closed field $k$ of characteristic $0$ or bigger than $2g+2$.

In previous related works, starting from~\cite{Vis98}, the characteristic has always been assumed to be greater than $2g$ instead of $2g+2$. The assumption on the characteristic is used for the construction of the Chow envelope in \S\ref{sec: Chow envelope}, specifically in the proof of Lemma \ref{lem: Chow envelopes n>1}. 

For instance, when $g=2$ and with the notation of \cite[Lemma 3.2]{Vis98}, let $p$ be a point in $\Delta_r \smallsetminus \Delta_{r+1}$ be a point represented by a homogeneous polynomial $f \in K[x_0,x_1]$ of degree 6. Write $f = u^2 v$, where $v \in K[x_0, x_1]$ is square-free over $K$. Obviously, the degree of $u$ must be at most $r$, because otherwise $p$ would lie in $\Delta_{r+1}$.  Furthermore, if its degree is less than the characteristic, $v$ remains square-free over any field extension of $K$. As $v$ has degree less or equal than $6$, this is automatically true if $\mathrm{char}(k)\geq 7$, but it can fail if $\mathrm{char}(k)$ is only assumed to be greater than $4$ as the example illustrates.

\begin{example}\label{ex: characteristic}(Due to Bernardo Tarini)
Let $p$ be a prime and let $K=\overline{\mathbb{F}}_p(t)$. The binary form $f:=x(x^p+ty^p)\in K[x,y]$ corresponds to a point in $\P^{p+1}$ and splits as $f=x(x+t^{1/p} y)^p$ in $\overline{\mathbb{F}}_{p}(t^{1/p})$. By uniqueness of the factorization, $x^p+ty^p$ is irreducible over $\overline{\mathbb{F}}_p(t)$.
Taking $p=5$ and again with the notation of~\cite[Lemma 3.2]{Vis98}, this shows that $f\in\Delta_{2}$, but $f$ is not the image of any $(u,v)\in(\P^r\times\P^{6-2r})(K)$ under the morphism $\pi_r$, for any $r\in\{1,2,3\}$.
\end{example}

\subsection{Prym Pairs}

We begin by providing the precise definition of the stack $\RH_g$. 

\begin{definition}
    A Prym curve of genus $g$ is the datum of $(C, \eta, \beta)$ where $C$ is a smooth geometrically connected genus $g$ curve, $\eta \in \mathrm{Pic} (C)$ is non-trivial and $\beta: \eta^{\otimes 2} \to \mathcal{O}_C$ is an isomophism of invertible sheaves on $C$.
\end{definition}

\begin{definition}
    A family of Prym curves of genus $g$ is a smooth family of genus $g$ curves $f : C \to S$ with an invertible sheaf $\eta$ on $C$ and a isomorphism $\beta: \eta^{\otimes 2} \to \mathcal{O}_C$
    such that the restriction of these data to any geometric fiber of $f$ gives rise to a Prym curve.
\end{definition}

\begin{definition} \label{defn: Rg}
    The prestack $\R_g^{\mathrm{pre}}$ is defined over the Big étale site $\mathrm{Sch}_{\text{ét}}$
    as the category whose objects over a scheme $S$ are families $(C \to S, \eta, \beta)$ of genus $g$ Prym curves over $S$ and a morphism $(C \to S, \eta, \beta) \to (C' \to S', \eta', \beta')$ is the data of a cartesian diagram 
    \begin{equation}\label{eqn: morphism prestack}
    \begin{tikzcd}
    C\arrow{r}{\varphi} \rar\dar\drar[phantom, "\square"] & C' \arrow{d} \\
    S \arrow[r,"f"'] & S'
    \end{tikzcd}
    \end{equation}
    such that there exists a isomorphism $\tau: \varphi^* \eta' \to \eta$ \footnote{Observe that we are adopting the convention that the datum of $\tau$ is not included in the
    definition of a morphism in $\mathcal R_g^{\mathrm{pre}}$. }such that the diagram 
    \begin{equation}\label{eqn: diagram of sheaves}
    \begin{tikzcd}
    \varphi^* \eta'^{\otimes 2}\arrow{r}{\tau^{\otimes 2}} \arrow[d,"\varphi^*(\beta'^{\otimes 2})"'] &  \eta^{\otimes 2} \arrow{d}{\beta^{\otimes 2}} \\
    \varphi^* \mathcal{O}_{C'} \arrow{r} & \mathcal{O}_C
    \end{tikzcd}
    \end{equation}
    commutes.

    Finally, the moduli stack $\R_g$ is the stackification of $\R_g^{\mathrm{pre}}$.
\end{definition}

\begin{definition}
    The moduli stack $\RH_g$ of hyperelliptic Prym pairs of genus $g$ is the pullback of the hyperelliptic locus $\mathcal{H}_g \subseteq \mathcal{M}_g$ in the moduli stack of smooth genus $g$ curves under the natural forgetful map $\R_g \to \M_g$. 
\end{definition}

\subsection{Results}

In order to find presentations of the stack $\RH_g$, the main geometric input is \cite[Lemma 4.3]{Ver13}, which we recall here. Let $C$ be a smooth hyperelliptic curve of genus $g$ defined over our base field \footnote{ In \cite{Ver13} the field $k$ is assumed to be the field of complex numbers, but the result we are interested in remains true (with the same proof) over every algebraically closed field of characteristic different from $2$.} and assume $g \geq 2$.  Let $q : C \to \mathbb{P}^1$ be its associated double cover (well-defined up to an isomorphism of $\mathbb{P}^1$). Then the ramification divisor of $q$ is the set
$$
W = \{ w_1, \ldots , w_{2g+2} \}
$$
of the Weierstrass points of $C$. Let $H=q^* \mathcal{O}_{\mathbb{P}^1}(1)$ be the hyperelliptic line bundle of $C$ and for $n=1,\ldots, \lfloor (g+1)/2 \rfloor$ let $E_n$ be the family of effective reduced divisors supported on $2n$ distinct points of $W$. For $e \in E_n$ the line bundle $H^{\otimes n} \otimes \mathcal{O}_C(-e) $ is a non-trivial square root of $\mathcal{O}_C$, i.e. an element of $\mathrm{Pic}^0(C)[2] \smallsetminus \{ \mathcal{O}_C \}$. Let 
\[
\begin{tikzcd}
    \beta_n: E_n \arrow[r] & \mathrm{Pic}^0(C)[2] \smallsetminus \{ \mathcal{O}_C \}
\end{tikzcd}
\]
be the map sending $e$ to $H^{\otimes n} \otimes \mathcal{O}_C(-e)$ and $B_n= \beta_n(E_n)$.

\begin{lemma}\cite[Lemma 4.3]{Ver13}\label{lemma: idea}
    \begin{itemize}
        \item[(i)] The map $\beta_n$ is injective for $n \leq \lfloor \frac{g}{2} \rfloor$;
        \item[(ii)] for $n= \frac{g+1}{2}$ the map $\beta_n$ is $2:1$ over its image;
        \item[(iii)] we have a disjoint union $ \mathrm{Pic}^0(C)[2] \smallsetminus \{ \mathcal{O}_C \}= \bigsqcup_{1 \leq n \leq \frac{g+1}{2}} B_n$.
    \end{itemize}
\end{lemma}

It follows that 
$$
\RH_g = \bigsqcup_{1 \leq n \leq \frac{g+1}{2}} \RH_g^n
$$
where $\RH_g^n$ is the locus where $\eta= H^{\otimes n}(-e)$ for some $e \in E_n$.

In this paper, we provide explicit presentations of all $\RH_{g}^n$ as quotient stacks for $g \geq 2$ and $n = 1$, as well as for $g$ odd and $1 \leq n < (g+1)/2$. This, in particular, improves upon \cite[Theorem 4.4]{Ver13}. The remaining cases are deferred to a subsequent work currently in progress. Using these presentations, we also compute the integral Chow rings of these stacks.

\subsubsection{Case $g$ even and $n=1$}

Suppose that $g$ is even and that $n=1$. We start with some notation.

\begin{notation}
     Let $\GI \subseteq \GL_2$ be the subgroup of matrices preserving the set of lines $\{ k (1,0), k(0,1)\} \subseteq k^2$. We will denote by $V$ be the representation of $\GI$ induced by the standard representation of $\GL_2$ via the inclusion $\GI \subseteq \GL_2$.

    We will also denote by $\Gamma$ the representation of $\GI$ induced by the sign representation of $\mu_2$ via the quotient map $\GI \to \mu_2$.
\end{notation}  

Using the above notation, we can state our presentation of $\RH_g^1$. 

\begin{theorem}\label{thm: presentation g even n=1}
    Suppose the genus $g$ is even. Then, we have an isomorphism of algebraic stacks
    $$
    \RH_g^1 \cong \Bigg[ \frac{\Sym^{2g}(V^\vee) \otimes \det (V)^{\otimes g-1} \otimes \Gamma \smallsetminus \Delta}{ \GI }  \Bigg]
    $$
    where $\Delta$ is the locus of polynomials having either a root at $0$ or $\infty$, or having a double root.
\end{theorem}

The proof is similar to the case $g=2$ treated in \cite[Section 2.4]{CIL24} and it is presented in \S\ref{subsec:presn=1}.

In \S\ref{Sec: Chow n=1}, we use the machinery of equivariant intersection theory and the above presentation to compute the integral Chow ring of $\RH_g^1$ for even $g$. Before stating the result, we require some further notation.

\begin{notation}
    Following \cite{Lar19}, we set 
    \begin{align*}
        &\beta_i= c_i(V) \in \CH^*(B (\GI)), \\
        & \gamma= c_1(\Gamma) \in \CH^*(B (\GI)),
    \end{align*}
    where we denoted the vector bundles over $B(\GI)$ associated to the representations $V$ and $\Gamma$ using the same symbols.
\end{notation}

\begin{theorem}\label{thm: Chow for g even n=1}
    Suppose the genus $g$ is even. Then, we have
    $$
    \CH^*(\RH_g^1)=\frac{\mathbb Z[\beta_1, \beta_2, \gamma]}{(2\beta_1, 2\gamma, 4g \beta_2, \gamma (\gamma+\beta_1), \beta_1(\beta_1 + \gamma))}.
    $$    
\end{theorem}

Setting $g=2$, $\beta_i=\lambda_i$ for $i=1,2$, this recovers \cite[Theorem 4]{CIL24}. Finally, in \S\ref{subsec: interpretation generators n=1}, we give a geometric interpretation of the classes $\beta_1,\beta_2$ and $\gamma$ as Chern classes of certain natural vector bundles on $\RH_g^1$.

\subsubsection{Case $g$ odd and $n=1$}

Suppose that $g$ is odd and that $n=1$.

\begin{notation}\label{not: representations PGL2}
     Let $\GII \subseteq \Gm \times \PGL_2$ be the subgroup consisting of pairs $(t,[A])$ with $A$ a matrix preserving the set of lines $\{ k (1,0), k(0,1)\} \subseteq k^2$. For every $j \geq 1$, we will denote by $W_j$ the $2j+1$-dimensional representation of $\PGL_2$ where the underlying vector space is the space of homogeneous polynomials of degree $2j$ in two variables and 
     $$
     [A] \cdot f(X,Y)= \det (A)^j (f \circ A^{-1}) (X,Y).
     $$
     By an abuse of notation we will also denote by $W_j$ the induced representation of $\GII$ where $\Gm$ acts trivially and the action of $\Gm \rtimes \mu_2$ is given by the inclusion $\Gm \rtimes \mu_2 \subseteq \PGL_2$.

     We will also denote by $\Gamma$ the representation of $\GII$ induced by the sign representation of $\mu_2$ under the quotient map $\GII \to \mu_2$.

     Finally, we will denote by $\chi$ the standard $1$-dimensional representation of $\Gm$ and also, with an abuse of notation, the induced representation of $\GII$ via the projection onto the first factor $\GII \to \Gm$.
\end{notation}  

\begin{theorem}\label{thm: presentation g odd n=1}
    Suppose the genus $g$ is odd. Then, we have an isomorphism of algebraic stacks
    $$
    \RH_g^1 \cong \Bigg[ \frac{W_g \otimes \chi^{\otimes -2} \otimes \Gamma \smallsetminus \Delta}{ \GII }  \Bigg]
    $$
    where $\Delta$ is the locus of polynomials having either a root at $0$ or $\infty$, or having a double root.
\end{theorem}

The proof is presented in \S\ref{subsec:presn=1}. The idea, which applies to both $g$ even and odd, can be summarized as follows. Lemma \ref{lemma: idea} provides a map
\[
\begin{tikzcd}
    k[X,Y]_{2g} \arrow[r] & \RH_g^1
\end{tikzcd}
\]
obtained by associating to an homogeneous polynomial $f(X,Y)$ of degree $2g$ the following data:
\begin{itemize}
    \item the double cover $q:C \to \P^1$ of $\P^1$, ramified along the zeros of $f(X,Y)$ and the points $0$ and $\infty$;
    \item the $2$-torsion line bundle on $C$ given by $q^* \cO_{\P^1}(1)(-\sigma_0-\sigma_{\infty})$ where $\sigma_0,\sigma_\infty \in C$ are the points over $0,\infty \in \P^1$ respectively;
    \item the isomorphism $\eta^2 \xrightarrow{\sim} \cO_C$ is given by the section $XY \in H^0(\P^1,\cO_{\P^1}(2))$.
\end{itemize}
We will prove that this map is a $\GI$-torsor when 
$g$ is even and a 
$\GII$-torsor when $g$ is odd.

The computation of the integral Chow ring of $\RH_g^1$ for odd $g$ is presented in \S\ref{Sec: Chow n=1}. Before stating our result, we introduce some further notation and previous results.

First, recall that the Chow ring of the classifying stack $B \PGL_2$ is well-known \cite{pandharipande1996chow,Vez98,DL18}:

\begin{equation}
    \CH^*(B \PGL_2)= \frac{\Z[c_2,c_3]}{(2c_3)}.
\end{equation}

\begin{notation}
    In this paper, we will also denote by $c_i \in \CH^*(B \GII)$ ($i=2,3$) the pullback of the $c_i \in \CH^*(B \PGL_2)$  under $B(\GII) \to B\PGL_2$.
    
    Furthermore, we set 
    \begin{align*}
        &t= c_1(\chi) \in \CH^*(B (\GII)), \\
        & \gamma= c_1(\Gamma) \in \CH^*(B (\GII)).
    \end{align*}
\end{notation}

\begin{theorem}\label{thm: Chow for g odd n=1}
    Suppose the genus $g$ is odd. Then, we have
    $$
    \CH^*(\RH_g^1)=\frac{\mathbb Z[c_2, t, \gamma]}{( 2\gamma, 4 t, \gamma^2+ g c_2)}.
    $$
    
\end{theorem}

In \S\ref{subsec: interpretation generators n=1}, we give a geometric interpretation of the classes $c_2,t$ and $\gamma$ as Chern classes of certain natural vector bundles on $\RH_g^1$.

\subsubsection{Case $1 < n < (g+1)/2$ and $g$ odd}

Next, we address the case $1 < n < (g+1)/2$ with $g$ odd. Most of the results we derive also hold for $n=1$, so we will assume $1 \leq n < (g+1)/2$ and $g$ odd throughout this section, specifying when the condition $n \neq 1$ is required.

Let $\mathcal{H}_g$ denote the moduli stack of hyperelliptic curves of genus $g$, and, for $a \geq 1$ let $\mathcal{D}_{a}$ denote the moduli stack parametrizing pairs $(\mathcal{P} \to S, D_a \subseteq \mathcal{P})$, where $\mathcal{P} \to S$ is a Brauer--Severi scheme of relative dimension $1$, and $D_a$ is a Cartier divisor that is finite and \'etale over $S$ of degree $a$. 

As explained in \cite[Theorem 4.1]{AV04} and \cite[Remark 2.1]{EH22}, there exists a map 
\begin{equation}\label{eqn: map in EH}
\begin{tikzcd}
    \mathcal{H}_g \cong \bigg[ \frac{\chi^{\otimes -2} \otimes W_{g+1} \setminus \Delta}{\mathbb{G}_m \times \mathrm{PGL}_2} \bigg] \arrow[r] & \mathcal{D}_{2g+2} \cong \bigg[ \frac{\chi^{\otimes -1} \otimes W_{g+1} \setminus \Delta}{\mathbb{G}_m \times \mathrm{PGL}_2} \bigg] \cong \bigg[ \frac{\P(W_{g+1}) \setminus \underline{\Delta}}{ \mathrm{PGL}_2} \bigg],
\end{tikzcd}
\end{equation}
which, in terms of objects, records the $2g+2$ branch points of the double cover associated with the hyperelliptic curve. In terms of quotient stacks, this map is the base change of the morphism $B \Gm \to B\Gm$ given by $t \mapsto t^2$ via $[ (\chi^{\otimes -1} \otimes W_{g+1} \setminus \Delta) /\mathbb{G}_m \times \mathrm{PGL}_2 ] \to B\Gm$. In particular, $\cH_g$ is the $\mu_2$-root-stack over $\cD_{2g+2}$ obtained by adding a square root of the line bundle $\cO_{\P(W_{g+1})}(-1)$. The loci $\Delta$ and $\underline{\Delta}$ denote the locus of singular polynomials in the affine and projective spaces, respectively.

Let $\mathcal{D}_{a,b}$ denote the open substack of $\mathcal{D}_a \times_{B \mathrm{PGL}_2} \mathcal{D}_b$ parametrizing tuples $(\mathcal{P} \to S, D_a, D_b)$, where $D_a$ and $D_b$ are disjoint divisors. By definition, we have isomorphisms

\begin{equation}\label{eqn: presentation Dab}
\mathcal{D}_{2a,2b} \cong \bigg[ \frac{(\chi^{(1)})^{-1} \otimes W_a \times (\chi^{(2)})^{-1} \otimes W_b \setminus \Delta}{\mathbb{G}_m^{(1)} \times \mathbb{G}_m^{(2)} \times \mathrm{PGL}_2} \bigg] \cong \bigg[ \frac{\P (W_a) \times \P( W_b) \setminus \underline{\Delta}}{\mathrm{PGL}_2} \bigg]
\end{equation}

where $\Delta$ and $\underline{\Delta}$ correspond to the locus of pairs of polynomials $(f, g)$ such that the product $fg$ is singular, respectively seen in the affine and projective cases. Here, $\mathbb{G}_m^{(1)}$ and $\mathbb{G}_m^{(2)}$ are both equal to $\mathbb{G}_m$, but the superscript indicates that $\mathbb{G}_m^{(i)}$ acts only on the $i$-th term of the product. Similarly, $\chi^{(i)}$ represents the character $\chi$, with the superscript specifying which $\mathbb{G}_m$ acts on it.

The reason why we introduced the stack $\mathcal{D}_{a,b}$ is explained in the following.

\begin{lemma}\label{lemma: cartesian diagram for RH}
    There is a Cartesian diagram
    \[
    \begin{tikzcd}
            \RH_g^n \arrow[d]\arrow[r] \arrow[dr, phantom, "\square"] & \cH_g \arrow[d]\\
            \cD_{2n,2g+2-2n}\arrow[r] & \cD_{2g+2}
        \end{tikzcd}
    \]
    where the bottom horizontal map corresponds to taking the sum of $D_{2n}$ and $D_{2g+2-2n}$, while the left vertical map remembers the degree $2n$ divisor $e$ from Lemma \ref{lemma: idea} and the complement $W \smallsetminus e$ of $e$ in the Weierstrass divisor. 

    In particular, with the notation in Equation \eqref{eqn: presentation Dab}, the map $\RH_g^n \to \mathcal{D}_{2n,2g+2-2n}$ is the $\mu_2$-root-gerbe associated to the line bundle $\cO_{\P(W_{2n})}(-1)\boxtimes\cO_{\P(W_{2g+2-2n})}(-1) $.
\end{lemma}
\begin{proof}
    Given an object $(C/S, \eta, \beta) \in \RH_g^n(S)$, we can associate two well-defined disjoint Weierstrass divisors, one of degree $2n$ and another of degree $2g + 2 - 2n$. These are obtained by considering the full Weierstrass divisor and identifying a subset of degree $2n$ using Lemma \ref{lemma: idea}. This construction defines the map $\RH_g^n \to \cD_{2n,2g+2-2n}$ in the statement.  

    We thus have an induced map  
    $
    \RH_g^n \to \cD_{2n,2g+2-2n} \times_{\cD_{2g+2}} \cH_g,
    $  
    over $\cH_g$, which is clearly essentially surjective. This map is an isomorphism because both stacks are representable, étale, and finite over $\cH_g$, with the same degree (see \cite[Proposition 11]{CIL24}).   
\end{proof}

In Proposition \ref{prop: Chow of root gerbes}, we explain how to derive the Chow ring of $\mu_n$-root-gerbes. Using the second presentation of $\mathcal{D}_{2n,2g+2-2n}$ in Equation \eqref{eqn: presentation Dab}, in \S\ref{sec: Chow n>1}, we compute the Chow ring of $\mathcal{D}_{2n,2g+2-2n}$ and finally put everything together, obtaining the Chow ring of $\RH_g^n$. The results are as follows.

\begin{notation}
    We denote by $\xi_{2j} \in \CH^*_{\PGL_2}(\P(W_j))$ the $\PGL_2$-equivariant first Chern class of $\cO_{\P(W_j)}(1)$.
\end{notation}

\begin{theorem}\label{thm: Chow Dab}
    Suppose $a,b>1$. Then 
    $$
    \CH^*(\cD_{2a,2b})= \frac{\Z[c_1,c_2,c_3,\xi_{2a},\xi_{2b}]}{I}
    $$
    where $I$ is the ideal generated by the following relations:
    \begin{align*}
         \bullet & \ c_1, \ 2 c_3 \ \mathrm{ coming \ from \ } B\PGL_2 \\
        \bullet & \ 2(2a-1)\xi_{2a} \\
        \bullet & \ 2\xi_{2a}^2-a(2a-2) c_2 \\
        \bullet & \ \xi_{2a}^4+c_2 \xi_{2a}^2+c_3 \xi_{2a}+3a^2(a-1)^2c_2^2 \\
        \bullet & \ 2(2b-1)\xi_{2b} \\
        \bullet & \ 2\xi_{2b}^2-b(2b-2) c_2 \\
        \bullet & \ \xi_{2b}^4+c_2 \xi_{2b}^2+c_3 \xi_{2b}+3b^2(b-1)^2 c_2^2 \\
        \bullet & \  2b \xi_{2a}+2a \xi_{2b} \\
        \bullet & \ 2 \xi_{2a} \xi_{2b} -2abc_2 \\
        \bullet & \ (2a-1)(2b-1)\xi_{2a}\xi_{2b}+b(2b-1) \xi_{2a}^2+a(2a-1) \xi_{2b}^2+ ab (a+b-1) c_2 \\
        \bullet & \ \xi_{2a} \xi_{2b}(\xi_{2a}+\xi_{2b})+c_2(b\xi_{2a}+a\xi_{2b})+c_2(b^2 \xi_{2a}+a^2\xi_{2b})-ab c_3 \\
        \bullet & \ \xi_{2a}^2 \xi_{2b}^2+c_2(b^2\xi_{2a}^2+a^2\xi_{2b}^2)-c_2(b(2b-1)\xi_{2a}^2+a(2a-1)\xi_{2b}^2)-(2a-1)(2b-1)c_2 \xi_{2a} \xi_{2b} \\
        &\ +c_3(b \xi_{2a}+a \xi_{2b})+abc_2^2-ab(a+b-1)c_2^2.
        \end{align*}        
\end{theorem}

\begin{theorem}\label{thm: Chow n>1}
    Suppose $1< n < (g+1)/2$ and that $g$ is odd. Then 
    $$
    \CH^*(\RH_g^n)= \frac{\Z[c_1,c_2,c_3,t,\xi_{2n},\xi_{2g+2-2n}]}{I+(\xi_{2n}+\xi_{2g+2-2n}+2t)}
    $$
    where $I$ is the ideal generated by the relations listed in Theorem \ref{thm: Chow Dab} where we take $a=n$ and $b=g+1-n$.
\end{theorem}

The interpretation of these classes in geometric terms is presentated in \S\ref{subsec: interpretation generators n>1}

The computation of $\CH^*(\cD_{2a,2b})$ differs slightly when $a$, $b$, or both are equal to $1$, and as a consequence, the computation of $\CH^*(\RH_g^1)$ is also affected. In any case, it is possible to derive the Chow ring of $\RH_g^1$ for odd $g$ using this strategy, and the result matches that stated in Theorem \ref{thm: Chow for g odd n=1} after the substitutions $\xi_2=\gamma$ and $\xi_{2g}=2t-\gamma$. See also Remark \ref{rmk: a=1 case} below.

Finally, from Lemma \ref{lemma: cartesian diagram for RH}, we also obtain a presentation of $\RH_g^n$ as a quotient stack.

\begin{theorem}\label{thm: presentation n>1}
    Suppose $1 \leq n < (g+1)/2$ and that $g$ is odd. Then there is an isomorphism of stacks
    $$
    \RH_g^n \cong \bigg[ \frac{(\chi^{(1)})^{-1} \otimes W_{n} \times \chi^{(1)} \otimes (\chi^{(2)})^{-2} \otimes W_{g+1-n} \setminus \Delta}{\mathbb{G}_m^{(1)} \times \mathbb{G}_m^{(2)} \times \mathrm{PGL}_2} \bigg]
    $$
    where the convention with superscripts is as in Equation \eqref{eqn: presentation Dab}.
\end{theorem}

The proof is presented in \S\ref{sec: presentation n>1}.

\subsection{Acknowledgments}
We would like to thank Dan Edidin, Andrea Di Lorenzo, Eric Larson, Hannah Larson, Aitor Iribar López, Michele Pernice, Bernardo Tarini, and Angelo Vistoli for several fruitful discussions. We are especially grateful to Dan Edidin and Andrea Di Lorenzo for discussions about their previous work. We also thank Jarod Alper, Martin Bishop, Felix Janda, and Minseon Shin for discussions regarding the Chow ring of $\mu_n$-root gerbes in \S\ref{sec: Chow root gerbes}. Finally, we thank Bernardo Tarini again for pointing out Example~\ref{ex: characteristic}. A.C. acknowledges support from SNF grant P500PT-222363.

\section{Presentations}\label{sec:pres}

\subsection{Presentation of $\RH_g^1$}\label{subsec:presn=1}

In this section, we prove Theorems \ref{thm: presentation g even n=1} and \ref{thm: presentation g odd n=1}. The strategy is similar to that used in \cite[Theorem 2]{CIL24} for the case $g=2$. We define the map that gives the isomorphism. The verification that this map is indeed an isomorphism is analogous to the argument in \cite[Sections 2.3 and 2.4]{CIL24} and is therefore not repeated here.

The first step is realizing that the two groups $\GI $ and $\GII$ naturally arise in the same way as follows.

Let
$$
G \subseteq \frac{\GL_2}{\mu_{g+1}}= \underline{\mathrm{Aut}}(\P^1,\cO(g+1))
$$
be the subgroup of automorphisms of $(\P^1,\cO(g+1))$ fixing the set $\{ 0, \infty \} \subseteq \P^1$. 

\begin{lemma}\label{lemma: group G}
    The group $G$ is isomorphic to $\GI$ when $g$ is even and to $\GII$ when $g$ is odd.
\end{lemma}

\begin{proof}

The required isomorphism is obtained by restricting to $G$ the isomorphisms in \cite[Proposition 4.4]{AV04} 

\begin{equation}\label{eqn: iso of groups}
 \frac{\GL_2}{\mu_{g+1}} \xrightarrow{\sim} 
\begin{cases}
 \GL_2 &\text{when $g$ is even},\\
 \Gm \times \PGL_2  &\text{when $g$ is odd }\, 
\end{cases} 
\end{equation}
given by 
$$
 A \mapsto 
 \begin{cases}
 \det(A)^{\frac{g}{2}} A &\text{when $g$ is even},\\
 (\det(A)^{\frac{g+1}{2}},[A])  &\text{when $g$ is odd. }\, 
\end{cases} 
$$
\end{proof}

Denote by $X$ the affine space of homogeneous polynomials of degree $2g$ and by $\Delta \subseteq X$ the subscheme of those polynomials having a root at $0$ or $\infty$, or a double root. Depending on the parity of $g$, the group $G$ acts on $X$ via the isomorphism in Lemma \ref{lemma: group G} and the actions described in Theorems \ref{thm: presentation g even n=1} and \ref{thm: presentation g odd n=1} respectively. Let $\mathcal{X}^{\mathrm{pre}}=[ (X \smallsetminus \Delta) /G ]^{\mathrm{pre}}$ be the quotient prestack. There is a natural map
\begin{equation}\label{eqn: def morphism}
\begin{tikzcd}
    \alpha^{\mathrm{pre}}: \mathcal{X}^{\mathrm{pre}} \arrow[r] & \RH_g^1
\end{tikzcd}
\end{equation}
defined as follows. 

First we define it on objects. Given a morphism $F:S \to X \smallsetminus \Delta$, we map $F$ to the triple $(C/S,\eta,\beta)$ where 
\begin{equation}\label{eqn: def C}
C = \underline{\mathrm{Spec}}_{\mathbb{P}^1_S}(\mathcal{O}_{\mathbb{P}^1_S} \oplus \mathcal{O}_{\mathbb{P}^1_S}(-g-1)),
\end{equation}
with the $\mathcal{O}_{\mathbb{P}^1_S}$-algebra structure on $\mathcal{O}_{\mathbb{P}^1_S} \oplus \mathcal{O}_{\mathbb{P}^1_S}(-g-1)$ given by 
\[
\begin{tikzcd}
    \mathcal{O}_{\mathbb{P}^1_S}(-g-1) \otimes \mathcal{O}_{\mathbb{P}^1_S}(-g-1) \arrow[r,"XY.F"] & \mathcal{O}_{\mathbb{P}^1_S};
\end{tikzcd}
\]
here $XY$ is the homogenous monic polynomial of degree $2$ with zeros $0$ and $\infty$. The curve $C$ comes with a map
\begin{equation}\label{eqn: map q}
\begin{tikzcd}
    q: C\arrow[r] & \P_S^1,
\end{tikzcd}
\end{equation}
which admits two sections $\sigma_0,\sigma_{\infty}: S \to C$, corresponding to the $0$ and $\infty$ section of $\mathbb{P}^1_S \to S$. We set 
\begin{equation}\label{eqn: def eta}
    \eta= q^* \O_{\P_S}(1) \otimes \O_C(-\sigma_0-\sigma_\infty)
\end{equation}
and
\begin{equation}\label{eqn: def trivialization}
\begin{tikzcd}
    \eta^{\otimes 2} \arrow[r,"\sim"] & \O_C
\end{tikzcd}
\end{equation}
is given by $XY \in H^0(C,q^* (\O_{\P^1_S}(2) \otimes \O_C(-\sigma_0-\sigma_\infty)))$.

Now, we define the morphism \eqref{eqn: def morphism} at the level of morphisms. A morphism 
\[
\begin{tikzcd}
    (F: S \to X \smallsetminus \Delta) \arrow[r] & (F': S' \to X \smallsetminus \Delta)
\end{tikzcd}
\]
in $\mathcal{X}^{\mathrm{pre}}$ is by definition the data of a morphisms $f: S \to S'$ and $h:S \to G$ such that $F' \circ f = h \cdot F$. 

Let $(C/S, \eta,\beta)$ and $(C'/S', \eta',\beta')$ be the objects in $\mathcal{X}^{\mathrm{pre}}$ associated to $F$ and $F'$ respectively. Call  
\[
\begin{tikzcd}
    \phi= f \times h: \P_S^1 \arrow[r] & \P_{S'}^1.
\end{tikzcd}
\]
Here, we view $h$ as an automorphism of $\P^1$ via the isomorphism \eqref{eqn: iso of groups}. We have a morphism
\begin{equation}\label{eqn: def map on morphism part 1}
\begin{tikzcd}
    \varphi: C \arrow[r] & C'
\end{tikzcd}
\end{equation}
given by
$
\varphi^{\#}: \phi^* \O_{\P_{S'}^1}(-g-1) \to \O_{\P_{S}^1}(-g-1)
$
defined at $(s,\ell) \in \P_S^1=S \times \P^1$ by 
$$
 (\varphi^{\#}(\phi^*(v))) (x,\ell)= h(s)^{-1}. \ v (f(s),g(s). \ell)
$$

Again, $G$ acts on $(\P^1, \O(-g-1))$ via the isomorphism in \eqref{eqn: iso of groups}. Then, the morphism
\[
\begin{tikzcd}
    \phi^*(\O_{\P_{S'}^1} \oplus \O_{\P_{S'}^1}(-g-1)) \arrow[r] & \O_{\P_{S}^1} \oplus \O_{\P_{S}^1}(-g-1)
\end{tikzcd}
\]
induced by $\varphi^{\#}$ is an homomorphism of $\O_{\P_S^1}$-algebras. This depends on the action of $G$ on $X$ and reduces to the following two statements:

\begin{itemize}
    \item when $g$ is even, an element $A= \begin{pmatrix}
    a & 0 \\
    0 & b 
    \end{pmatrix}$ or $A=
    \begin{pmatrix}
    0 & a \\
    b & 0 
    \end{pmatrix}\in \GI$ corresponds via the isomorphism \eqref{eqn: iso of groups} to $\det (A)^{-g/(2g+2)}A \in\GL_2/\mu_{g+1}$ which acts on the space of homogeneous polynomials of degree $2g+2$ by precomposition with its inverse. We have
    \begin{equation}\label{eqn: map to Hg g even}
    \bigg( XY f(X,Y) \bigg) \circ \bigg( \det (A)^{-g/(2g+2)} A  \bigg) ^{-1} = XY \bigg( A \cdot f(X,Y) \bigg);
    \end{equation}
    \item when $g$ is odd, an element $(t,[A]) \in \GII$ corresponds via the isomorphism \eqref{eqn: iso of groups} to 
    $ t^{1/(g+1)} \det(A)^{1/2} A \in \GL_2/ \mu_{g+1}$ which again acts on the space of homogeneous polynomials of degree $2g+2$ by precomposition with its inverse. We have 
    \begin{equation}\label{eqn: map to Hg g odd}
    \bigg( XY f(X,Y) \bigg) \circ \bigg( t^{1/(g+1)} \det(A)^{-1/2} A \bigg) ^{-1} = XY \bigg( (t,[A]) \cdot f(X,Y) \bigg)
    \end{equation}
\end{itemize}

To conclude the definition of $\alpha^{\mathrm{pre}}$ on morphisms we need to provide, étale locally on $S$, a morphism $\tau$ between $\varphi^* \eta'$ and $\eta$ making the diagram \eqref{eqn: diagram of sheaves} to commute. This is done as in \cite[Lemma 23]{CIL24}.

The stackification of the map \eqref{eqn: def morphism} is the required isomorphism $\mathcal{X} \xrightarrow{\sim} \RH_g^1$. We refer to \cite[Section 2.4.]{CIL24} for the verification of this statement.

\subsection{Presentation of $\RH_g^n$ for $1<n<(g+1)/2$ and $g$ odd}\label{sec: presentation n>1}

In this section, we prove Theorem \ref{thm: presentation n>1}. It is worth noting that this result will not be used elsewhere in this work.

It follows from Equations \eqref{eqn: map in EH} and \eqref{eqn: presentation Dab} and Lemma \ref{lemma: cartesian diagram for RH} that there is a cartesian diagram

\[
    \begin{tikzcd}
    \RH_g^n \arrow[r]\arrow[d] \arrow[dr, phantom, "\square"] & \cH_g = \bigg[ \frac{\chi^{ \otimes -2} \otimes W_{g+1} \smallsetminus \Delta}{\Gm \times \PGL_2} \bigg] \arrow[r]\arrow[d] \arrow[dr, phantom, "\square"] & B \Gm \arrow[d]\\
    \mathcal{D}_{2n,2g+2-2n} = \bigg[ \frac{(\chi^{(1)})^{-1} \otimes W_{2n} \times (\chi^{(2)})^{-1} \otimes W_{2g+2-2n} \setminus \Delta}{\mathbb{G}_m^{(1)} \times \mathbb{G}_m^{(2)} \times \mathrm{PGL}_2} \bigg] \arrow[r] & \mathcal{D}_{2g+2} = \bigg[ \frac{\chi^{\otimes -1} \otimes W_{g+1} \setminus \Delta}{\mathbb{G}_m \times \mathrm{PGL}_2} \bigg] \arrow[r] & B\Gm
    \end{tikzcd}
\]
where the rightmost map is induced by the squaring map $\Gm \to \Gm$. 

\begin{lemma} 
    The following diagram
    \[
    \begin{tikzcd}
    B(\Gm \times \Gm) \arrow[r,"{(\lambda,t) \mapsto t}"'] \arrow[d, "{(\lambda,t) \mapsto (\lambda, \lambda^{-1}t^2)}"'] 
    & B\Gm \arrow[d, "{t \mapsto t^2}"'] \\
    B(\Gm \times \Gm) \arrow[r, "{(\lambda,t) \mapsto \lambda t}"'] 
    & B\Gm
    \end{tikzcd}
    \]
    is Cartesian.
\end{lemma}

\begin{proof}
    The corresponding diagram of groups is cartesian.
\end{proof}

\begin{proof}[Proof of Theorem \ref{thm: presentation n>1}]
    Since the map $\cD_{2n,2g+2-2n} \to B \Gm$ factors through $B(\Gm \times \Gm) \to B\Gm$ $(\lambda,t) \mapsto \lambda t$, the map $\RH_g^1 \to \cD_{2n,2g+2-2n}$ is the base change of the map $B(\Gm \times \Gm) \to B(\Gm \times \Gm) $ in the above lemma. Theorem \ref{thm: presentation n>1} follows.
\end{proof}

\section{Preliminary computational lemmas}\label{sec: computational Lemmas}

This section presents several computational lemmas that will be useful later. Specifically, it includes results on the relationship between $\mathrm{GL}_N$-equivariant Chow rings and the associated torus equivariant Chow rings, the Chow rings of $\mu_n$-root gerbes, and computations of the Chow rings of the classifying stacks $BG$ and certain $G$-spaces for the groups $G$ appearing in our presentations. 

For a first reading, the results (though not the definitions or notation) in this section can (and perhaps should) be skipped. Readers may proceed to the subsequent sections and refer back here as needed.  

\subsection{Recovering $\mathrm{GL}_n$-equivariant classes from $\mathrm{T}$-equivariant computations}
    Recall that a linear group $G$ is called special if every $G$-torsor is Zariski-locally trivial. The most common examples are tori and $\mathrm{GL}_n$ for every $n$. One key feature is that the $G$-equivariant Chow ring of a smooth $G$-algebraic space $X$ is a direct summand (as a $\mathrm{CH}_G^*(X)$-module) of the maximal torus equivariant Chow ring of $X$ (see~\cite[Proposition 2.2]{EF09}). The most basic case is the Chow ring of the classifying stack of $\mathrm{GL}_n$, where the maximal torus $T\cong\Gm^n$ consists of diagonal matrices, and inclusion is equivalent to inclusion of the ring of symmetric functions $c_1,\ldots,c_n$ into the polynomial ring in the $n$ variables $t_1,\ldots,t_n$. We write an explicit basis for $\mathrm{CH}^*(B\mathrm{T})$ as a $\mathrm{CH}^*(B\mathrm{GL}_n)$-module.
    \begin{lemma}\label{lem:basisGLn}
        Let $n\geq2$. The ring $\mathrm{CH}^*(B\mathrm{T})=\mathbb{Z}[t_1,\ldots,t_n]$ is a free $\mathrm{CH}^*(B\mathrm{GL}_n)=\mathbb{Z}[c_1,\ldots,c_n]$-module of rank $n!$. A base consists of monomials
        \[
            \prod_{i=1}^{n}t_i^{d_i}\text{ with }0\leq d_i\leq n-i\text{ for every }1\leq i\leq n.
        \]
    \end{lemma}
    \begin{proof}
        The freeness and rank part is well known, see for instance~\cite[Theorem 6.2]{Dem73}; therefore, it is enough to show that the elements in the statement generate $\mathrm{CH}^*(BT)$, as their number matches the rank. Also the basis is already know, but we give a proof here as we have not found a suitable reference; usually, the basis is given in terms of Schubert polynomials and over a field. Consider the sequence of extensions of rings
        \[
            \Z[c_1,\ldots,c_n]\subset\mathbb{Z}[c_1,\ldots,c_n][t_1]\subset\mathbb{Z}[c_1,\ldots,c_n][t_1,t_2]\subset\ldots\subset\mathbb{Z}[c_1,\ldots,c_n][t_1,\ldots,t_n]=\mathbb{Z}[t_1,\ldots,t_n].
        \]
        For every $0\leq r\leq s$ and $s>0$, denote by $c_r^{(s)}$ the elementary symmetric polynomial of degree $r$ in the $s$ variables $t_{n-s+1},\ldots,t_n$, with $c_0^{(s)}=1$ for every $s$. The identity $c_r^{(s)}=c_r^{(s+1)}-t_{n-s}c_{r-1}^{(s)}$ shows that $\Z[c_1^{(n-m)},\ldots, c_{n-m}^{(n-m)}] \subseteq \Z[c_1,\ldots,c_n][t_1,\ldots, t_m]$ for $m=1,\ldots, n$. In particular, the monic polynomial $\prod_{i=m+1}^n(x-t_i)$ with root $t_{m+1}$ has coefficients in $\mathbb{Z}[c_1,\ldots,c_n][t_1,\ldots,t_m]$ and $1,t_m,\ldots, t_m^{n-m}$ generate $\Z[c_1,\ldots,c_n][t_1,\ldots,t_m]$ as a $\Z[c_1,\ldots,c_n][t_1,\ldots,t_{m-1}]$-module for $m=1,\ldots,n$ (where we set $\Z[c_1,\ldots,c_n][t_1,\ldots,t_{m-1}]:=\Z[c_1,\ldots,c_n]$ for $m=1$). This concludes.
    \end{proof}
    \begin{notation}
        For every $\underline{d}=(d_1,\ldots,d_n)$ multi-index of length $n$ satisfying
        \begin{equation}\label{eq:multiindex}
            0\leq d_i\leq n-i\text{ for every }1\leq i\leq n.
        \end{equation}
        we set $\underline{t}^{\underline{d}}:=\prod t_i^{d_i}$ and $|\underline{d}|=\sum d_i$.
    \end{notation}
    \begin{lemma}\label{lemma:consbasisalgebraicversion}
        Let $R$ a $\mathbb{Z}[c_1,\ldots,c_n]$-positively-graded-algebra and $J\subset R$ an ideal. Let $\tilde{J}$ be the extension of $J$ in $\tilde{R}:=R\otimes_{\mathbb{Z}[c_1,\ldots,c_n]}\mathbb{Z}[t_1,\ldots,t_n]$. Fix an integer $k\geq0$, and suppose given $\alpha_{\underline{d}},\beta_{\underline{d}}\in R$ for every $\underline{d}$ satisfying~\eqref{eq:multiindex}, with $\alpha_{\underline{d}}$ of degree $k-|\underline{d}|$.
        \begin{enumerate}
            \item As $R$-modules, we have 
            \[
                \widetilde{J}=\bigoplus_{\underline{d}}\underline{t}^{\underline{d}}J\subset\widetilde{R}.
            \]
            \item In $\widetilde{R}$, we have
            \begin{align*}
                \sum\underline{t}^{\underline{d}}\alpha_{\underline{d}}=\sum\underline{t}^{\underline{d}}\beta_{\underline{d}}\mod\widetilde{J}
            \end{align*}
            if and only if, in $R$,
            \[
                \alpha_{\underline{d}}\equiv\beta_{\underline{d}}\mod J\qquad\text{for every }\underline{d}.
            \]
        \end{enumerate}
        In particular, applying the first point to the ideal $J=R$, we get a decomposition of $\widetilde{R}$ as a free $R$-module.
    \end{lemma}
    \begin{proof}
        By Lemma~\ref{lem:basisGLn}, we know that $\Z[t_1,\ldots, t_n]$  is free as a $\Z[c_1,\ldots, c_n]$-module, and we have an explicit basis. Applying the functor $\otimes_{\Z[c_1,\ldots, c_n]}\Z[t_1,\ldots, t_n]$ to $R$, this induces a decomposition $\widetilde{R}=\oplus\underline{t}^{\underline{d}}R$, and the map $R\rightarrow\widetilde{R}$ identifies the first with the corresponding summand. This also show that $\widetilde{R}$ is flat over $R$. Consequently, we can identify $\widetilde{J}=J\otimes_R\widetilde{R}$, which inherits the stated decomposition. The second part follows from the first. Indeed, the quotient of $\widetilde{R}$ by $\widetilde{J}$ inherits an induced decomposition.
    \end{proof}
    The following is a translation in terms of Chow rings.
    \begin{proposition}\label{prop:generalizeddecomposition}
        Let $X$ be a smooth $k$-scheme with a $\mathrm{GL}_n$-action, and let $\mathbb{G}_m^n\cong T\subset\mathrm{GL}_n$ be the maximal torus consisting of diagonal matrices. Let $J\subset\mathrm{CH}_{\mathrm{GL}_n}^*(X)$ be an ideal, and let $\widetilde{J}$ be the extension of $J$ in $\mathrm{CH}_T^*(X)$. Fix an integer $k\geq0$, and suppose given $\alpha_{\underline{d}},\beta_{\underline{d}}\in\mathrm{CH}^*_{\mathrm{GL}_n}(X)$ for every $\underline{d}$ satisfying~\eqref{eq:multiindex}, with $\alpha_{\underline{d}}$ of degree $k-|\underline{d}|$.
        \begin{enumerate}
            \item We have
            \[
                \widetilde{J}=\bigoplus_{\underline{d}}\underline{t}^{\underline{d}}J.
            \]
            \item We have
            \begin{align*}
                \sum\underline{t}^{\underline{d}}\alpha_{\underline{d}}=\sum\underline{t}^{\underline{d}}\beta_{\underline{d}}\mod\widetilde{J}
            \end{align*}
            if and only if
            \[
                \alpha_{\underline{d}}\equiv\beta_{\underline{d}}\mod J\qquad\text{for every }\underline{d}.
            \]
        \end{enumerate}
        In particular, applying the first point to the ideal $J=\mathrm{CH}_{\mathrm{GL}_n}^*(X)$, we get a decomposition of $\mathrm{CH}_T^*(X)$ as a free $\mathrm{CH}_{\mathrm{GL}_n}^*(X)$-module.
    \end{proposition}
    \begin{proof}
        Apply the Lemma above to $R=\mathrm{CH}_{\mathrm{GL}_n}^*(X)$, and $\widetilde{R}=\mathrm{CH}^*_{\mathrm{T}}(X)$; this is justified by the isomorphism of $\mathrm{CH}^*_{\mathrm{GL}_n}(X)$-modules
        \begin{equation}\label{eq:isoBri}
        \begin{tikzcd}
            \mathrm{CH}_{\mathrm{GL}_n}^*(X)\otimes_{\mathrm{CH}^*(\mathrm{BGL}_n)}\mathrm{CH}^*(\mathrm{BT})\arrow[r,"\cong"] & \mathrm{CH}^*_{\mathrm{T}}(X),
        \end{tikzcd}
        \end{equation}
        see~\cite[Theorem 6.7]{Bri97}.
    \end{proof}

\subsection{The Chow ring of $\mu_n$-root gerbes}\label{sec: Chow root gerbes}

Recall that, given an algebraic stack $\mathcal{Y}$ and a line bundle $\mathcal{L}$ on $\mathcal{Y}$, the $\mu_n$-root gerbe $ \pi: \mathcal{X} = \mathcal{Y}(\sqrt[n]{\mathcal{L}}) \to \cY$ is defined as the base change of the map $\mathcal{Y} \to B\Gm$ associated with $\mathcal{L}$, via the morphism $B\Gm \to B\Gm$ given by $t \mapsto t^n$.

In this section we compute the Chow ring of $\mathcal{X}$ from that of $\mathcal{Y}$. We assume that $\cY$ is a smooth algebraic stack stratified by locally closed substacks which are each isomorphic to the quotient stack of an algebraic group acting on an algebraic space, so that the Chow rings of $\cY$ and $\cX$ is well-defined \cite[Theorem 1.1.12]{kresch1999cycle}

\begin{proposition}\label{prop: Chow of root gerbes}
    Let $\cY$, $\cL$ and $\pi: \cX= \cY(\sqrt[n]{\mathcal{L}}) \to \cY$ be as above. Then, we have an isomorphism of rings
    $$
    \CH^*(\cX) \cong \frac{\CH^*(\cY)[t]}{(c_1(\cL)-nt)}.
    $$7
\end{proposition}

The authors have been informed of an independent direct proof of Proposition \ref{prop: Chow of root gerbes} by J. Alper, M. Bishop, F. Janda, and M. Shin.

One possible approach is to observe that $\mu_n$-root gerbes are a special case of weighted projective bundles, specifically $\cX = \P(\cL)$, where $\mathbb{G}_m$ acts on $\cL$ with weight $-n$, and to apply the weighted projective bundle formula \cite[Theorem 3.12]{AreObAbr}.

Here, we present another self-contained proof of the statement. First, we need a lemma.

\begin{lemma} \label{lemma: Kunneth for Gm} \cite[Lemma 10]{Oe18} 
    Let $\cY$ be as above and endow it with the trivial $\Gm$-action. Then, the cup product map
    $$
    \CH^*(B\Gm) \otimes_{\Z} \CH^*(\cY) \to \CH^*(\cY \times B\Gm)
    $$
    is an isomorphism.
\end{lemma}

\begin{proof}[Proof of Proposition \ref{prop: Chow of root gerbes}]
    Let $\cV = \cL \smallsetminus 0$ denote the complement of the zero section of $\cL$. Then, by definition, we have $\cX \cong [\cV / \Gm]$, where the action is given by $t \cdot v = t^{-n} v$ for $t \in \Gm$ and $v \in \cV$. Similarly, $\cY \cong [\cV / \Gm]$, where the action is given by $t \cdot v = t v$ for $t \in \Gm$ and $v \in \cV$.  

    We claim that the morphism 
    \[
    \begin{tikzcd}
        \rho: \cX \arrow[r] & \cY \times B \Gm
    \end{tikzcd}
    \]
    obtained from $\pi$ and the $n$-th root $\cL^{1/n}$ of $\cL$ is a $\Gm$-torsor associated to $c_1(\cL)-nt$. The statement then follows from Lemma \ref{lemma: Kunneth for Gm}. The fact that it is a $\Gm$-torsor follows from the following Cartesian diagram
    \[
        \begin{tikzcd}
        \cV \arrow[r]\arrow[d] \arrow[dr, phantom, "\square"] & \cY \arrow[r]\arrow[d] \arrow[dr, phantom, "\square"] & \mathrm{Spec}(k) \arrow[d]\\
        \cX \arrow[r] & \cY \times B\Gm \arrow[r] & B\Gm
        \end{tikzcd}
    \]
    where all the maps $\cV \to \cX=[\cV / \Gm] $, $\cV \to \cY= [\cV / \Gm]$ and $\mathrm{Spec}(k) \to B\Gm$ are the tautological $\Gm$-torsors. Finally, let $\mathcal{M}$ denote the line bundle associated to $\rho$. By Lemma \ref{lemma: Kunneth for Gm}, we have 
    \[
    c_1(\mathcal{M}) = c_1(\mathcal{M} _{| \mathcal{Y} \times \mathrm{Spec}(k)}) + c_1(\mathcal{M} _{| \{ y \} \times B\mathbb{G}_m}),
    \]
    where $y$ is any $k$-point of $\mathcal{Y}$.

    By the above Cartesian diagram, $c_1(\mathcal{M} _{| \mathcal{Y} \times \mathrm{Spec}(k)}) = c_1(\mathcal{L})$, and the restriction of $\rho$ to $\{y\} \times B\mathbb{G}_m$ is simply $B\mu_n \to B\mathbb{G}_m$. This concludes the proof.
\end{proof}

\subsection{ Chow rings of the classifying stacks $B\mathrm{G}$ and of some $\mathrm{G}$-spaces }

The Chow ring of $\PGL_2$ is already known \cite{pandharipande1996chow,Vez98, DL18}:
\begin{equation}\label{eqn: Chow PGL2}
    \CH^*(B \PGL_2)=\frac{\Z[c_2,c_3]}{(2 c_3)}
\end{equation}
where the classes $c_i$ ($i=2,3$) are the Chern classes of the representation $W_1$. Equivalently, since $W_1$ is isomorphic to the adjoint representation $\mathsf{sl}_2$ of $\PGL_2$, these are also the Chern classes of $\mathsf{sl}_2$ (as in \cite{Vez98}). Finally, there is an isomorphism of stacks $B \PGL_2 \cong [ \mathcal{S}/ \GL_3]$ where $\mathcal{S}$ is the space of smooth degree $2$ homogeneous polynomials in two variables (see \cite[Proposition 1.5]{DL18}). Under this isomorphism the classes $c_i$ are the pullback from $B \GL_3$ of the Chern classes of the standard $\GL_3$-representation (see \cite[ Lemma 1.3]{DL18}).

In the sequel, we will use any of these three interpretations without further reference.

We will also need to know the $\PGL_2$-equivariant Chow ring of products of $\P^1$.

\begin{notation}
    We will always interpret $\P^1$ as the space of linear forms in two variables (up to scalar) and a matrix $[A] \in \PGL_2$ acts on it by precomposition with its inverse.
\end{notation}

The $\PGL_2$-equivariant Chow ring of $\P^1$ is computed in \cite[Lemma 5.1]{GV08}:

\begin{equation}\label{eqn: PGL2 Chow of P1}
    \CH^*_{\PGL_2}(\P^1)= \frac{\Z[c_2,c_3,\tau]}{(c_3, \tau^2+c_2)}= \Z[\tau]
\end{equation}

where $\tau=c_1^{\PGL_2}(\cO_{\P^1}(2))$. The $\PGL_2$-linearization on $\cO_{\P^1}(2)$ is obtained by requiring that the inclusion of bundles $\cO(-2) \subseteq  W_1 \otimes \cO_{\P^1}$ on $\P^1$ is $\PGL_2$-equivariant.

For products of projective lines we have the following:

\begin{lemma}\label{lem: PGL2 Chow prod proj lines}
    The $\PGL_2$-equivarinat Chow ring of $(\P^1)^m$ is
    $$
        \mathrm{CH}_{\PGL_2}((\P^1)^m)=\frac{\mathbb{Z}[c_2,c_3,\tau_1,[\Delta_{1,2}],\ldots,[\Delta_{1,m}]]}{(c_3,\tau_1^2+c_2,[\Delta_{1,2}]^2-[\Delta_{1,2}]\tau_1,\ldots,[\Delta_{1,m}]^2-[\Delta_{1,m}]\tau_1)}.
    $$
    where we denoted by $[\Delta_{i,j}]$ the pullback of the class of the diagonal in $\P^1 \times \P^1$ under the $(i,j)$-projection ($i\neq j$), and by $\tau_i$ the pullback of $\tau$ under the $i$-th projection.

    Moreover, if $m>2$ and $i,j,l$ are all distinct, we have 
    \begin{equation}\label{eq: relation diagonals}
        [\Delta_{i,j}]=[\Delta_{l,i}]+[\Delta_{l,j}]-\tau_l.
    \end{equation}
\end{lemma}

\begin{proof}
    We can assume $m\geq2$ and work by induction on $m$. Let $m=2$. As $\Delta_{1,2}\cong\P^1$ via the diagonal embedding $\delta: \P^1 \to (\P^1)^2$, and $[(\P^1\times\P^1\setminus\Delta_{1,2})/\PGL_2]\cong B\Gm$, the localization sequence gives
    \begin{equation}\label{eq: P1xP1 localization sequence}
    \begin{tikzcd}
        0\arrow[r] & \CH_{\PGL_2}^*(\P^1)\arrow[r,"\delta_*"] & \CH_{\PGL_2}^*(\P^1\times\P^1)\arrow[r,"\phi"] & \CH^*(B\Gm)\arrow[r] & 0.
    \end{tikzcd}
    \end{equation}
    The first map is injective because the composite $\P^1\xrightarrow{\delta}\P^1\times\P^1\xrightarrow{\text{pr}_1}\P^1$ is an isomorphism.
    Notice that $2[\Delta_{1,2}]=\tau_1+\tau_2$, as the associated line bundles are isomorphic $\PGL_2$-equivariantly. Since $\tau_1^2=-c_2$, we have $\phi(\tau_1^2)=-\phi(c_2)=\lambda^2$ for any generator $\lambda$ of $\Pic(B\Gm)$, by~\cite[page 5]{Vez98}. Therefore, $\phi(\tau_1)$ is a generator of $\CH^*(B\Gm)$. Since every homogeneous element of $\CH_{\PGL_2}^*(\P^1)$ is a power of $\delta^*(\tau_1)$ multiplied by an integer, the projection formula shows that the image of $\delta_*$ is generated by $[\Delta_{1,2}]$ as a $\Z[c_2,\tau_1]$-module. Then, equation~\eqref{eq: P1xP1 localization sequence} shows that $\CH^*_{\PGL_2}(\P^1\times\P^1)$ is generated as a ring by the classes in the statement. Now, we pass to the relations, the first two being already known. Restricting the equation $2[\Delta_{1,2}]=\tau_1+\tau_2$ to $\Delta_{1,2}$ and using the torsion freeness of $\CH_{\PGL_2}^*(\P^1)$ in degree 1, we get the relation $[\Delta_{1,2}]_{|\Delta_{1,2}}=\tau$, and the projection formula gives the third relation in the statement. We are left with showing these are the only relations.

    Notice that $\CH_{\PGL_2}^*(\P^1\times\P^1)$ is generated as a $\Z[c_2]$-module by 1, $\tau_1$, $[\Delta_{1,2}]$ and $[\Delta_{1,2}]\tau_1$. Therefore, any new homogeneous relation is of the form $ac_2^k+bc_2^{k-1}\tau_1+dc_2^{k-1}[\Delta_{1,2}]+ec_2^{k-2}[\Delta_{1,2}]\tau_1$, with $a,b,d,e\in\Z$. Pushing forward along the first projection onto $\P^1$, we get $dc_2^{k-1}+ec_2^{k-2}\tau=0$, implying $d=e=0$. Then, restricting to $[\Delta_{1,2}]$ we get that $a=b=0$; this concludes the case $m=2$.

    Now assume $m\geq3$ and that the statement is true for product of less than $m$ projective lines. Let $\widetilde{\Delta}\subset(\P^1)^m$ be the union of all the diagonals $\Delta_{i,j}$. Then, the localization sequence together with the surjection $\oplus_{i<j}\CH_{\PGL_2}^*(\Delta_{i,j})\rightarrow\CH_{\PGL_2}^*(\widetilde{\Delta})$ yield the exact sequence
    \[
    \begin{tikzcd}
        \oplus_{i<j}\CH_{\PGL_2}^*(\Delta_{i,j})\arrow[r,"\sum\delta_{i,j}"] & \CH_{\PGL_2}^*((\P^1)^m)\arrow[r,"\phi"] & \CH^*_{\PGL_2}((\P^1)^m \smallsetminus \widetilde{\Delta})\arrow[r] & 0.
    \end{tikzcd}
    \]
    Notice that the last Chow ring is isomorphic to $\mathbb{Z}$, concentrated in degree 0, as the quotient $((\P^1)^m\setminus\widetilde{\Delta})/\PGL_2$ is an open subscheme of an affine space, as $m\geq3$. Using the projection formula as above, it follows that the Chow ring is generated as a $\Z[c_2,\tau_1]$-algebra by the classes $[\Delta_{i,j}]$.
    
    Now, we prove that $[\Delta_{i,j}]=[\Delta_{l,i}]+[\Delta_{l,j}]-\tau_l$; for this, it is enough to show that $[\Delta_{2,3}]=[\Delta_{1,2}]+[\Delta_{1,3}]-\tau_1$ in $\CH_{\PGL_2}^*((\P^1)^3)$. Since $\phi(\tau_1)=0$, we can write $\tau_1=a[\Delta_{1,2}]+b[\Delta_{1,3}]+c[\Delta_{2,3}]$ with $a,b,c\in\Z$. Restricting to $\Delta_{1,2}$ and $\Delta_{1,3}$, we get $a=b=1$, $c=-1$, as wanted.
    
    Using the just proved equation~\eqref{eq: relation diagonals} with $l=1$, we get that the $\PGL_2$-equivariant Chow ring of $(\P^1)^m$ is generated as a ring by the elements in the statement. We are left with showing that the ideal of relations is generated by the classes in the statement. We do as for the case $m=2$. We know that $\CH_{\PGL_2}^*((\P^1)^m)$ is generated as a $\Z[c_2]$-module by classes of the form $\tau_1^{\epsilon}\prod_{i\in I}[\Delta_{1,i}]$, with $I\subset\{2,\ldots,m\}$ and $\epsilon=0,1$; hence, any new homogenous relation is of the form
    \[
        \theta:=ac_2^k+bc_2^{k-1}\tau_1+\sum_Id_Ic_2^{k-|I|}\prod_{i\in I}[\Delta_{1,i}]+\sum_Ie_Ic_2^{k-1-|I|}\tau_1\prod_{i\in I}[\Delta_{1,i}]
    \]
    with $I\subset\{2,\ldots,m\}$ and the coefficients in $\Z$ (if an exponent of $c_2$ is negative we set the relative coefficient to 0 ). For every $J\subset\{2,\ldots,m\}$, let $p_J:(\P^1)^m\rightarrow(\P^1)^{m-|J|}$ be the projection that forgets the $j$-th variable for every $j\in J$. Then, $p_{J*}(\prod_{i\in I}[\Delta_{1,i}])\not=0$ if and only if $J\subset I$, in which case it is the product $\prod_{i\in I\setminus J}[\Delta_{1,u_{J}(i)}]$; here $u_{J}:\{2,\ldots,m\}\setminus J\xrightarrow{\sim}\{2,\ldots,m-|J|\}$ is the natural relabelling due to forgetting some components (if $I=J=\{2,\ldots,m\}$ then the product is empty and we set it to be 1). Moreover, $p_{J*}(\tau_1)=p_{J*}p_J^*(\tau_1)=0$, as $1\not\in J$. Therefore,
    \[
        p_{J*}(\theta)=\sum_{J\subset I}d_Ic_2^{k-|I|}\prod_{i\in I\setminus J}[\Delta_{1,u_J(i)}]+\sum_{J\subset I}e_Ic_2^{k-1-|I|}\tau_1\prod_{i\in I\setminus J}[\Delta_{1,u_J(i)}]\in\CH_{\PGL_2}^*((\P^1)^{m-|J|}).
    \]
    By induction, this is 0 if and only if $d_I=e_I=0$ for every $I$ containing $J$. Varying $J$, we obtain $\theta=ac_2^k+bc_2^{k-1}\tau_1$, and restricting to $\Delta_{1,2}$ we get $a=b=0$, again by induction. This concludes.
\end{proof}

The Chow ring of the classifying stack $B \GI$ has already been computed in \cite{Lar19}.

\begin{theorem}\cite[Theorem 5.2.]{Lar19}\label{thm: chow BG}
    The Chow ring of $B(\GI)$ is given by
    $$
    \CH^*(B(\GI)) \cong \frac{\Z[\beta_1, \beta_2, \gamma]}{(2\gamma, \gamma(\gamma +\beta_1))}.
    $$
\end{theorem}

For $B(\GII)$ we have the following.

\begin{theorem}\label{thm: Chow BGII}
        The Chow ring of $B(\GII)$ is 
        $$
            \mathrm{CH}^*(B(\GII))=\frac{\mathbb{Z}[c_2,c_3,\gamma,t]}{(2\gamma,\gamma^3+c_2\gamma+c_3,2c_3)}=\frac{\mathbb{Z}[ c_2,\gamma, t]}{(2\gamma)}.
        $$
    \end{theorem}
    \begin{proof}
        By Lemma~\ref{lemma: Kunneth for Gm},
        $$
        \mathrm{CH}^*(B(\GII))= \mathrm{CH}^*(B(\Gm \rtimes \mu_2))[t],
        $$
        so it is enough to compute $\mathrm{CH}^*(B(\Gm \rtimes \mu_2))$. Call $\mathsf{H}=\Gm \rtimes \mu_2$. Consider $\P^2= \P(W_1)$ with $ \PGL_2$ acting on it. Let $U \subseteq \P^2$ the locus of polynomials with no repeated roots and let $j:Z \hookrightarrow \P^2$ be its complement (taken with the reduced structure). Then $U$ is preserved by the $\PGL_2$ action and $\PGL_2$ acts transitively on it with stabilizers exactly $\mathsf{H}$. The excision sequence then yields 
        \[
        \begin{tikzcd}
            \CH^*_{\PGL_2}(Z) \arrow[r,"j_*"] & \CH^*_{\PGL_2}(\P^2) \arrow[r] & \CH^*(B \mathsf{H}) \arrow[r] & 0.
        \end{tikzcd}
        \]
        The squaring map $ \P^1 \xrightarrow{\sim} Z$ is a $\PGL_2$-equivariant isomorphism and the induced pullback 
        \[
        \begin{tikzcd}
            \CH^*_{\PGL_2}(\P^2)= \frac{\Z[c_2, c_3,\xi]}{ (2 c_3, \xi^3+c_2 \xi +c_3)}=\Z[c_2,\xi] \arrow[r] & \CH^*_{\PGL_2}(\P^1)= \Z[\tau]
        \end{tikzcd}
        \]
        is surjective as it maps $\xi \mapsto \tau$. Here we denoted $\xi= c_1^{\PGL_2}(\cO_{\P^2}(1))$. It follows that
        the image of $j_*$ is generated by $[Z]= 2 \xi \in \CH^*_{\PGL_2}(\P^2)$ and thus
        $$
        \CH^*(B \mathsf{H})=\frac{\Z[c_2,\xi]}{(2 \xi)}.
        $$
        Finally, $\xi$ agrees with the pullback of $\gamma$ from $B \mu_2$ as it is the unique $2$-torsion class of degree $1$ (note that the pullback of $\gamma$ to $B \mathsf{H}$ is non-zero as $\mathsf{H} \to \mu_2$ is surjective). This concludes the proof.
    \end{proof}

    We will also need the $\GII$-equivariant Chow ring of $\P^1$. The action of $\Gm$ is trivial and the action of $\Gm \rtimes \mu_2$ is given by the inclusion $\Gm \rtimes \mu_2 \subseteq \PGL_2$.

    \begin{lemma}
        The $\GII$-equivariant Chow ring of $\P^1$ is 
        $$
        \CH^*_{\GII}(\P^1)= \frac{\Z[c_2,c_3,\tau,\gamma,t]}{(2\gamma, \tau^2+c_2,\gamma(\gamma+\tau),c_3)}= \frac{\Z[\tau,\gamma,t]}{(2 \gamma, \gamma(\gamma+\tau))}.
        $$
        Moreover, we have
        \begin{equation}\label{eqn: class X,Y}
        [\{ X, Y\}]=\tau+ \gamma \in \CH^*_{\GII}(\P^1).
        \end{equation}
    \end{lemma}
    \begin{proof}
        Let $\H=\Gm \rtimes \mu_2$. It is enough to compute the $\H$-equivariant Chow group of $\P^1$. The relations $c_3=0$ and $\tau^2+c_2=0$ already hold in the $\PGL_2$-equivariant Chow ring of $\P^1$ and persist here. The relations $2 \gamma=0$ and $\gamma^3+ \gamma c_2 + c_3 =0$ instead come from the Chow of the classifying stack $B\H$. Note that the relation $\gamma^3+ \gamma c_2 + c_3 =0$ is implied by the others in the statement.

        Next, we show that $\CH^*_{\H}(\P^1)$ is generated by $\tau$ and $\gamma$ as a $\Z$-algebra. Note that the action of $\H$ on $\P^1 \smallsetminus \{ X, Y \}$ is transitive with stabilizers $\mu_2$. This yields a commutative diagram
        \[
        \begin{tikzcd}
            \mathrm{CH}_{\H}^*(\{X,Y\})\arrow[r]\arrow[d,"\cong"] & \mathrm{CH}_{\H}^*(\P^1)\arrow[r] & \mathrm{CH}_{\H}^*(\P^1\setminus\{X,Y\})\arrow[r]\arrow[d,"\cong"] & 0\\
            \mathrm{CH}^*(\mathrm{B}\mathbb{G}_m) & \mathrm{CH}^*(\mathrm{B\H})\arrow[l]\arrow[r]\arrow[u] & \mathrm{CH}^*(\mathrm{B}\mu_2).
        \end{tikzcd}
        \]
        Since $c_2 \in \CH^*(B\H)$ maps to $c_2=-\tau^2 \in \CH^*_{\H}(\P^1)$ and, by \cite[Page 9]{Vez98}, to the square of a generator in $ \CH^*(B\Gm)$, we deduce that the restriction of $\tau \in \CH^*_{\H}(\P^1)$ to $\CH^*_{\H}(\{ X, Y \})= \CH^*(B \Gm)$ is a generator. The push-pull formula implies that the map $ \mathrm{CH}_{\H}^*(\{X,Y\}) \to \mathrm{CH}_{\H}^*(\P^1)$ is given by $\tau_{|\{ X, Y \}} \mapsto \tau \cdot [\{ X, Y \}]$. Also,  $\{ X , Y\}$ is the zero locus of a section of $\cO_{\P^1}(2) \otimes \Gamma$, and thus we have $[\{ X, Y \}]= \tau+\gamma \in \CH^*_\H(\P^1)$. This shows that $\CH^*_{\H}(\P^1)$ is generated as an $\mathbb{Z}$-algebra by $\tau$ and $\gamma$.

        The relation $\gamma(\gamma + \tau)$ arises from the fact that, since the restriction of $\gamma$ to $B\Gm$ is torsion, it must be zero. Consequently, it is also zero in $\mathrm{CH}^*_{\mathsf{H}}({X, Y})$, and thus its image $\gamma(\gamma + \tau)$ in $\mathrm{CH}^*_{\mathsf{H}}(\mathbb{P}^1)$ is also zero. 

        Finally, we show that there are no further relations. Suppose $p$ is a non-zero element of $\Z[\tau,\gamma]/(2 \gamma, \gamma (\gamma+\tau))$. Write it as a sum with coefficients in $\Z$ of monomials $\tau^{\ell}$ and $\gamma^k$. Also the non-zero coeffients of the monomials  $\gamma^k$ can be taken to be $1$. Suppose that $p$ is $0$ in $\CH^*_{\H}(\P^1)$. Then restricting to $\mathrm{CH}_{\H}^*(\{X,Y\})$ we see that the coefficient of each $\tau^{\ell}$ must vanish. Similarly, restricting to $\CH^*_{\H}(\P^1 \smallsetminus \{ X,Y \})$, one see that the remaining coefficients are also $0$. This concludes the proof.
        \end{proof}

        We will also need the class of $\{ X Y \} \subseteq \P(W_1)$. This is analogous to the computation of $[ \{ X Y \} ] \in \CH^*_{\GI}(\P(\Sym^2 V^\vee))$ in \cite[Lemma 33]{CIL24}.

        \begin{lemma}\label{lemma: class of XY for odd g}
            The class of $\{ X Y \} \subseteq \P(W_1)$ is 
            $$
             \xi^2+\xi \gamma +c_2+\gamma^2 \in \CH^*_{\GII}(\P(W_1)).
            $$
            where $\xi=c_1^{\GII}(\cO_{\P(W_1)}(1))$.
        \end{lemma}
        \begin{proof}
            Note that $XY$ is a section of the bundle $W_1 \otimes \Gamma$ over $B\GII$. Define $Q$ as the quotient 
    \begin{equation}\label{eqn: quotient bundle}
    \begin{tikzcd}
        0 \arrow[r] & \mathcal{O}\arrow[r,"XY"] & W_1 \otimes \Gamma  \arrow[r] & Q \arrow[r] & 0.
    \end{tikzcd}
    \end{equation}
    Let $\pi: [\P(W_1)/ \GII] \to B (\GII)$ be the natural projection. There is a natural diagram on $\P (W_1)$
    $$
    \begin{tikzcd}
	&& 0 \\
	&& {\mathcal O} \\
	0 & {\mathcal O_{\P( W_1)} (-1)\otimes \Gamma} & {\pi^* W_1 \otimes \Gamma} \\
	&& {\pi^*Q} \\
	&& 0
	\arrow[from=3-1, to=3-2]
	\arrow[from=3-2, to=3-3]
	\arrow[from=1-3, to=2-3]
	\arrow["XY", from=2-3, to=3-3]
	\arrow[from=3-3, to=4-3]
	\arrow[from=4-3, to=5-3]
	\arrow["\psi"', from=3-2, to=4-3]
\end{tikzcd}
    $$
    and $\{XY\} \subset \P(W_1)$ is precisely the locus where
    $$
    \psi \in  H^0(\P(W_1), \pi^* Q \otimes \mathcal{O}_{\P(W_1)}(1) \otimes \Gamma^\vee)
    $$
     vanishes. Since the codimension of $\{XY\} \subseteq \P(W_1)$ is $2$, the expected one, we have 
    \begin{align*}
    [\{XY\}] = &c_2^{\GII} (\pi^* Q \otimes \mathcal{O}_{\P(W_1)}(1) \otimes \Gamma^\vee) \\
    = & c_2^{\GII}( \pi^* Q) +c_1^{\GII}(\pi^*Q)(\xi+\gamma)+(\xi+\gamma)^2 \\
    = & \xi^2+c_2+\gamma^2+\xi \gamma.
    \end{align*}
    This concludes the proof.
    \end{proof}

    \subsection{$\GL_3$-counterparts}\label{sec: GL3-counterparts}

    In \cite{DL18}, Di Lorenzo introduced the notion of $\GL_3$-counterparts.
    
    \begin{definition}\cite[Definition 1.3 and Definition 2.2]{DL18}
        Let $X$ be a scheme of finite type over $\mathrm{Spec}(k)$ endowed with a $\PGL_2$-action. Then a $\GL_3$-counterpart of $X$ is a scheme $Y$ endowed with a $\GL_3$-action and an isomorphism $ [Y/\GL_3] \cong [X/\PGL_2]$. Given two schemes $X$ and $X'$ with a $\PGL_2$-action, and a $\PGL_2$-equivariant morphism $f:X\rightarrow X'$, a $\GL_3$-counterpart of $f$ is the datum of $\GL_3$-counterparts $Y$ and $Y'$ of $X$ and $X'$ respectively, together with a $\GL_3$-equivariant morphism $g:Y\rightarrow Y'$ such that the diagram
        \[
        \begin{tikzcd}
            {[X/\PGL_2]}\arrow[r]\arrow[d,"\cong"'] & {[X'/\PGL_2]}\arrow[d,"\cong"]\\
            {[Y/\GL_3]}\arrow[r] & {[Y'/\GL_3]}
        \end{tikzcd}
        \]
        commutes.
    \end{definition}
    
    In \cite[Section 1]{DL18}, the author explicitly computed a $\GL_3$-counterpart of the quotient stack $[\P(W_m) / \PGL_2]$. Here, we recall his construction as well as the definition of certain natural substacks.  

    Roughly speaking, $[\P(W_m) / \PGL_2]$ parametrizes degree $2m$ divisors on a smooth rational curve. Let $\mathcal{S} \subseteq k[X_1,X_2,X_3]_2 \smallsetminus \{ 0\}$ be the subset of smooth degree $2$ polynomials. A matrix $A \in \GL_3$ acts on a point of $\mathcal{S}$ by precomposition with $A^{-1}$ and we have isomorphisms $[\mathcal{S}/\GL_3]\cong B \PGL_2 \cong \mathfrak{M}_0$, where $\mathfrak{M}_0$ denotes the moduli stack of smooth genus $0$ curves.

    For $m \in \Z_{\geq 1}$, let $V_m$ be the cokernel of the bundle map
    \begin{equation*}
    \begin{tikzcd}[row sep=tiny]
        \mathcal{S} \times k[X_1,X_2,X_3]_{m-2} \arrow[r] & \mathcal{S} \times k[X_1,X_2,X_3]_{m}\\
        (q,f) \arrow[r,mapsto] & (q,qf)
    \end{tikzcd}    
    \end{equation*}

    The action of $\GL_3$ on $\mathcal{S}$ extends to an action on the projective bundle $\P(V_m)$ by setting $A \cdot (q,[f])=(q \circ A^{-1}, [f \circ A^{-1}])$. Moreover, we have an isomorphism $[\P(V_m) /\GL_3] \cong [\P(W_m)/ \PGL_2]$ obtained by identifying the $(q,[f])$ of $\P(V_m)$ with the divisor given by the zero locus of $f$ in the conic defined by $q$ (see \cite[Proposition 3.2]{DL18}).

    \begin{notation}
        We denote by $\xi_{2m} \in \CH^*_{\GL_3}(\P(V_m))$ the $\GL_3$-equivariant first Chern class of $\cO_{\P(V_m)}(1)$. By an abuse of notation we will also denote with the same symbol its pullback to $\CH^*_{\Gm^3}(\P(V_m))$, $\Gm^3 \subseteq \GL_3$ is the maximal torus consisting of diagonal matrices.
    \end{notation}
    
    For $r,\ell \in \Z_{\geq 0}$ such that $r+\ell \leq m$, we have inclusions
    \[
    \begin{tikzcd}
        j_{m;r,\ell}: \P(V_{m-r-\ell}) \arrow[r,hookrightarrow] & \P(V_m)
    \end{tikzcd}
    \]
    given by $(q,[f]) \mapsto (q,[f \cdot X_0^r X_1^\ell])$. These are not $\GL_3$-equivariant, but only $\Gm^3$-equivariant. Thus we have $\Gm^3$-equivariant classes $[W_{m;r,\ell}] \in \CH^*_{\Gm^3}(\P(V_m))$ defined by the images of the maps $j_{m;r,\ell}$.

    We will need to compute some of these classes. A quadric form $q \in \mathcal{S}$ can be written as $\sum_{\underline{i}} a_{\underline{i}} X_1^{i_1} X_2^{i_2} X_3^{i_3}$ where $\underline{i}=(i_1,i_2,i_3) \in \Z_{\geq 0}^3$ is such that $|\underline{i}|=i_1+i_2+i_3=2$. Let $\mathcal{S}_{\underline{i}} \subseteq \mathcal{S}$ be the open locus where $a_{\underline{i}} \neq 0$. The restriction projective bundles $\P(V_m)$ to each $\mathcal{S}_{\underline{i}}$ is trivial and 
    \begin{equation}\label{eqn: Chow P(Vm) restricted to Si}
    \CH^*_{\Gm^3}(\P(V_m)_{|\mathcal{S}_{\underline{i}}})=\frac{\CH^*_{\Gm^3}(\P(V_m))}{(i_1 t_1+ i_2 t_2 + i_3 t_3).}
    \end{equation}
    See \cite[Lemmas 4.1 and 4.5]{DL18}. Moreover, by \cite[Lemma 4.6]{DL18}, the restriction of $[W_{m;r,0}]$ to $\P(V_m)_{|\mathcal{S}_{\underline{i}}}$ with $\underline{i}=(0,i_2,i_3)$ is given by

    \begin{equation}\label{eqn: restriction class W}
        [W_{m;r,0}]_{|\P(V_m)_{|\mathcal{S}_{\underline{i}}}}= \prod (\xi_{2m}- \underline{k} \cdot \underline{t}) 
    \end{equation}
where the product is over $\underline{k}$ such that $|\underline{k}|=m$, $k_j< i_j$ for some $j$ and $k_1<r$.

In our paper, we will need to explicitly compute $[W_{m;1,0}]$ and $[W_{m,2,0}]$. This is done by computing their restrictions over $\mathcal{S}_{0,2,0},\mathcal{S}_{0,1,1}$ and $\mathcal{S}_{0,0,2}$ using Equation \eqref{eqn: restriction class W} and then applying the following lemma.

\begin{lemma}\label{lemma: Pvmunicitymod}
    Let $k\leq2m$ and $\alpha\in\mathrm{CH}_{\Gm^3}^k(\P(V_m))$ such that is $0$ modulo $2t_2$, modulo $t_2+t_3$ and modulo $2t_3$. Then, $\alpha=0$.
    \end{lemma}
    \begin{proof}
        Since the relation coming from the pojective bundle formula is a monic polynomial of degree $2m+1$ in $\xi_{2m}$, the relations involving classes of degree less than or equal to $2m$ are generated by $t_1+t_2+t_3=0$ and $2t_1t_2t_3=0$. Replacing $t_1$ with $-t_2-t_3$, we can represent  $\alpha=\alpha(\xi_m,t_2,t_3)$ as a polynomial in $\xi_m$, $t_2$ and $t_3$ in $\Z[\xi_{2m},t_2,t_3]/(2 t_2 t_3 (t_2+t_3))$. By assumption, we can write $\alpha=2t_2\alpha_1= 2t_3\alpha_2$ in such a ring. Therefore, $2t_2t_3(t_2+t_3)$ divides $2t_2\alpha_1-2 t_3\alpha_2$ in $\mathbb{Z}[\xi_{2m},t_2,t_3]$. It follows that $t_3$ divides $\alpha_1$ and, in $\Z[\xi_{2m},t_2,t_3]/(2 t_2 t_3 (t_2+t_3))$, we have $\alpha=2t_2t_3\alpha_{1,2}=(t_2+t_3)\beta$ for some $\beta$. Using the same reasoning above, we get that $t_2+t_3$ divides $\alpha_{1,2}$ in $\Z[\xi_{2m},c_2,c_3]$. Therefore, $\alpha=2t_2t_3(t_2+t_3)\gamma=0$ in $\CH^k_{\Gm^3}(\P(V_m))$. 
    \end{proof}

    \begin{proposition}\label{prop: computation W[m;1,0,0]}
        Let $m\geq2$. The $\Gm^3$-equivariant class of $[W_{m;1,0}]$ is
        \begin{align*}
            [W_{m;1,0}]=(\xi_{2m}-mt_3)(\xi_{2m}-(2m-1)t_2-(m-1)t_3)+m(m-1)t_2(t_2+t_3).
        \end{align*}
    \end{proposition}
    \begin{proof}From Equation \eqref{eqn: restriction class W}, we have
        \begin{align*}
            [W_{m;1,0}]|_{\mathcal{S}_{0,2,0}}=&(\xi_{2m}-m t_3)(\xi_{2m}-t_2-(m-1)t_3)\\
            [W_{m;1,0}]|_{\mathcal{S}_{0,1,1}}=&(\xi_{2m}-mt_2)(\xi_{2m}-mt_3)\\
            [W_{m;1,0}]|_{\mathcal{S}_{0,0,2}}=&(\xi_{2m}-mt_2)(\xi_{2m}-(m-1)t_2-t_3).
        \end{align*}
        A computation shows that the stated expression for $[W_{m;1,0,0}]$ restricts to the above on the strata $\mathcal{S}_{\underline{i}}$. The conclusion follows from Lemma \ref{lemma: Pvmunicitymod}.
    \end{proof}

    \begin{remark}\label{rmk: computation W[m;1,0,0]}
    We can rewrite
        \[
        [W_{m;1,0}]=\xi_{2m}^2+(2m-1)t_1\xi_{2m}+(m^2c_2+m(2m-1)t_1^2).
        \]
        Indeed,
        \[
        [W_{m;1,0}]=\xi_{2m}^2-(2m-1)(t_2+t_3)\xi_{2m}+m(3m-2)t_2t_3+m(m-1)(t_2^2+t_3^2).
        \]
        By \cite[Lemma 4.4]{DL18}, we have $t_1+t_2+t_3=c_1=0$, and thus the coefficient relative to $\xi_{2m}$ is $(2m-1)t_1\xi_{2m}$. 
        Next, notice that $t_2t_3=(-t_1-t_3)(-t_1-t_2)=t_1^2+c_2$ and that $t_1^2+t_2^2+t_3^2=c_1^2-2c_2=-2c_2$. Therefore, the constant term can be rewritten as
        \begin{align*}
            &m(3m-2)(t_1^2+c_2)+m(m-1)(-2c_2-t_1^2)\\
            =&m(3m-2)c_2+m(2m-1)t_1^2-2m(m-1)c_2\\
            =&m^2c_2+m(2m-1)t_1^2.
        \end{align*}
    \end{remark}
    
    \begin{proposition}\label{prop: computation W[m;2,0,0]}
    The $\Gm^3$-equivariant class of $[W_{m;2,0}]$ is given as follows:
        \begin{itemize}
            \item If $m$ is odd, 
        \begin{align*}
            [W_{m;2,0}]=&(\xi_{2m}-(m-1)t_2-mt_3)(\xi_{2m}-t_1-(m-1)t_2-(m-1)t_3)\\
            &\cdot(\xi_{2m}-mt_2-(m-1)t_3)(\xi_{2m}-t_1-(m-2)t_2-(m-2)t_3)\\
            &+2t_2t_3 \bigg( m(m-1)\xi_{2m}^2-\frac{m^{2}(m-1)^2}{2} t_2^{2} \bigg).
        \end{align*}
        \item If $m$ is even, 
        \begin{align*}
            [W_{m;2,0}]=&(\xi_{2m}-mt_2-mt_3)(\xi_{2m}-t_1-(m-2)t_2-(m-1)t_3)\\
            &\cdot(\xi_{2m}-(m-1)t_2-(m-1)t_3)(\xi_{2m}-t_1-(m-1)t_2-(m-2)t_3)\\
            &+2t_2t_3 \bigg( m(m-1)\xi_{2m}^2-\frac{m^{2}(m-1)^2}{2}t_2^{2} \bigg).
        \end{align*}
        \end{itemize}
    \end{proposition}
    \begin{proof}
        As done for $[W_{m;1,0}]$, the proof consists of computing the restrictions of $[W_{m;2,0}]$ to the open strata $\mathcal{S}_{\underline{i}}$ for $\underline{i} = (0,2,0), (0,1,1), (0,0,2)$ and verifying that the stated formula agrees with these restrictions.
    \end{proof}
    
   For every $n,m$, the $\GL_3$-counterpart of the maps
    \[
    \begin{tikzcd}[row sep=tiny]
        \psi_{n}:\P(W_n)\arrow[r] & \P(W_{2n}), && f\arrow[r,mapsto] & f^2\\
        \psi_{n,m}:\P(W_n)\times\P(W_m)\arrow[r] & \P(W_{n+m}), && (f,g)\arrow[r,mapsto] & fg
    \end{tikzcd}
    \]
    are
    \[
    \begin{tikzcd}[row sep=tiny]
        \psi_{n}':\P(V_n)\arrow[r] & \P(V_{2n}), && (q,f)\arrow[r,mapsto] & (q,f^2)\\
        \psi_{n,m}':\P(V_n)\times_{\mathcal{S}}\P(V_m)\arrow[r] & \P(V_{n+m}), && (q,f,g)\arrow[r,mapsto] & (q,fg)
    \end{tikzcd}
    \]
    see~\cite[Subsection 3.4]{DL18}. This will be useful in~\S\ref{sec: Chow n>1}.
    
    When $m$ is odd, there is still an action of $\PGL_2$ on $\P^m$, viewing this projective space as parametrizing binary forms (up to scalar) of degree $m$, with the action given by precomposition with the inverse of the matrix. The key difference from the even case is that $\P^m$ is not the projectivization of a $\PGL_2$-representation. However, in~\cite[Section 5]{DL18}, Di Lorenzo showed that, even when $m$ is odd, $\P^m$ admits an easily describable $\GL_3$-counterpart, denoted $\mathcal{Y}_m$. Roughly speaking, the scheme $\mathcal{Y}_1$ is the closed subscheme of $\P(V_1)$ of pairs $(q,l)$ such that the line $L=\{l=0\}$ is tangent to the conic $Q=\{q=0\}$, while $\mathcal{Y}_{2s+1}$ is the closed subscheme of $\P(V_{2s+1})$ defined as the image of $\phi_s':\mathcal{Y}_1\times_{\mathcal{S}}\P(V_s)\rightarrow\P(V_{2s+1})$ sending $(q,l,f)$ to $(q,lf^2)$. Then, one gets $\GL_3$-counterparts to the squaring and multiplication maps in a similar way as we did above for the even case. For instance, the $\GL_3$-counterpart of
    \[
    \begin{tikzcd}
        \phi_{r,s}:\P^{2r+1}\times\P^{2s+1}\arrow[r] & \P(W_{r+s+1}), && (f,g)\arrow[r,mapsto] & fg
    \end{tikzcd}
    \]
    is $\phi_{r,s}':\mathcal{Y}_{2r+1}\times_{\mathcal{S}}\mathcal{Y}_{2s+1}\rightarrow\P(V_{r+s+1})$ induced by the following commutative diagram
    \[
    \begin{tikzcd}[row sep=large]
        \mathcal{Y}_{2r+1}\times_{\mathcal{S}}\mathcal{Y}_{2s+1}\arrow[r,hookrightarrow]\arrow[rrrd,dashed,"\phi_{r,s}'"] & \P(V_{2r+1})\times_{\mathcal{S}}\P(V_{2s+1})\arrow[rr,"\psi_{2r+1,2s+1}'"] && \P(V_{2r+2s+2})\\
        \mathcal{Y}_1\times_{\mathcal{S}}\P(V_r)\times_{\mathcal{S}}\mathcal{Y}_1\times_{\mathcal{S}}\P(V_s)\arrow[r]\arrow[u] & \P(V_1)\times_{\mathcal{S}}\P(V_{r+s})\arrow[rr,"\psi_{1,r+s}'"'] && \P(V_{r+s+1})\arrow[u,"\psi_{r+s+1}'"']
    \end{tikzcd}
    \]
    where the bottom left morphism is induced by $\psi_{r,s}'$ and $\phi_{1,1}':\mathcal{Y}_1\times_{\mathcal{S}}\mathcal{Y}_1\rightarrow\P(V_1)$ sending $(q,l_1,l_2)$ to $(q,l)$, where $l$ is the unique line in $\P^2$ passing through the two tangent points of $L_1$ and $L_2$ with $Q$. 

    We conclude the subsection by extending the definition of $j_{m;r,l}$ (for $0\leq r+l\leq m$) to the odd case.

    \begin{definition}\label{def: k morphism}
        Let $m\geq1$ and $0\leq r+l\leq m$. Define
        \[
        \begin{tikzcd}
            k_{m;r,l}:\mathcal{Y}_{2m+1-2r-2l}\arrow[r] & \mathcal{Y}_{2m+1}
        \end{tikzcd}
        \]
        as the map induced by the commutative diagram
        \[
        \begin{tikzcd}
            \mathcal{Y}_1\times_{\mathcal{S}}\P(V_{m-r-l})\arrow[rr,"\phi'_{m-r-l}"]\arrow[dd,"1\times j_{m;r,l}"] & &\P(V_{2m+1-2r-2l})\arrow[dd,"j_{2m+1;2r,2l}"]\arrow[r,hookleftarrow] & \mathcal{Y}_{2m+1-2r-2l}\arrow[dd,dashed,"k_{m;r,l}"]\\
            \\
            \mathcal{Y}_1\times_{\mathcal{S}}\P(V_{m})\arrow[rr,"\phi'_m"] & &\P(V_{2m+1})\arrow[r,hookleftarrow] & \mathcal{Y}_{2m+1}.
        \end{tikzcd}
        \]
    \end{definition}

    For instance, the image of $k_{m;1,0}$ is set-theoretically the locus of $(q,lf^2)\in\mathcal{Y}_{2m+1}$ where the degree $m$ polynomial $f$ in three variables is divisible by $X_0$. Notice that the maps are all $\Gm^3$-equivariant.

\section{The integral Chow rings for $n=1$}\label{Sec: Chow n=1}

In this section, we prove Theorems \ref{thm: Chow for g even n=1} and \ref{thm: Chow for g odd n=1}. While the overall proof strategy is the same, the computations differ between the two cases. We will handle both simultaneously, highlighting the distinctions at each step.

The first step is to projectivize. Define the natural projection maps as follows:

\begin{itemize}
    \item when $g$ is even, let 
    \[
    \begin{tikzcd}
        p: \Sym^{2g}(V^\vee) \otimes \det (V) ^{\otimes g-1} \otimes \Gamma \smallsetminus \Delta \arrow[r] & \P(\Sym^{2g} (V^\vee)) \smallsetminus \underline{\Delta};
    \end{tikzcd}
    \]
    \item when $g$ is odd, let
    \[
    \begin{tikzcd}
        p: W_g \otimes \chi^{\otimes - 2} \otimes \Gamma \smallsetminus \Delta \arrow[r] & \P(W_g) \smallsetminus \underline{\Delta}.
    \end{tikzcd}
    \]
\end{itemize}
In both cases, $p$ serves as the projection onto the corresponding projective space with the locus $\underline{\Delta}$ of polynomials having either a root at $0$ or $\infty$ or having a double root removed.

\begin{notation}\label{not: xi j various groups}
    When there is no risk of confusion, we will use $G$ to refer to the group $\GI$ or $\GII$ depending on the parity of $g$.

    When $g$ is even, we denote by $\P^{j}$ the space $\P(\Sym^{j}(V^\vee))$ and by $\xi_j =c_1^G(\cO_{\P^{j}}(1)) \in \CH^*_G(\P^{j})$ the first Chern class of the $G$-linearized $\cO_{\P^j}(1)$.

    When $g$ is odd, we denote by $\P^{2j}$ the space $\P(W_j)$ and define $\xi_{2j}=c_1^G(\cO_{\P^{2j}}(1)) \in \CH^*_G(\P^{2j})$ to be the first Chern class of the $G$-linearized $\cO_{\P(W_j)}(1)$. We will also denote by $\xi_{2j}$ the $\PGL_2$-equivariant Chern class of $\cO_{\P(W_j)}(1) \in \CH^*_{\PGL_2}(\P(W_j))$.

    Finally, when both the dimension of $\P^{2j+1}$ and $g$ are odd, the space $\P^{2j+1}$ will be the space of homogeneous polynomials of degree $2j+1$ (up to scalar) and an element $(t,[A])$ of $G$ acts on it by precomposition with the inverse of the matrix $A$. 
\end{notation}

\begin{remark}
        The classes $\xi_{2j} \in \CH^*_{\GL_3}(\P(V_j))$ and $\xi_{2j} \in \CH^*_{\PGL_2}(\P(W_j))$ correspond to each other under the isomorphism $[\P(V_j) /\GL_3] \cong [\P(W_j)/ \PGL_2]$. This correspondence justifies the use of the same name for these classes.
\end{remark}

It is immediate to observe the following.

\begin{lemma}\label{lemma: reduction to prjective spaces}
    The pullback $p^*$ in Chow induces isomorphisms:
    \begin{itemize}
        \item when $g$ is even, 
        $$
        p^*: \frac{\CH^*_{\GI}(\P^{2g} \smallsetminus \underline{\Delta})}{(-\xi_{2g}+(g-1) \beta_1+\gamma)} \xrightarrow{\sim} \CH^*_{\GI}(\Sym^{2g}(V^\vee) \otimes \det (V) ^{\otimes g-1} \otimes \Gamma \smallsetminus \Delta);
        $$
        \item when $g$ is odd, 
        $$
        p^*: \frac{\CH^*_{\GII}(\P^{2g} \smallsetminus \underline{\Delta})}{(\xi_{2g}+2t+\gamma)} \xrightarrow{\sim} \CH^*_{\GI}(W_g \otimes \chi^{\otimes -2} \otimes \Gamma \smallsetminus \Delta).
        $$
    \end{itemize}
\end{lemma}
\begin{proof}
    Depending on the parity of $g$, the morphism $p$ is a $\mathbb{G}_m$-torsor with associated line bundle $\O_{\P^{2g}}(-1) \otimes \mathrm{det}(V)^{\otimes g-1} \otimes \Gamma $ or $\O_{\P^{2g}}(-1) \otimes \chi^{\otimes -2} \otimes \Gamma$.    
\end{proof}

The excision sequence yields

\begin{equation}\label{eqn: excission n=1}
\begin{tikzcd}
    \CH^*_G(\underline{\Delta}) \arrow[r] & \CH^*_G(\P^{2g}) \arrow[r] & \CH^*_G(\P^{2g} \smallsetminus \underline{\Delta}) \arrow[r] & 0.
\end{tikzcd}
\end{equation}

Write $\uD= \uD_1 \cup \uD_2$, where $i_1: \uD_1 \hookrightarrow \P^{2g}$ corresponds to polynomials having a root at $0$ or at $\infty$ and $i_2: \uD_2 \hookrightarrow \P^{2g}$ to those having a double root. These two are both $G$-equivariant subspaces of $\P^{2g}$. Scheme-theoretically, $\uD_1$ has two irredubible components, each of which are hyperplanes in $\P^{2g}$ and are swapped by the action of $\mu_2$. In particular, the quotient $[\uD_1/G]$ is irreducible. We will denote by $\uD_{1,1}$ the intersection of the two components of $\uD_1$ and by $i_{1,1}$ its inclusion in $\P^{2g}$. 

The next lemma is clear.

\begin{lemma}
    The pushforward map
    \[
    \begin{tikzcd}
        \CH^*_G(\uD_1) \oplus \CH^*_G(\uD_2) \arrow[r] & \CH^*_G(\uD)
    \end{tikzcd}
    \]
    is surjective.
\end{lemma}

Thus we may replace $\CH^*_G(\uD)$ with $\CH^*_G(\uD_1) \oplus \CH^*_G(\uD_2)$ in the exact sequence \eqref{eqn: excission n=1} and study the images of $i_{1*}$ and $i_{2*}$ separately.

We start with the image of $i_{1*}$. Since $[(\uD_1 \smallsetminus \uD_{1,1}) / G]$ is isomorphic to $[\mathbb{A}^{2g-1} / (\Gm \times \Gm)]$ when $g$ is even (resp. $[\mathbb{A}^{2g-1} / \Gm] \times B\Gm$ when $g$ is odd), the excision sequence
\[
\begin{tikzcd}
    \CH^*_G(\uD_{1,1}) \arrow[r] & \CH^*_G(\uD_{1}) \arrow[r] & \CH^*_G(\uD_1 \smallsetminus \uD_{1,1}) \arrow[r] & 0
\end{tikzcd}
\]
shows that the ring $\CH^*_G(\uD_1)$ is generated as a $\CH^*(BG)$-algebra by the image of $\CH^*_G(\uD_{1,1})$ and

\begin{itemize}
    \item when $g$ is even, by classes $\alpha_1$ and $\alpha_2$ that restrict in $ \CH^*(\uD_1 \smallsetminus \uD_{1,1})$ to the first Chern classes $\chi_1$ and $\chi_2$ of the standard representations of each factor of $\Gm \times \Gm$;
    \item when $g$ is odd, by a class $\alpha$ that restricts in $ \CH^*(\uD_1 \smallsetminus \uD_{1,1})$ to the first Chern classes of the standard representation of $\Gm$.
\end{itemize}

Additionally, for $g$ even, we have $\beta_1=\alpha_1+\alpha_2$ in $\CH^*(\uD_1 \smallsetminus \uD_{1,1})$. Therefore, in this case, we do not need $\alpha_2$, and we will denote $\alpha_1$ simply by $\alpha$ for consistency.

\begin{lemma}
    The $G$-equivariant Chow ring of $\uD_1 \smallsetminus \uD_{1,1}$ is generated as a $\CH^*(BG)$-module by $1$ and the restriction of $\alpha$.
\end{lemma}
\begin{proof}
    When $g$ is even, it is enough to observe that $\alpha^2=\chi_1^2=\beta_1 \chi_1-\beta_2$ in $\CH^*_G(\uD_1 \smallsetminus \uD_{1,1})$. If instead $g$ is odd, by \cite[page 11]{Vez98}, then $\alpha^2=-c_2$ in $\CH^*_G(\uD_1 \smallsetminus \uD_{1,1})$.
\end{proof}

The following is an immediate corollary. 

\begin{corollary}\label{cor: Chow ring of uD1}
    The $G$-equivariant Chow ring of $\uD_1$ is generated by $\CH^*_G(\uD_{1,1})$, $1$ and $\alpha$ as a $\CH^*(BG)$-module.
\end{corollary}

The map $i_{1,1}:  \uD_{1,1} \hookrightarrow \P^{2g}$, is the inclusion of a linear space. Consequently, the induced pullback $i_{1,1}^*$ is surjective, and by the push-pull formula, the image of the pushforward map $i_{1,1*}$ is the ideal generated by $i_{1,1*}(1)= [\uD_{1,1}] \in \CH^*_G(\P^{2g})$.

In conclusion, the image of $\CH^*_G(\uD_1) \to \CH^*_G(\P^{2g})$ is the ideal generated by $[\uD_{1,1}]$, $[\uD_1]$ and ${i_1}_*(\alpha)$.

Let
\[
\begin{tikzcd}
    \varphi: \P^1 \times \P^{2g-1} \arrow[r] & \P^{2g}   
\end{tikzcd} 
\] 
be the multiplication map $(F,G) \mapsto FG$.

\begin{lemma}\label{lemma: push of varphi}
    We have the following formulas for $\varphi_*$:
    \begin{itemize}
        \item when $g$ is even,  $\varphi_*(\xi_1)=\xi_{2g}$;
        \item when $g$ is odd,  $\varphi_*(\tau)=2 \xi_{2g}$.
    \end{itemize}
\end{lemma}
\begin{proof}
    The case $g$ even is \cite[Lemma 4.4]{Lar19}. We prove the case when $g$ is odd. Since $\tau \in \CH^*_G(\P^1 \times \P^{2g-1})$ is pulled back from $\CH^*_{\Gm \times \PGL_2}(\P^1 \times \P^{2g-1})$, we can compute the $\Gm \times \PGL_2$-equivariant pushforward $\varphi_*(\tau)$ and then restrict it to $\CH^*_G(\P^{2g})$. The advantage is that $\CH^*_{\Gm \times \PGL_2}(\P^{2g})$ is torsion free in degree $\leq 2$ and thus we can work with $\Q$ coefficients. Consider the commutative diagram of $\Gm \times \mathrm{PGL}_2$-equivariant multiplication maps
        \[
        \begin{tikzcd}
            \P^1 \times (\P^1)^{2g-1} \arrow[d,"="]\arrow[r,"\theta"] & \P^1 \times \P^{2g-1}\arrow[d,"\varphi"]\\
            (\P^1)^{2g}\arrow[r,"\rho"] & \P^{2g}
        \end{tikzcd}
        \]
        and notice that $\theta^*(\tau)=\tau_{1}$, where $\tau_i$ denotes the pullback of $\tau$ under the $i$-th projection. Then, we have that
        \begin{align*}
            \varphi_*(\tau)
            &=\frac{1}{(2g-1)!}\varphi_*(\theta_*\theta^*\tau) \\
            &=\frac{1}{(2g-1)!}\rho_*(\tau_{1}) \\
            &=2\xi_{2g},
        \end{align*}
            where in the last equality we have used the identity 
            \begin{equation}\label{eqn: rho(tau)}
                \rho_*(\tau_i)=\frac{\rho_*(\sum_j \tau_j)}{2g}=\frac{\rho_*(\rho^* (2 \xi_{2g}))}{2g}=2(2g-1)! \xi_{2g}
            \end{equation} 
            for all $i=1,\ldots, 2g$.
\end{proof}

\begin{lemma} \label{lemma: class of uD1}
    We have 
    \begin{equation*}
     [\uD_{1}]=
    \begin{cases}
        2 \xi_{2g}-2g \beta_1 &\text{when $g$ is even;}\\
        2\xi_{2g}  &\text{when $g$ is odd. }\, 
\end{cases} 
\end{equation*}
\end{lemma}

\begin{proof}
    We deal with the case $g$ even first. In this case, by \cite[Lemma 32]{CIL24}, $[\{ X, Y\}]= 2 \xi_1-(\beta_1+\gamma) \in  \CH^*_{\GI}(\P^1)$ and 
    $$
    [\uD_1]=\varphi_*([\{ X, Y\}])= \varphi_*(2 \xi_1-(\beta_1+\gamma))= 2 \xi_{2g}- 2g (\beta_1 +\gamma)
    $$
    where in the last equality we have used Lemma \ref{lemma: push of varphi}.

    The case when $g$ is odd is similar, but uses Equation \eqref{eqn: class X,Y}.
    
\end{proof}

\begin{lemma}\label{lemma: class of uD11}
    We have 
    \begin{equation*}
     [\uD_{1,1}]=
    \begin{cases}
        \xi_{2g}^2+ \xi_{2g}(\gamma-2g \beta_1) + 4 g^2 \beta_2 &\text{when $g$ is even;}\\
        \xi_{2g}^2+\gamma \xi_{2g} +g^2 c_2+\gamma^2  &\text{when $g$ is odd. }\, 
\end{cases} 
\end{equation*}
\end{lemma}

\begin{proof}
    Suppose $g$ is even. This relies on certain results about multiplication maps from \cite[Section 4]{Lar19} and 
    \cite[Section 3.2]{CIL24}. In particular, the classes $s_r^j$ are natural elements of $\CH^*(\P(\Sym^r(V^\vee)))$ for $j=0,\ldots,r$, starting with $s_r^0=1$ and $s_r^1= \xi_{r}$. Let $\mathrm{mult}: \P^{2g-2} \times \P^2 \to \P^{2g}$ be the multiplication map. Then, we have
    \begin{align*}
            [\uD_{1,1}]&= \mathrm{mult}_*(1 \otimes [\{XY\}]) \\
            &= \mathrm{mult}_*(s_{2g-2}^0 \otimes s_{2}^2)+(\gamma-\beta_1)\mathrm{mult}_*(s_{2g-2}^0 \otimes s_{2}^1)+2\beta_2 \varphi_*(s_{2g-2}^0 \otimes s_{2}^0)\\
            &=s_{2g}^2+(2g-1)(\gamma-\beta_1)s_{2g}^1+2g(2g-1)\beta_2s_{2g}^0\\
            &=\xi_{2g}^2-2g\beta_1\xi_{2g}+\gamma\xi_{2g}+4g^2\beta_2.
        \end{align*}
        The second equality uses the computation of  $[\{X,Y\}] \in \CH^*_G(\P^2)$ in \cite[Lemma 33]{CIL24}, the pushforwards $\mathrm{mult}_*(s_{2g-1}^\alpha \otimes s_{1}^\beta)$ are computed using \cite[Lemma 4.4]{Lar19}, and in the last equality we have used \cite[Lemma 4.2]{Lar19} to write $s_{2g}^2=\xi_{2g}^2-\beta_1 \xi_{2g}+(2g) \beta_2$.

        Suppose now that $g$ is odd. This time we denote  by $\mathrm{mult}: \P(W_{g-1}) \times \P(W_1) \to \P(W_{g})$ the multiplication map. As before and using Lemma \ref{lemma: class of XY for odd g}, we have
        $$
        [\uD_{1,1}]= \mathrm{mult}_*(1 \otimes [\{XY\}]) = \mathrm{mult}_*(\xi_2^2)+\gamma\cdot\mathrm{mult}_*(\xi_{2})+g(2g-1)(c_2+\gamma^2).
        $$
        We compute the pushforwards $\mathrm{mult}_*(\xi_2^2)$ and $\mathrm{mult}_*(\xi_{2g})$ first $\Gm \times \PGL_2$-equivariantly and then we restrict it to $\GII$. Since the Chow ring $\Gm \times \PGL_2$ has no torsion in degree $\leq 2$, we are allowed to use $\Q$ coefficients. We have a commutative diagram of $\Gm \times \PGL_2$-equivariant maps
        \[
        \begin{tikzcd}
            (\P^1)^{2g-2}\times(\P^1)^2\arrow[d,"="]\arrow[r,"\theta"] & \P^{2g-2}\times\P^2\arrow[d,"\mathrm{mult}"]\\
            (\P^1)^{2g}\arrow[r,"\rho"] & \P^{2g}
        \end{tikzcd}
        \]
        Using that $\theta^*(2\xi_2)=\tau_{2g-1}+\tau_{2g}$ (and thus that $\theta_*(\tau_{2g-1}+\tau_{2g})=4 (2g-2)! \xi_2$) and \eqref{eqn: rho(tau)}, we get that
        \begin{align*}
            \mathrm{mult}_*(\xi_2)&=\frac{1}{2}\mathrm{mult}_*(2\xi_2) \\
            &=\frac{1}{4}\frac{1}{(2g-2)!}\rho_*(\tau_{2g-1}+\tau_{2g})\\
            &=(2g-1)\xi_{2g}.
        \end{align*}
        Similarly, we have
        \begin{align*}
            \mathrm{mult}_*(\xi_2^2)&=\frac{1}{4}\mathrm{mult}_*(4\xi_2^2)\\
            &=\frac{1}{4}\frac{1}{(2g-2)!2}\rho_*((\tau_{2g-1}+\tau_{2g})^2)\\
            &=\frac{1}{8}\frac{1}{(2g-2)!}(2\rho_*(\tau_{2g}^2)+2\rho_*(\tau_{2g-1}\tau_{2g}))\\
            &=\xi_{2g}^2-g(g-1)c_2.
        \end{align*}

        where in the first equality we used that 
        $$
        8 (2g-2)! \xi_{2}^2 = \theta_*(\theta^*((2 \xi_2)^2))=\theta_*((\tau_{2g-1}+\tau_{2g})^2)
        $$
        and in the third equality the identity 
        \begin{equation}\label{eqn rho(taui tauj)}
        \rho_*(\tau_{2g-1}\tau_{2g})=\frac{\rho_*\rho^*((2 \xi_{2g})^2)-\rho_*(\sum_j \tau_j^2)}{2g (2g-1)}=4(2g-2)!\xi_{2g}^2+(2g)(2g-2)! c_2
        \end{equation}
        and $\tau_{i}^2=-c_2$ for all $i=1,\ldots, 2g$.

        Combining these two computations, we get the stated conclusion.
\end{proof}

Next, we identify the class $\alpha$ in Corollary \ref{cor: Chow ring of uD1} and compute its pushforward in $\P^{2g}$. Call
$$
j: \P^{2g-1} \times \{ X,Y\} \to \P^{2g-1} \times \P^1
$$
the $G$-equivariant map obtained from the inclusion $ \{ X, Y \} \hookrightarrow \P^1$ and let
$$
\psi: \P^{2g-1} \times \{ X, Y \} \to \uD_1
$$
be so that $i_1 \circ \psi= \varphi \circ j$. Form the cartesian diagram
\[
        \begin{tikzcd}
            {\left[\frac{\mathbb{A}^{2g-1}\times\{X,Y\}}{G}\right]}\arrow[d,"v"]\arrow[r,"\cong"] \arrow[dr, phantom, "\square"] & {\left[\frac{\uD_1\setminus \uD_{1,1}}{G}\right]}\arrow[d,"u"]\\
            {\left[\frac{\P^{2g-1}\times\{X,Y\}}{G}\right]}\arrow[r,"\psi"'] & {\left[\frac{\uD_1}{G}\right]}.
        \end{tikzcd}
        \]

\begin{lemma}
    We can choose 
    $$
    \alpha=
    \begin{cases}
        \psi_*( j^*(\xi_{1})) &\text{when $g$ is even;}\\
        \psi_*(j^*(\tau)) &\text{when $g$ is odd. }\, 
        \end{cases}
    $$
\end{lemma}

\begin{proof}
    Suppose $g$ is even. We want to show that $u^*(\psi_*(j^*(\xi_1)))=v^*(\xi_1)$ is a generator of $\CH^*_G(\{ X, Y\})= \CH^*(B (\Gm \times \Gm))$ as a $\CH^*(\GI)$-algebra. The image of $\xi_1$ under pullback map
    \[
    \begin{tikzcd}
        \CH^*_{\GI}(\P^1) \arrow[r] & \CH^*_{\GI}(\{X,Y\})=\CH^*(B (\Gm \times \Gm))= \Z[t_1,t_2]
    \end{tikzcd}
    \]
    satisfies the relation $0=\xi_1^2-\beta_1\xi_1+\beta_2=(\xi_1 -t_1)(\xi_1-t_2)$ and must thus be one of the $t_i$. Which of the two $t_i$ depends on which isomorphism $[\{ X,Y \}/ \GII] \cong B (\Gm \times \Gm)$ we choose.
    
    Suppose now that $g$ is odd. Then, we aim to show that $u^*(\psi_*(j^*(\tau)))=v^*(\tau)$ is a generator of $\CH^*_G(\{ X, Y\})= \CH^*(B \Gm)[t]$ as a $\CH^*(\Gm \rtimes \mu_2)[t]$-algebra. Here we identified $[\{X,Y\}/\Gm \rtimes \mu_2]$ with $B \Gm$ using the inclusion $X \in \{X,Y\}$. This follows from the fact that $\tau^2=-c_2 \in \CH^*_{\GII}(\P^1)$ and that (up to a sign) $c_2$ pulls backs to the square of a generator of $\CH^*(B\Gm)$ under $B \Gm \to B(\Gm \rtimes \mu_2)$ (See \cite[page 11]{Vez98}).
\end{proof}

\begin{remark}
    Suppose that $g$ is even. Then, we have $\psi^* i_1^* \xi_{2g}=\xi_{2g-1}+ j^* \xi_1 $, and thus 
    $$
    {i_1}_*(\alpha)= \xi_{2g} \cdot [\uD_1] - (i_1 \circ \psi)_*(\xi_{2g-1}).
    $$
    Thus, the image of $\CH^*_G(\uD_1) \to \CH^*_G(\P^{2g})$ is also the ideal generated by $[\uD_{1,1}]$, $[\uD_1]$ and $(i_1 \circ \psi)_*(\xi_{2g-1})$.
    
    This is closer to the strategy used in \cite{CIL24} for $g=2$. Here, we will directly compute ${i_1}_*(\alpha)$ avoiding the need to distinguish between the cases $g$ odd and $g$ even.
\end{remark}

We can finally compute ${i_1}_*(\alpha)$.
\begin{lemma}\label{lemma: pushforward of alpha}
    The pushforward of $\alpha$ to $\P^{2g}$ is given by:
    $$
    {i_1}_*(\alpha)=\begin{cases}
         (\beta_1 + \gamma) \xi_{2g}-4g \beta_2 &\text{when $g$ is even;}\\
        -2g c_2 &\text{when $g$ is odd. }\, 
        \end{cases}
    $$.
\end{lemma}
\begin{proof}
    Suppose $g$ is even. Then, we have
    \begin{align*}
    {i_1}_*(\alpha)&= {i_1}_* (\psi_*(j^*(\xi_1))) \\
    &= \varphi_*(j_*(j^*\xi_1)) \\
    &= \varphi_*(\xi_1 \cdot [\{X,Y\}]) \\
    &= \varphi_*(  \beta_1 \xi_1 -\gamma \xi_1 -2 \beta_2)\\
    &=  (\beta_1 + \gamma) \xi_{2g}-4g \beta_2
    \end{align*}
    where in the fourth equality we used that $[\{X,Y\}]= 2 \xi_{1}-(\beta_1+\gamma) \in \CH^*_{\GI}(\P^1)$ (see \cite[Lemma 32]{CIL24}) and the projective bundle relation $\xi_1^2-\beta_1 \xi_1 +\beta_2=0$; in the last equality, we applied Lemma \ref{lemma: push of varphi}.

    When $g$ is odd the chain of equalities remains the same, except we replace $\xi_1$ with $\tau$, in the fourth equality use Equation \eqref{eqn: class X,Y} and the relation $\tau^2=-c_2 \in \CH^*_{\GII}(\P^1)$; in the last equality, use that $\gamma$ is $2$-torsion.
\end{proof}

Lemmas \ref{lemma: class of uD1}, \ref{lemma: class of uD11} and \ref{lemma: pushforward of alpha}  together give generators of the image of $\CH^*_G(\uD_1) \to \CH^*_G(\P^{2g})$. 

Next we deal with $\uD_2$. There is a well-known $G$-equivariant envelope (see \cite[Page 603]{EG98} for the definition) of $\uD_2$, namely the disjoint union of the collection of maps
$$
\pi_r: \P^r \times \P^{2g-2r} \to \P^{2g}
$$
given by $(f,g) \mapsto f^2 g$ for $r=1,\ldots,g$ (see \cite[Lemma 3.3]{Vis98},\cite[Proposition 4.1]{EF09}).

It follows that the image of $\CH^*_G(\uD_2) \to \CH^*_G(\P^{2g})$ is equal to the ideal generated by the images of the pushforwards along all $\pi_r$. For other groups $G$, this has been computed. In our case, we have the following.

\begin{proposition}\label{prop: image i_2}
    The image of ${i_2}_*: \CH^*_G(\uD_2) \to \CH^*_G(\P^{2g})$ is the ideal generated by:
    \begin{itemize}
        \item when $g$ is even, by the classes 
        \begin{itemize}
            \item ${\pi_1}_*(1)=2(2g-1) \xi_{2g}-2g(2g-1)\beta_1$,
            \item  ${\pi_1}_*(\xi_{1})=\xi_{2g}^2-\beta_1 \xi_{2g}-2g(2g-2)\beta_2$;
        \end{itemize}
        \item when $g$ is odd, by the classes
        \begin{itemize}
            \item ${\pi_1}_*(1)=2(2g-1) \xi_{2g}$,
            \item ${\pi_1}_*(\tau)=2 \xi_{2g}^2-2g(g-1) c_2$,
            \item ${\pi_2}_*(\xi_{2}^2)=\xi_{2g}^4+(c_2+ 2g(g-1)c_2 )\xi_{2g}^2+c_3 \xi_{2g}+g^2(g-1)^2c_2$
        \end{itemize}
    \end{itemize}
\end{proposition}

The proof is rather lengthy; let us first establish the following statement separately before delving into it. This can be regarded as an improvement of \cite[Corollary 6.3]{DL18}.

\begin{lemma}\label{lemma: pi2(xi2)}
     Let $m \in \Z_{\geq 0}$ and consider the $\PGL_2$-equivariant multiplication maps
     $$
     \pi_r: \P^r \times \P^{2m-r} \to \P^{2m}
     $$
     Then, modulo ${\pi_1}_*(1)$ and ${\pi_1}_*(\tau)$, the pushforward to $\CH_{\PGL_2}^4(\P^{2m})$ of $\xi_2$ under $\pi_2$ is given by
    $$
    {\pi_2}_*(\xi_2^2)=\xi_{2m}^4+(c_2+ 2m(m-1)c_2 )\xi_{2m}^2+c_3 \xi_{2m}+m^2(m-1)^2c_2^2 .
    $$
\end{lemma}

In this section, we will only use this when $m = g$ is even. However, in~\S\ref{sec: Chow n>1}, this assumption will no longer hold. 

\begin{proof}
    When $r$ is even, the $\GL_3$-counterpart of $\pi_r$ is:  
    \[
    \pi'_r=\psi'_{r,m-r} \circ (\psi'_{r/2} \times 1) : \P(V_{r/2}) \times_{\mathcal{S}} \P(V_{m-r}) \to \P(V_m),
    \]  
    where $\psi'_{r,m-r}$ and $\psi'_{r/2}$ are the maps defined in \S\ref{sec: GL3-counterparts}. The main observation is the following obvious equality 
   $$
   [W_{m;2,0}]={\pi'_2}_*([W_{1;1,0}]).
   $$
   (see \S\ref{sec: GL3-counterparts} for the definition of $W_{m;h,\ell}$).
   By Remark \ref{rmk: computation W[m;1,0,0]}, the right-hand-side is equal to
   \begin{equation}\label{eqn: push of W[1;1,0,0]}
       {\pi'_2}_*(\xi_2^2)+t_1 {\pi'_2}_*(\xi_2)+(c_2+t_1^2) {\pi'_2}_*(1)
   \end{equation}
   and we wish to compute ${\pi'_2}_*(\xi_2^2)$. On the other hand, Proposition \ref{prop: computation W[m;2,0,0]} provides an expression for the left-hand side. We now rewrite this expression in terms of the basis $1,t_1,t_2,t_1^2,t_1t_2, t_1^2 t_2$ in Proposition \ref{prop:generalizeddecomposition} and extract the coefficient of $1$. By Proposition \ref{prop:generalizeddecomposition} and Equation \eqref{eqn: push of W[1;1,0,0]}, this coefficient equals ${\pi'_2}_*(\xi_2^2)$. First, using the relation $t_1+t_2+t_3=0$, we obtain
   \begin{align*}
            [W_{m;2,0}]=&(\xi_{2m}+m t_1+t_2)(\xi_{2m}+(m-2)t_1)(\xi_{2m}+(m-1)t_1-t_2)(\xi_{2m}+(m-3)t_1) \\
            &+2(c_2+t_1^2)\bigg( m(m-1)\xi_{2m}^2-\frac{m^{2}(m-1)^2}{2}t_2^{2} \bigg)
    \end{align*}
    when $m$ is odd, and
    \begin{align*}
            [W_{m;2,0}]=&(\xi_{2m}+m t_1)(\xi_{2m}+(m-2)t_1+t_2)(\xi_{2m}+(m-1)t_1)(\xi_{2m}+(m-3)t_1-t_2) \\
            &+2(c_2+t_1^2)\bigg( m(m-1)\xi_{2m}^2-\frac{m^{2}(m-1)^2}{2}t_2^{2} \bigg)
    \end{align*}
    when $m$ is even.
    By forming the products and extracting the coefficient of $1$ we obtain the claimed result. Specifically, the coefficients of $\xi_{2m}^4$ and $\xi_{2m}^3$ are clearly $1$ and $0$ in both cases. The coefficient of $\xi_{2m}^2$, derived from the product, is a sum of the monomials $t_1^2$, $t_1 t_2$ and $t_2^2$. Replacing $t_2^2$ using the identity $t_2^2=-t_1^2-t_1 t_2 -c_2$ and discarding terms involving $t_1^2,t_1t_2$, we are left in both cases with $2 m (m-1) c_2 +c_2$. Similarly, the coefficient of $\xi_{2m}$ is sum of monomials $t_1^2 t_2, t_1^3$ and $t_1 t_2^2$. Substituting the last two with $c_3- t_1 c_2$ and $-c_3-t_1 ^2 t_2$ respectively, discarding terms involving $t_1^2 t_2, t_1,$ and using $2c_3=0$, we are left with $c_3$, independently of the parity of $m$. Finally, the coefficient of $\xi_{2m}^0$ is a sum of monomials $t_1^4, t_1^3t_2,t_1^2t_2^2$, along with the term $c_2^2 m^2(m-1)^2$. We perform the following substitutions: $t_1^2 t_2^2=-t_1^3t_2-c_2t_1^2-t_1^4$, $t_1^3t_2=c_3t_2-c_2t_1t_2$, and $t_1^4=-t_1^2c_2+t_1c_3$. After discarding the monomials where $t_1$ or $t_2$ appear with positive power, we are left with $c_2^2 m^2(g-1)^2$. This concludes the proof.
\end{proof}

\begin{proof}[Proof of Proposition \ref{prop: image i_2}]
    We treat the case $g$ even and $g$ odd differently. 

    First, suppose $g$ is even. Since $\pi= \sqcup_r \pi_r$ is a $G$-equivariant envelope and $\pi_r^*(\xi_{2g})=\xi_r + \xi_{2g-2r}$, it is enough to compute the pushforward of ${\pi_{r*}}(\xi_r^\ell)$ for $\ell=0,\ldots, r$. Note that $G \subseteq \GL_2$, the action on each projective space is the restriction of an action from $\GL_2$ and $\pi_r$ is $\GL_2$-equivariant. Denote by ${\pi_{r*}^{G}}$ and ${\pi_{r*}^{\GL_2}}$ the $G$ and $\GL_2$-equivariant pusforwards respectively, and by $\xi_r^{G}$ and $\xi_r^{\GL_2}$ the equivariant first Chern classes of $\cO_{\P^r}(1)$. Then, the restriction of  
    $
    {\pi_{r*}^{\GL_2}}((\xi_r^{\GL_2})^\ell)
    $
    to the $G$-equivariant Chow ring of $\P^{2g}$ agrees with ${\pi_{r*}^{G}}((\xi_r^{G})^\ell)$. Thus, the statement follows from \cite[Proposition  4.2 and Lemma 4.3]{EF09}.

    Suppose now that $g$ is odd. The pusforwards ${\pi_1}_*(1), {\pi_1}_*(\tau), {\pi_2}_*(\xi_{2}^2) $ can be computed $\Gm \times \PGL_2$-equivariantly and thus the stated formulas follow from \cite[Proposition 5.2]{FV11} and Lemma \ref{lemma: pi2(xi2)} above.
    
    Let $I$ be the ideal in $\CH^*(\P^{2g})$ generated by the classes ${\pi_1}_*(1), {\pi_1}_*(\tau), {\pi_2}_*(\xi_{2}^2) $ and $J$ the image of ${i_2}_*$. We will show that $I \otimes_{\Z} \Z[1/2]= J \otimes_{\Z} \Z[1/2]$ and $I \otimes_{\Z} \Z_{(2)}= J \otimes_{\Z} \Z_{(2)}$. 
    
    After inverting $2$, the natural map 
    \begin{equation}\label{eqn: aux map}
        \begin{tikzcd}
            \CH^*_{\Gm \times \PGL_2}(\P^{2g} \smallsetminus \uD_2)[1/2]\arrow[r] & \CH^*_{\GII}(\P^{2g} \smallsetminus \uD_2)[1/2]
        \end{tikzcd}
    \end{equation}
    becomes an isomorphism. Indeed, by Equation \eqref{eqn: Chow PGL2}, Lemma \ref{thm: Chow BGII} and the projective bundle formula for Chow rings, the map 
    $$
    \begin{tikzcd}
        \CH^*_{\Gm \times \PGL_2}(\P^{2g})[1/2]\arrow[r] & \CH^*_{\GII}(\P^{2g})[1/2]
    \end{tikzcd}
    $$
    is an isomorphism, and thus \eqref{eqn: aux map} is surjective. Injectivity on the other hand is always true: consider the commutative diagram
    \[
        \begin{tikzcd}
            {\CH^*_{\Gm \times \PGL_2}(\P^{2g} \smallsetminus \uD_2)[1/2]}\arrow[d, hook]\arrow[r]  & {\CH^*_{\GII}(\P^{2g} \smallsetminus \uD_2)[1/2]}\arrow[d] &\\
            { \CH^*_{\Gm \times \SL_2}(\P^{2g} \smallsetminus \uD_2)[1/2]}\arrow[r] & {\CH^*_{\Gm \times \widetilde{G}}(\P^{2g} \smallsetminus \uD_2)[1/2]} \arrow[r] & {\CH^*_{\Gm \times \Gm}(\P^{2g} \smallsetminus \uD_2)[1/2]}
        \end{tikzcd}
        \]
        where $\widetilde{G} \subseteq \SL_2$ is the preimage of $\Gm \rtimes \mu_2$ under $\mathrm{SL}_2\rightarrow\mathrm{PGL}_2$, and $\Gm \subseteq \widetilde{G}$ is the subgroups of diagonal matrices. Injectivity of the left vertical arrow follows from \cite[Lemma 5.4]{FV11}. In conclusion, the injectivity of the top horizontal arrow follows if we can prove the injectivity of the composition of the two bottom horizontal arrows. This conclusion holds because $\SL_2$ is special.

        The Chow ring of $\CH^*_{\Gm \times \PGL_2}(\P^{2g} \smallsetminus \uD_2)$ is known by \cite[Proposition 5.2]{GV08} and \cite[Corollary 6.3]{DL18} \footnote{There is a typo in the statement. The correct computation of the image of ${\pi_1}_*$ is Proposition 4.2 of the Arxiv version arXiv:1802.04519}, up to the pushforward of ${\pi_2}_*(\xi_2^2)$ which is computed in Lemma \ref{lemma: pi2(xi2)}. This shows that $I \otimes_{\Z} \Z[\frac{1}{2}]= J \otimes_{\Z} \Z[\frac{1}{2}]$.

        Finally, we show after tensoring with $ \Z_{(2)}$. For $r$ odd, the map $\P^1 \times \P^{r-1} \times \P^{2g-2r} \to \P^{r} \times \P^{2g-2r} $ has odd degree, thus the induced pushforward is surjective. Since the composition of this map with $\pi_r$ factors through $\pi_1$, we conclude that the image of ${\pi_r}_*$ lies in the image of ${\pi_1}_*$ (which is generated by ${\pi_1}_*(1)$ and ${\pi_1}_*(\tau)$ because $\tau^2=-c_2$ in $\CH^*_{\GII}(\P^1)$). If instead $r$ is even, then $\P^r=\P(W_{r/2})$ and, as done above in the case when $g$ is even, the pushforward of $\xi_r^\ell$ can be computed $\Gm \times \PGL_2$-equivariantly and again lies in $I \otimes \Z_{(2)}$ by \cite[Corollary 6.3]{DL18}. This concludes the proof.
\end{proof}

We can finally prove Theorems \ref{thm: Chow for g even n=1} and \ref{thm: Chow for g odd n=1}.

\begin{proof}[Proof of Theorem \ref{thm: Chow for g even n=1} ]
    We have shown that the Chow ring of $\RH_g^1$ is the quotient of 
    $$
    \CH^*(B(\GI))[\xi_{2g}]=\frac{\Z[\beta_1,\beta_2,\gamma,\xi_{2g}]}{(2\gamma,\gamma(\gamma+\beta_1))}
    $$
    modulo the relations:
    \begin{itemize}
        \item from Lemma \ref{lemma: reduction to prjective spaces}: $\xi_{2g}-(g-1)\beta_1+\gamma$;
        \item from $\uD_1$: $2\xi_{2g}-2g \beta_1$, $\xi_{2g}^2+ \xi_{2g}(\gamma-2g \beta_1) + 4 g^2 \beta_2$ and $(\beta_1+\gamma)\xi_{2g}-4g\beta_2$;
        \item from $\uD_2$: $2(2g-1) \xi_{2g}-2g(2g-1)\beta_1$ and 
        $\xi_{2g}^2-\beta_1 \xi_{2g}-2g(2g-2)\beta_2$;
        \item the projective bundle relation, which is a polynomial $p(\xi_{2g})$ of degree $2g+1$ in $\xi_{2g}$ with coefficients in $\CH^*(B\GI)$.
    \end{itemize}
    The check that the relations in the first three bullets are in the ideal 
    \begin{equation}\label{eqn: ideal for g even and n=1}
    (2\beta_1,2\gamma, 4g \beta_2, \gamma(\beta_1+\gamma), \beta_1(\beta_1+\gamma)).
    \end{equation}
    is a simple calculation. We show that also $p(\xi_{2g})=0$ modulo the above ideal. Let $V_i \subseteq \Sym^{2g}(V^\vee)$ be the subspace $k X^{2g-i} Y^i \oplus k X^{i} Y^{2g-i}$ for $i=1,\ldots,g$ ,$i\neq g/2$ and $V_{g/2}=k X^{g/2} Y^{g/2}$. For $f \in \Sym^{2g}(V^\vee)$, write $f= \sum_{i=0}^{2g} a_i X^{2g-i} Y^i$. Then, $\uD_{1,1}$ is the zero locus of a section of $\cO_{\P^{2g}}(1) \otimes V_0$, namely $\uD_{1,1}=\{ a_0 a_{2g}=0 \}$. Thus, we have (formally)
    $$
    [\uD_{1,1}] =c_2^{G}( \cO_{\P^{2g}}(1) \otimes V_0)= \xi_{2g}^2 c_t^G(\cO_{\P^{2g}}(1) \otimes V_0)_{|t=\xi_{2g}^{-1}}
    $$
    while
    $$
    p(\xi_{2g})=\xi_{2g} c_t^G(V_{g/2})_{|t=\xi_{2g}^{-1}} \prod_{i=0, i\neq g/2}^{g} \xi_{2g}^2 c_t^G(V_i )_{|t=\xi_{2g}^{-1}}
    $$
    Here, for any vector bundle $E$ of rank $r$ we set $c_t(E)=1+tc_1(E)+t^2c_2(E)+\ldots+ t^r c_r(E)$. In particular, we conclude that $p(\xi_{2g})$ is in the ideal \eqref{eqn: ideal for g even and n=1}.

    Finally, we show that the listed relations generate the ideal \eqref{eqn: ideal for g even and n=1}. Replacing $\xi_{2g}$ with $(g-1)\beta_1+\gamma$, we get:
    \begin{itemize}
        \item from the second relation: $2\beta_1=0$;
        \item  from the third relation (and using $2\beta_1=0$): $\beta_1^2-\beta_1\gamma+4g^2\beta_2=0$;
        \item from the fifth relation (and using $2\beta_1=0$): $4g(g-1)\beta_2=0$.
    \end{itemize} 
    Multiplying the second bullet by $g-1$ and using $2 \beta_1=0$ yields $\beta_1(\beta_1+\gamma)=0$ and thus also $4g^2\beta_2=0$ and, again from the very last bullet, $4g\beta_2=0$. This concludes the proof.
\end{proof}

\begin{proof}[Proof of Theorem \ref{thm: Chow for g odd n=1} ]
    The proof is similar to the one of Theorem \ref{thm: Chow for g even n=1}. We sketch it here.
    
    The Chow ring of $\RH_g^1$ is the quotient of 
    $$
    \CH^*(B(\GII))[\xi_{2g}]=\frac{\Z[c_2,\gamma,t, \xi_{2g}]}{(2\gamma)}
    $$
    modulo the relations:
    \begin{itemize}
        \item from Lemma \ref{lemma: reduction to prjective spaces}: $\xi_{2g}=2t+\gamma$;
        \item from $\uD_1$: $2\xi_{2g}$, $\xi_{2g}^2+ \gamma \xi_{2g}+g^2 c_2 + \gamma^2$ and $2g c_2$;
        \item from $\uD_2$: $2(2g-1) \xi_{2g}$, 
        $2 \xi_{2g}^2-2g(g-1) c_2$ and $\xi_{2g}^4+(c_2+2g(g-1)c_2)\xi_{2g}^2+c_3\xi_{2g}+ g^2(g-1)^2 c_2^2$;
        \item the projective bundle relation, which is a polynomial $p(\xi_{2g})$ of degree $2g+1$ in $\xi_{2g}$ with coefficients in $\CH^*(B(\GII))$.
    \end{itemize}
    We leave to the reader to check that the relations in the first three bullets generate the ideal 
    \begin{equation}\label{eqn: ideal for g odd and n=1}
    ( 2\gamma, 4 t, \gamma^2+ g c_2).
    \end{equation}
    The fact that $p(\xi_{2g})=0$ modulo the ideal \eqref{eqn: ideal for g odd and n=1} follows exactly as in the proof of Theorem \ref{thm: Chow for g even n=1}.
\end{proof}

\subsection{Geometric interpretation of the generators when $n=1$}\label{subsec: interpretation generators n=1}

In this subsection we describe explicit vector bundles on $\RH_g^1$ whose Chern classes generate the Chow rings; in general, they will depend on the parity of $g$, except for the generator $\gamma$.

First, we show that there exists a representable, finite, étale cover of degree 2
\begin{equation}\label{eq:mu2torsor}
    \begin{tikzcd}
    \Phi_1:\mathcal{H}_{g,2}^w\arrow[r] & \RH_g^1,
\end{tikzcd}
\end{equation}
where the first stack parametrizes double pointed hyperelliptic curves whose markings are Weierstrass points, see~\cite{EH22}. Explicitly, an object of $\mathcal{H}_{g,2}^w$ over a scheme $S$ is a commutative diagram
\[
\begin{tikzcd}
    C\arrow[rd,"f"]\arrow[rr,"q"] & & P\arrow[ld,"\pi"]\\
    & S\arrow[lu,bend left=40,"{\sigma_1,\sigma_2}"]
\end{tikzcd}
\]
where $C\rightarrow S$ is an hyperelliptic curve of genus $g$ realized as a 2 to 1 cover of the Brauer-Severi scheme $P$ of relative dimension 1 over $S$, and $\sigma_1$, $\sigma_2$ are two Weierstrass sections. Notice that $P$ is then equal to the projectivization of the rank-2 vector bundle $\pi_*(\O_P(q\circ\sigma_1))$, in particular it is endowed with a natural $\O(1)$. Then, the map~\eqref{eq:mu2torsor} associates to every object as above the line bundle $\eta=q^*\O(1)\otimes\O(-\sigma_1-\sigma_2)$ on $C$. The isomorphism $\beta:\eta^{\otimes2}\rightarrow\cO$ is ètale locally defined as follows. Trivialize $P$ locally on $S$, so that the first two sections map to $\infty$ and $0$, and then take $\beta$ to be the isomorphism induced by $XY$ as in~\S\ref{subsec:presn=1}.

The map $\Phi_1$ can be described in terms of morphisms between quotient stacks as follows. Recall the presentation $\RH_g^1 \cong [(\mathbb{A}^{2g+1} \setminus \Delta)/G]$, where $G \cong \GI$ if $g$ is even, and $G \cong \GII$ if $g$ is odd (see Theorems~\ref{thm: presentation g even n=1} and~\ref{thm: presentation g odd n=1}). In~\cite[Proposition 4.1]{EH22}, Edidin and Hu proved that $\mathcal{H}_{g,2}^w$ is isomorphic to the quotient of $\mathbb{A}^{2g+1} \setminus \Delta$ by a two-dimensional split torus $\Gm^2$, where the torus action depends on the parity of $g$.

In both cases, the inclusion of $\Gm^2$ as the maximal torus in $G$ and the identity map on $\mathbb{A}^{2g+1} \setminus \Delta$ together induce the morphism~\eqref{eq:mu2torsor}. Hence, $\Phi_1$ is a $\mu_2$-torsor. As such, it induces a 2-torsion element of $\Pic(\RH_g^1)$ given by the Kummer sequence
\[
\begin{tikzcd}
    0\arrow[r] & \mu_2\arrow[r] & \Gm\arrow[r,"x\mapsto x^2"] & \Gm\arrow[r] & 0
\end{tikzcd}
\]
and the isomorphism $\mathrm{H}_{\text{ét}}^1(\RH_g^1,\mathbb{G}_m)\cong\Pic(\RH_g^1)$.
\begin{proposition}\label{prop:geometricgamma}
    The element in $\Pic(\RH_g^1)$ associated to the $\mu_2$-torsor $\Phi_1:\mathcal{H}_{g,2}^w\rightarrow\RH_g^1$ is $\gamma$.
\end{proposition}
\begin{proof}
    By the above description of $\Phi_1$, there exists a commutative diagram with cartesian squares
    \begin{equation}\label{eq: diag mu2 torsor}
    \begin{tikzcd}
        \mathcal{H}_{g,2}^w\arrow[d]\arrow[r] & \RH_g^1\arrow[d]\\
        B(\Gm\times\Gm)\arrow[d]\arrow[r] & BG\arrow[d]\\
        \mathrm{Spec}(k)\arrow[r] & B\mu_2
    \end{tikzcd}
    \end{equation}
    where $BG\rightarrow B\mu_2$ is induced by the quotient $G \rightarrow\mu_2$. Notice that the class in $\Pic(B\mu_2)\cong\Z/2\Z$ induced by the universal $\mu_2$-torsor is the generator $\gamma$. Since taking the 2-torsion line bundle associated to a $\mu_2$-torsor is functorial, the class in $\Pic(\RH_g^1)$ associated to $\mathcal{H}_{g,2}^w\rightarrow\RH_g^1$ is $\gamma$.
\end{proof}

There is an alternative description of $\gamma$, analogous to~\cite[Lemma 41]{CIL24}. Let $U= \mathbb{A}^{2g+1} \smallsetminus \Delta$. Over $U$ we have the following commutative diagram of $G$ equivariant maps

\begin{equation*}\label{eqn: commutative triangles}
    \adjustbox{scale=0.95,center}{
        \begin{tikzcd}
         D=\sigma_0(U) \sqcup \sigma_{\infty}(U) \arrow[hookrightarrow]{r} \arrow{dr}{\delta} & C =\underline{\mathrm{Spec}}_{\P_U}( \mathcal{O}_{\P^1_U} \oplus \mathcal{O}_{\P_U}(-g-1))  \arrow{r}{q} \arrow{d}{\pi}  & \P^1_U   \arrow{dl}{\rho}  \\
        & U  & &
        \end{tikzcd}
    }
\end{equation*}

The action of $G$ on $C$ is given by Lemma \ref{lemma: group G} and we see that $D$ is $G$-equivariant. Moreover, $\delta$ descends to a degree $2$ cover
\[
\begin{tikzcd}
    \delta: [D/G] \arrow[r] & \RH_g^1
\end{tikzcd}
\]
which, from the description $\cH^w_{g,2}=[(\mathbb{A}^{2g+1} \smallsetminus \Delta)/ \Gm\times\Gm]$, we may identify with $\Phi_1$.

\begin{lemma}\label{lem: alternative description gamma}
    We have $c_1^G(\delta_*\mathcal{O}_D)=c_1({\Phi_1}_*\mathcal{O}_{\mathcal{H}_{g,2}^w})=\gamma$.
\end{lemma}

This is an immediate consequence of the following general fact, and diagram~\eqref{eq: diag mu2 torsor}.

\begin{lemma}
    Let $f:\mathcal{Y}\rightarrow\mathcal{X}$ be a $\mu_n$-torsor between algebraic stacks over a field $k$ of characteristic not dividing $n$, and let $\alpha\in\Pic(\mathcal{X})$ be the associated class. Then,
    \[
    c_1(f_*\cO_{\mathcal{Y}})=
    \begin{cases}
    0 & \text{if }n\text{ is odd},\\
    \frac{n}{2}\alpha & \text{ if }n\text{ is even}.
    \end{cases}
    \]
    In particular, if $n=2$ then $c_1(f_*\cO_{\mathcal{Y}})=\alpha$.
\end{lemma}
\begin{proof}
    We have a cartesian diagram
    \[
    \begin{tikzcd}
        \mathcal{Y}\arrow[r,"f"]\arrow[d,"\psi"] & \mathcal{X}\arrow[d,"\phi"]\\
        \mathrm{Spec}(k)\arrow[r,"p"] & B\mu_n.
    \end{tikzcd}
    \]
    with $\alpha=\phi^*(\beta)$, where $\beta\in\Pic(B\mu_n)\cong\Z/n\Z$ is the generator associated to the universal $\mu_n$-torsor $p$. Together with the isomorphism $\phi^*p_*k\cong f_*\psi^*k=f_*\cO_{\mathcal{Y}}$, this reduces to the case $f=p$ and $\alpha=\beta$. In this case, we have $p_*k=\mathds{1}\oplus\Lambda \oplus\ldots\oplus\Lambda^{\otimes (n-1)}$, where $c_1(\Lambda)=\beta$.
    This holds because for every finite group $G$ and universal $G$-torsor $p:\mathrm{Spec}(k)\rightarrow BG$, the coherent sheaf $p_*k$ is the regular representation, which splits as above by Maschke's Theorem.
\end{proof}

Next, we identify the other generators. There is a morphism $\RH_g^1 \to \cH_g$ that remembers only the hyperelliptic curve. This morphism can be described as a map of quotient stacks as follows. By Theorems \ref{thm: presentation g even n=1} and \ref{thm: presentation g odd n=1}, together with~\cite[Corollary 4.7]{AV04}, the map is given as  
\[
\RH_g^1 = \bigg[ \frac{\mathbb{A}^{2g+1} \smallsetminus \Delta }{G}\bigg] \to \cH_g = \bigg[ \frac{\mathbb{A}^{2g+3} \smallsetminus \Delta}{G'} \bigg],
\]  
where $G \subseteq G'$ depends on the parity of $g$. Specifically:  
\begin{itemize}
    \item For $g$ even, $G = \GI$ and $G' = \GL_2$,
    \item For $g$ odd, $G = \GII$ and $G' = \Gm \times \PGL_2$.
\end{itemize}

This map associates to a degree $2g$ homogeneous polynomial $f(X,Y)$ the degree $2g+2$ homogeneous polynomial $XY f(X,Y)$, inducing the above morphisms thanks to Equations \eqref{eqn: map to Hg g even} and \eqref{eqn: map to Hg g odd}.  

It follows that the classes $\beta_1, \beta_2$ in Theorem \ref{thm: Chow for g even n=1} and $c_1, c_2, c_3, t$ in Theorem \ref{thm: Chow for g odd n=1} are pulled back from $\cH_g$. These classes are already known to be the Chern classes of certain vector bundles.  

Let $\pi: \mathcal{C}_g \to \cH_g$ be the universal curve and $\mathcal{W} \subseteq \mathcal{C}_g$ the universal Weierstrass divisor. For $g$ even, we have  
\[
\beta_i = (-1)^i c_i(\mathcal{V}_g) \footnote{In ~\cite[Proposition 5.2]{EF09}, the authors interpret $\GL_2$ as $\mathrm{Aut}(\P^1,\cO_{\P^1}(1))$, while for us it is $\mathrm{Aut}(\P^1,\cO_{\P^1}(-1))$. In particular, the standard $\GL_2$ representation of $\GL_2$ differ by a dual and this justifies the difference of signs. },
\]  
where $\mathcal{V}_g$ is the rank $2$ vector bundle $\pi_* \omega_\pi^{\otimes g/2}((1-g/2) \mathcal{W})$. See~\cite[Proposition 5.2]{EF09} for a proof.  

For $g$ odd,  
\[
c_i = c_i(\mathcal{E}_g),
\]  
where $\mathcal{E}_g$ is the rank $3$ vector bundle $\pi_* \omega_\pi^{\vee}(\mathcal{W})$, and  
\[
t = c_1(\mathcal{L}_g),
\]  
where $\mathcal{L}_g = \pi_* \omega_\pi^{\otimes (g+1)/2}(\frac{(1-g)}{2} \mathcal{W})$. See the discussion after ~\cite[Theorem 7.2]{DL18} or \cite[Theorem 7.2]{FV11}.

\section{The integral Chow rings for $1<n<(g+1)/2$ and $g$ odd } \label{sec: Chow n>1}
In this section we prove Theorem~\ref{thm: Chow Dab} and obtain from that Theorem~\ref{thm: Chow n>1}. 

We will always assume $n>1$ and $b\geq a\geq1$. For the computation of the Chow ring of $\cD_{2a,2b}$ we will restrict to $b\geq a>1$; most of the arguments work for $1=a\leq b$ as well, but the result is slightly different and we will not need it, as in the end we will set $a=n>1$.

Thanks to the presentation of $\cD_{2a,2b}$ found in Equation~\eqref{eqn: presentation Dab}, we need to compute the $\mathrm{PGL}_2$-equivariant Chow ring of $\P(W_a)\times\P(W_{b})\setminus\underline{\Delta}$.
We start with the $\PGL_2$-equivariant Chow ring of $\P(W_{a})\times\P(W_{b})$.
\begin{lemma}\label{lem: Chow product proj spaces}
    We have
    \begin{equation}
        \CH_{\PGL_2}^*(\P(W_{a})\times\P(W_{b}))\cong\frac{\mathbb{Z}[c_2,c_3,\xi_{2a},\xi_{2b}]}{(2c_3,p_{a}(\xi_{2a}),p_{b}(\xi_{2b}))}
    \end{equation}
    where $p_{j}(\xi_{2j})$ are the monic polynomials of degree $2j+1$ coming from the projective bundle formula applied to the single factors.
\end{lemma}
\begin{proof}
    Apply the projective bundle formula to the projective bundle map $[(\P(W_{a})\times\P(W_{b}))/\PGL_2]\rightarrow[\P(W_{a})/\PGL_2]$.
\end{proof}

\subsection{Equivariant Chow Envelopes and First Relations}\label{sec: Chow envelope}
By Lemma~\ref{lem: Chow product proj spaces} and the excision sequence for Chow groups, we are left with computing the image of the pushforward along the inclusion
\[
\begin{tikzcd}
    i:\uD\arrow[r,hookrightarrow] & \P(W_{a})\times\P(W_{b}).
\end{tikzcd}
\]
Recall that $\uD$ is the closed subscheme parametrizing pairs $(f,g)$ of homogeneous polynomials of degree $a$ and $b$ respectively such that $FG$ admits a square factor. It follows that $\uD$ has three $\PGL_2$-equivariant irreducible components $\uD_1$, $\uD_2$ and $\uD_{1,2}$, parametrizing pairs $(F,G)$ such that $F$ is not square free, $F$ is not square free, and that $F,F$ share a common factor, respectively. In particular, it is enough to compute the image of the $\PGL_2$-equivariant pushforward along the inclusions $i_1:\uD_1\hookrightarrow\P(W_a)\times\P(W_b)$, $i_2:\uD_2\hookrightarrow\P(W_a)\times\P(W_b)$ and $i_{1,2}:\uD_{1,2}\hookrightarrow\P(W_a)\times\P(W_b)$. This is done by finding Chow envelopes for every component. Let $r$ be a positive integer. We interpret $\P^r$ as the space of binary forms of degree $r$ up to scalar and that $\PGL_2$ acts on it by precomposition by the inverse of the matrix. When $r$ is even, this action coincides with the one on $\P(W_{r/2})$; so this is not in contrast with Notation \ref{not: xi j various groups}. Note instead that when $r$ is odd, there is no $\PGL_2$-representation whose projectivization is $\P^r$. There are squaring and multiplication maps
\[
\begin{tikzcd}[row sep=tiny]
    F_r:\P^r\times(\P^{2a-2r}\times\P^{2b})\arrow[r] & \P^{2a}\times\P^{2b}, & (h,(f,g))\arrow[r,mapsto] & (h^2f,g),\\
    G_r:\P^r\times(\P^{2a}\times\P^{2b-2r})\arrow[r] & \P^{2a}\times\P^{2b}, & (h,(f,g))\arrow[r,mapsto] & (f,h^2g),\\
    M_r:\P^r\times(\P^{2a-r}\times\P^{2b-r})\arrow[r] & \P^{2a}\times\P^{2b}, & (h,(f,g))\arrow[r,mapsto] & (hf,hg).
\end{tikzcd}
\]
These maps are clearly $\PGL_2$-equivariant, with the usual action of $\PGL_2$ on products of projective spaces, and factor through $\uD_1$, $\uD_2$ and $\uD_{1,2}$, respectively. Denote by $F$, $G$ and $M$ the disjoint union of the maps $F_r$, $G_r$ and $M_r$, respectively.
\begin{lemma}\label{lem: Chow envelopes n>1}
    The morphisms $F$, $G$ and $M$ form $\PGL_2$-equivariant Chow envelopes of $\uD_1$, $\uD_2$ and $\uD_{1,2}$, respectively, and the $\PGL_2$-equivariant pushforward is surjective.
\end{lemma}
\begin{proof}
    The proof is standard. For the maps $F$ and $G$, see~\cite[Lemmas 3.2, 3.3]{Vis98},~\cite[Proposition 4.1]{EF09} and~\cite[Section 5]{DL18}.
\end{proof}

The pushforward along $F_r$ and $G_r$ is essentially already known, with the only missing component provided by Lemma~\ref{lemma: pi2(xi2)}.  
We denote by $\tau$ both the $\PGL_2$-equivariant first Chern class $c_1^{\PGL_2}(\mathcal{O}_{\mathbb{P}^1}(2))$ on $\mathbb{P}^1$ and its pullback along projections from products of projective spaces.

\begin{lemma}\label{lem: computation i1 i2}
    The image of ${i_1}_*$ in $\CH_{\PGL_2}^*(\mathbb{P}^{2a} \times \mathbb{P}^{2b})$ is the ideal generated by the following elements:
    \begin{align*}
        {F_1}_*(1) &= 2(2a-1)\xi_{2a}, \qquad {F_1}_*(\tau) = 2\xi_{2a}^2 - 2a(a-1)c_2, \\
        {F_2}_*(\xi_2^2) &= \xi_{2a}^4 + (c_2 + 2a(a-1)c_2)\xi_{2a}^2 + c_3\xi_{2a} + a^2(a-1)^2c_2^2.
    \end{align*}
    Similarly, the image of ${i_2}_*$ is the ideal generated by:
    \begin{align*}
        {G_1}_*(1) &= 2(2b-1)\xi_{2b}, \qquad {G_1}_*(\tau) = 2\xi_{2b}^2 - 2b(b-1)c_2, \\
        {G_2}_*(\xi_2^2) &= \xi_{2b}^4 + (c_2 + 2b(b-1)c_2)\xi_{2b}^2 + c_3\xi_{2b} + b^2(b-1)^2c_2^2.
    \end{align*}
\end{lemma}

\begin{proof}
    The proof follows directly from~\cite[Corollary 5.3]{DL18} and Lemma~\ref{lemma: pi2(xi2)}.
\end{proof}

\subsection{Computation of ${M_r}_*$}

We are left with computing the image of ${i_{1,2}}_*$, and it suffices to do so modulo the relations we have already established. The strategy is as follows. First, we compute some low-degree pushforwards and define $J$ to be the ideal generated by them. Next, we show that the image of ${M_r}*$ is contained in $J$ after inverting $2$, as well as after inverting all integers except $2$. This will imply that the image of ${M_r}*$ is contained in $J$.

\subsubsection{Low degree classes and computation $\otimes\Z[1/2]$}

We begin by computing the pushforwards of $1$ and $\tau$ along $M_1$.

\begin{lemma}\label{lemma: M(1) and M(tau)}
    We have the equalities in $\CH^*_{\PGL_2}(\P^{2a} \times \P^{2b})$
    \begin{equation}
        {M_1}_*(1)=2b\xi_{2a}+2a\xi_{2b},\qquad {M_1}_*(\tau)=2\xi_{2a}\xi_{2b}-2abc_2.
    \end{equation}
\end{lemma}
\begin{proof}
     Since the $\PGL_2$-equivariant Chow ring of $\P^{2a}\times\P^{2b}$ is torsion free in degree $\leq2$, we can work $\otimes_{\Z} \Q$ and then clear denominators. Consider the following commutative diagram
    \[
    \begin{tikzcd}
        \P^1\times(\P^1)^{2a-1}\times(\P^1)^{2b-1}\arrow[d,"\sigma"]\arrow[r,"\delta"] & (\P^1)^{2a+2b}\arrow[d,"\rho"]\\
        \P^1\times(\P^{2a-1}\times\P^{2b-1})\arrow[r,"M_1"] & \P^{2a}\times\P^{2b}
    \end{tikzcd}
    \]
    where $\sigma$ and $\rho$ are (products of) multiplication maps, and $\delta(h,(f_1,\ldots,f_{2a-1}),(g_1,\ldots,g_{2b-1}))=(h,f_1,\ldots,f_{2a-1},h,g_1,\ldots,g_{2b-1}))$.
    Then, similarly to Lemma \ref{lemma: class of uD11}, we have
    \begin{align*}
        2(2a-1)!(2b-1)!{M_1}_*(1)&=\rho_*\delta_*(2) \\
        &=\rho_*(2\Delta_{1,2a+1}) \\
        &=\rho_*(\tau_1+\tau_{2a+1})\\
        &=2(2a-1)!(2b-1)!(2b\xi_{2a}+2a\xi_{2b})
    \end{align*}
    where in the first equality, we used the fact that $\sigma$ has degree $(2a-1)!(2b-1)!$. Here, $\Delta_{i,j}$ denotes the pullback of the diagonal of $\P^1 \times \P^1$ along the $(i,j)$-projection. The last equality follows from Equation \eqref{eqn: rho(tau)}, taking into account that $\rho$ now has degree $(2a)!(2b)!$.
    Simplifying one gets the first equation in the statement. Similarly,
    \begin{align*}
        2(2a-1)!(2b-1)!{M_1}_*(\tau)&=\rho_*\delta_*(2\tau) \\
        &=\rho_*(2\tau_1\Delta_{1,2a+1}) \\
        &=\rho_*(\tau_1^2+\tau_1\tau_{2a+1})\\
        &=-(2a)!(2b)!c_2+\rho_*(\tau_1\tau_{2a+1})\\
        &=-(2a)!(2b)!c_2+4(2a-1)!(2b-1)!\xi_{2a}\xi_{2b}.
    \end{align*}
    where in the last equality we used again Equation \eqref{eqn: rho(tau)}, but kept into account that $\rho$ now has degree $(2a)!(2b)!$. The second equation follows.
\end{proof}

To handle the case ${M_r}_*\otimes\Z[1/2]$ an important tool is~\cite[Lemma 5.4]{FV11}, which in particular implies that for every $r\geq0$ the pullback
    \[
    \begin{tikzcd}
        \CH_{\PGL_2}^*(\P^r\times\P^{2a-r}\times\P^{2b-r})\otimes\Z[1/2]\arrow[r] & \CH_{\SL_2}^*(\P^{r}\times\P^{2a-r}\times\P^{2b-r})\otimes\Z[1/2]
    \end{tikzcd}
    \]
    is injective. Here, $\mathbb{P}^j = \mathbb{P}(\Sym^j(V^\vee))$ on the right-hand side, and $V$ is regarded as an $\SL_2$ representation via the inclusion $\SL_2 \subseteq \GL_2$.

    In fact, the above is an isomorphism. To see this, denote by $\xi_j = c_1^{\SL_2}(\mathcal{O}_{\mathbb{P}^j}(1))$. The ring on the right is generated as a $\CH^*(B\SL_2)$-algebra by $\xi_{2r}$, $\xi_{2a-r}$, and $\xi_{2b-r}$. Since the map $\CH^*(B\PGL_2)[1/2] \to \CH^*(B\SL_2)[1/2]$ is an isomorphism, and $\mathcal{O}_{\mathbb{P}^r}(2)$ always admits a $\PGL_2$-linearization, the morphism is also surjective. Hence, it is an isomorphism.
    
    This allows us to perform computations $\SL_2$-equivariantly, which is more straightforward since one can reduce to the maximal torus, given that $\SL_2$ is a special group scheme.

\begin{proposition}\label{prop: Chow n>1 inverting 2}
    The ideal generated by the image of all the maps
    \[
    \begin{tikzcd}
        {M_r}_*:\CH_{\PGL_2}^*(\P^r\times\P^{2a-r}\times\P^{2b-r})\otimes\Z[1/2]\arrow[r] & \CH_{\PGL_2}^*(\P^{2a}\times\P^{2b})\otimes\Z[1/2]
    \end{tikzcd}
    \]
    is contained in the ideal generated by the image of ${F_1}_*$, ${G_1}_*$ and ${M_1}_*$. Moreover, inverting $2$, the image of ${M_1}_*$ is generated by ${M_1}_*(1)$ and ${M_1}_*(\tau)$.
\end{proposition}

\begin{proof}
    The above discussion explains how to reduce to the $\SL_2$-equivariant setting. In this context, every projective space is obtained as the projectivization of an $\SL_2$-representation. This allows us to apply the projective bundle formula (along with the projection formula) to conclude that, for every $r$, the image of ${M_r}_*$ is the ideal generated by ${M_r}_*(\xi_r^i)$ for $0 \leq i \leq r$, where $\xi_r = c_1^{\SL_2}(\mathcal{O}_{\mathbb{P}^r}(1))$. Since $\tau = 2\xi_1$, the final assertion of the proposition is immediate.

    Let $\Gm \subseteq \SL_2$ be the maximal torus consisting of diagonal matrices with determinant $1$. Since $\SL_2$ is special, for every smooth $\SL_2$-scheme $X$ and ideal $I \subseteq \CH^*_{\SL_2}(X)$, one has 
    \[
    ( I \cdot \CH^*_{\Gm}(X) ) \cap \CH^*_{\SL_2}(X) = I,
    \]
    see~\cite[Lemma 2.1]{FuVi}.
    In particular, it is enough to show that all ${M_r}_*(\xi_r^i)$ lie in the ideal generated by the images of ${F_1}_*$, ${G_1}_*$, and ${M_1}_*$, working $\Gm$-equivariantly. 
    
    Use the coordinates $a_0, \ldots, a_r$ on $\P^r$, so that every point $g \in \P^r$ can be written as $g=\sum_ia_iX^iY^{r-i}$. Notice that the coordinate hyperplanes $H_i^{(r)}=\{a_i=0\}$ are $\Gm$-invariant, thus they define $\Gm$-equivariant classes $h_i^{(r)}$. Let $t$ be the generator of the Chow ring of $B\Gm$ given by the character associated to 
    $\begin{bmatrix}
        \lambda & 0\\
        0 & \lambda^{-1}
    \end{bmatrix}\mapsto\lambda$. From~\cite[Lemma 2.4]{EF09} it follows that $h_i^{(r)}=\xi_r-(r-2i)t$.
    
    Let $I_1$ be the ideal generated by the image of the pushforward along the $\mathbb{G}_m$-equivariant maps $F_1$, $G_1$ and $M_1$. For every $r\geq2$ there exist commutative diagrams of $\Gm$-equivariant maps
    \[
    \begin{tikzcd}
        \P^{r-1}\times\P^{2a-r}\times\P^{2b-r}\arrow[rd]\arrow[r] & \P^r\times\P^{2a-r}\times\P^{2b-r}\arrow[r,"M_r"] & \P^{2a}\times\P^{2b}\\
        & \P^{r-1}\times\P^{2a-r+1}\times\P^{2b-r+1}\arrow[ru,"M_{r-1}"']
    \end{tikzcd}
    \]
    where the first horizontal map sends $(h,f,g)$ to $(Xh,f,g)$, and the other non-labeled map sends $(h,f,g)$ to $(h,Xf,Xg)$. It follows that it is enough to prove that ${M_r}_*(1)\in I_1$ for every $r\geq2$. Moreover, using the degree $2$ multiplication map $\P^1\times\P^1\rightarrow\P^2$, we can actually assume that $r>2$. From the commutative diagram
    \[
    \begin{tikzcd}
        (\P^1)^r\times(\P^1)^{2a-r}\times(\P^1)^{2b-r}\arrow[r,"\widetilde{M}_r"]\arrow[d] & (\P^1)^{2a}\times(\P^1)^{2b}\arrow[d,"\rho"]\\
        \P^r\times\P^{2a-r}\times\P^{2b-r}\arrow[r,"M_r"] & \P^{2a}\times\P^{2b}
    \end{tikzcd}
    \]
    one deduces that $r!(2a-r)!(2b-r)!{M_r}_*(1)=\rho_*(\widetilde{M}_{r*}(1))$. Now, the image of $\widetilde{M}_{r}$ is the locus where the first $r$ components of $(\P^1)^{2a}$ are equal to the first $r$ components of $(\P^1)^{2b}$, which is then a product of diagonals of copies of $\P^1\times\P^1$. A computation using~\cite[Lemma 2.5]{EF09} shows that ${M_r}_*(1)$ is equal to
    \begin{equation*}
        \sum_{s=0}^r\sum_{k=0}^{r-s}2^{s}\binom{r-k}{s}\binom{2a-k}{a-r}\binom{2b-r+k}{2b-r}\frac{(2b-r+s+k)!}{(2b-r+k)!}t^{s}h_0^{(2a)}\cdots h_{k-1}^{(2a)}\boxtimes h_0^{(2b)}\cdots h_{r-s-k-1}^{(2b)}
    \end{equation*}
    where empty products are considered to be equal to 1. Note that $h_0^{(2a)}h_1^{(2a)}$, $h_0^{(2b)}h_1^{(2b)}$, $h_0^{(2a)}h_0^{(2b)}$, $b\xi_{2a}+a\xi_{2b}$ and also
    \[
        \frac{(2a)!(2b)!}{r!(2a-r)!(2b-r)!}t^{r}
    \]
    are contained in $ I_1 \otimes \Z[1/2]$. For the last one use that $r>2$ and that
    $$
    2a (2a-1)(2a-2)t^2 =2(2a-1)\xi_{2a}^2= {F_1}_*(1) \xi_{2a} \in I_1\otimes\mathbb{Z}[1/2]. 
    $$
    Using these facts one concludes.
\end{proof}

This Proposition allows us to prove the following Lemma as well.

\begin{lemma}\label{lem: pushforward M1}
    The image of ${M_1}_*$ is contained in the ideal $I_1$ generated by the classes ${M_1}_*(1)$, ${M_1}_*(\tau)$ and the image of ${F_1}_*$ and ${G_1}_*$.
\end{lemma}
\begin{proof}
    Thanks to Proposition~\ref{prop: Chow n>1 inverting 2}, we know that the containment holds after inverting $2$. Therefore, it is enough to show the containment after applying $\otimes \Z_{(2)}$; for the rest of the proof we will then assume to have inverted every odd integer. Notice that the morphism
    \[
    \begin{tikzcd}
        \psi:(\P^1)^3\times\P^{2a-2}\times\P^{2b-2}\arrow[r] & \P^1\times\P^{2a-1}\times\P^{2b-1}
    \end{tikzcd}
    \]
    given by $(h,h_1,h_2,f,g)\mapsto(h,h_1f,h_2g)$ has odd degree, hence it induced a surjective pushforward in Chow. In particular, it is enough to compute the image of the pushforward along $\psi_1:=M_1\circ\psi$. The advantage is that $2a-2$ and $2b-2$ are now even and the projective bundle formula applied to $\P^{2a-2}$ and $\P^{2b-2}$ shows that it is enough to compute the pushforward along $\psi_1$ of classes coming from $(\P^1)^3$. The $\PGL_2$-equivariant Chow ring of $(\P^1)^3$ is computed in Lemma ~\ref{lem: PGL2 Chow prod proj lines}, where we showed that it is generated as a $\CH^*(B\PGL_2)$-module by 1, $\tau_1$, $[\Delta_{1,2}]$, $[\Delta_{1,3}]$, $[\Delta_{1,2}]\tau_1$, $[\Delta_{1,3}]\tau_1$, $[\Delta_{1,2}][\Delta_{1,3}]$ and $[\Delta_{1,2}][\Delta_{1,3}]\tau_1$. The pushforwards ${\psi_1}_*(1)= (2a-1) (2b-1) {M_1}_*(1)$ and ${\psi_1}_*(\tau_1)= (2a-1)(2b-1) {M_1}_*(\tau)$ are computed in Lemma \ref{lemma: M(1) and M(tau)}. All the other classes can be regarded as pushforward along some diagonal map. As the composite of a diagonal together with the map $\psi_1$ factors through $F_1$ or $G_1$, this concludes.
\end{proof}

Next, we compute the pushforward ${M_2}_*$. Notice that in this case we are working with projective spaces of even dimension, which are the projectivization of a $\PGL_2$-representation.

First, we need a couple of lemmas that will be useful also later.
We will need to work with $\GL_3$-counterparts $F_r'$, $G_r'$ and $M_r'$ to $F_r$, $G_r$ and $M_r$, respectively. For $F$, the $\GL_3$-counterparts are
\[
\begin{tikzcd}[row sep=tiny]
    F_{2r}':=\pi_{2r}'\times1:\P(V_r)\times_{\mathcal{S}}\P(V_{a-2r})\times_{\mathcal{S}}\P(V_{b})\arrow[r] & \P(V_a)\times_{\mathcal{S}}\P(V_b),\\
    F_{2r+1}':=\pi_{2r+1}'\times1:\mathcal{Y}_{2r+1}\times_{\mathcal{S}}\P(V_{a-2r-1})\times_{\mathcal{S}}\P(V_{b})\arrow[r] & \P(V_a)\times_{\mathcal{S}}\P(V_b),
\end{tikzcd}
\]
and similarly for $G$; see~\cite[Section 5]{DL18} for the definition of $\pi_r'$. For $M$, the $\GL_3$-counterparts
\[
\begin{tikzcd}[row sep=tiny]
    M_{2r}'= (\psi'_{r,r-a} \times \psi'_{r,r-b}) \circ \Delta_{\P(V_r)}:\P(V_r)\times_{\mathcal{S}}\P(V_{a-r})\times_{\mathcal{S}}\P(V_{b-r})\arrow[r] & \P(V_{a})\times_{\mathcal{S}}\P(V_{b}),\\
    M_{2r+1}'= (\phi'_{r,a-r-1} \times \phi'_{r,b-r-1}) \circ \Delta_{\mathcal{Y}_{2r+1}}:\mathcal{Y}_{2r+1}\times_{\mathcal{S}}\mathcal{Y}_{2a-2r-1}\times_{\mathcal{S}}\mathcal{Y}_{2b-2r-1}\arrow[r] & \P(V_{a})\times_{\mathcal{S}}\P(V_{b}),
\end{tikzcd}
\]
where the maps $\psi'_{i,j}$ and $\phi'_{i,j}$ are defined in \S\ref{sec: GL3-counterparts}.

Recall that, by Remark~\ref{rmk: computation W[m;1,0,0]}, for every $r>0$, we have $[W_{r;1,0}]=\xi_{2r}^2+r^2c_2+(2r-1)t_1\xi_{2r}+r(2r-1)t_1^2$ in $\CH_{\Gm^3}^*(\P(V_r))$.

\begin{lemma}\label{lem: deriving from W}
    For $r>0$ and $i \geq 0$, write ${M_{2r}'}_*( \xi_{2r}^i \cdot [W_{r;1,0}])=\alpha_0^{(i)}+\alpha_1^{(i)} t_1+\alpha_2^{(i)} t_1^2$, with $\alpha_j^{(i)} \in \CH^*_{\PGL_2}(\P(V_r))$.
    Then, we have
    \begin{align*}
        r(2r-1){M_{2r}}_*(\xi_{2r}^i)=\alpha_2^{(i)},\qquad (2r-1){M_{2r}}_*(\xi_{2r}^{i+1})=\alpha_1^{(i)},\qquad {M_{2r}}_*(\xi_{2r}^{i+2})=\alpha_0^{(i)}-r^2c_2{M_{2r}}_*(\xi_{2r}^i).
    \end{align*}
\end{lemma}
\begin{proof}
    It follows immediately from Proposition~\ref{prop:generalizeddecomposition}.
\end{proof}

\begin{lemma}\label{lem: pushforward W1100}
    We have ${M_2'}_*([W_{1;1,0}])=[W_{a;1,0}]\boxtimes[W_{b;1,0}]=\alpha_0+\alpha_1t_1+\alpha_2t_1^2$, with
    \begin{align*}
        \alpha_0=&\ \xi_{2a}^2\xi_{2b}^2+c_2(b^2\xi_{2a}^2+a^2\xi_{2b}^2)+a^2b^2c_2^2+c_3(b\xi_{2a}+a\xi_{2b}),\\
        \alpha_1=&\ \xi_{2a}\xi_{2b}((2b-1)\xi_{2a}+(2a-1)\xi_{2b})+c_2((2a-1)b^2\xi_{2a}+(2b-1)a^2\xi_{2b})\\
        &\ -(2a-1)(2b-1)c_2(b\xi_{2a}+a\xi_{2b})+abc_3,\\
        \alpha_2=&\ (2a-1)(2b-1)\xi_{2a}\xi_{2b}+b(2b-1)\xi_{2a}^2+a(2a-1)\xi_{2b}^2+ab(a+b-1)c_2.
    \end{align*}
\end{lemma}
\begin{proof}
    Immediate by definition and Remark~\ref{rmk: computation W[m;1,0,0]}, together with the relations $t_1^3+c_2t_1+c_3=0$ and $2c_3=0$.
\end{proof}

Putting everything together, we have:

\begin{corollary}\label{cor: computation M2}
    The image of ${M_2}_*$ is generated by
    \begin{align*}
        {M_2}_*(1)=&\ (2a-1)(2b-1)\xi_{2a}\xi_{2b}+b(2b-1)\xi_{2a}^2+a(2a-1)\xi_{2b}^2+ab\left(a+b-1\right)c_2,\\
        {M_2}_*(\xi_2)=&\ \xi_{2a}\xi_{2b}((2b-1)\xi_{2a}+(2a-1)\xi_{2b})+c_2\left((2a-1)b^2\xi_{2a}+(2b-1)a^2\xi_{2b}\right)\\
        &\ -(2a-1)(2b-1)c_2\left(b\xi_{2a}+a\xi_{2b}\right)+abc_3,\\
        {M_2}_*(\xi_2^2)=&\ \xi_{2a}^2\xi_{2b}^2-c_2\left(b(b-1)\xi_{2a}^2+a(a-1)\xi_{2b}^2\right)-(2a-1)(2b-1)c_2\xi_{2a}\xi_{2b}\\
        &\ +ab\left(ab-a-b+1\right)c_2^2+c_3\left(b\xi_{2a}+a\xi_{2b}\right).
    \end{align*}
\end{corollary}
\begin{proof}
    From the projective bundle formula and the projection formula we know that the image of ${M_2}_*$ is generated by the classes in the statement,which are computed in the above lemmas.
\end{proof}

\subsubsection{End of Computation}

In this section we conclude the proof of Theorems \ref{thm: Chow Dab} and \ref{thm: Chow n>1} by showing that the ideal generated by the classes in Lemmas \ref{lem: computation i1 i2}, \ref{lemma: M(1) and M(tau)} and Corollary \ref{cor: computation M2} generate the image of $i_*$. 

\begin{notation}
    Denote by $I_2$ the ideal in $\CH^*_{\GL_3}(\P(V_a) \times_\mathcal{S} \P(V_b))$ generated by the image of $F_*$, $G_*$, ${M_1}_*$ and ${M_2}_*$,and by $\widetilde{I}_2$ its extension in $\CH^*_{\Gm^3}(\P(V_a) \times_\mathcal{S} \P(V_b))$.
\end{notation}

\begin{lemma}\label{lem: projective polynomials n>1}
    Let $2\leq a\leq b$. Regard $p_a(\xi_{2a})$ and $p_b(\xi_{2b})$ of Lemma~\ref{lem: Chow product proj spaces} as elements of the ring $\Z[c_2,c_3,\xi_{2a},\xi_{2b}]/(2c_3)$, and $I_2$ as the ideal in $\Z[c_2,c_3,\xi_{2a},\xi_{2b}]/(2c_3)$ generated by the classes computed in Lemmas~\ref{lem: computation i1 i2},~\ref{lem: pushforward M1} and Corollary~\ref{cor: computation M2}. Then, $p_{a},p_{b}\in I_2$.
\end{lemma}
\begin{proof}
    We prove this for $p_a(\xi_{2a})$, as the argument for $p_b(\xi_{2b})$ is completely analogous. Let $\widetilde{I}_2$ be the extension of $I_2$ in $\Z[t_1, t_2, t_3, \xi_{2a}, \xi_{2b}]/(t_1 + t_2 + t_3, 2t_1t_2t_3)$, where the $c_i$ are regarded as elementary symmetric functions in the $t_i$. 

    First, we prove that $p_a(\xi_{2a})$ lies in $\widetilde{I}_2$ modulo $2t_3$.
    By Equation~\eqref{eqn: Chow P(Vm) restricted to Si}, working modulo $2t_3$ is equivalent to restricting to $\mathcal{S}_{0,0,2}$, and $p_a(\xi_{2a})$ restricts to the usual projective relation in the Chow ring of $\P(V_a)|_{\mathcal{S}_{0,0,2}} \cong \P^{2a}$. Therefore, $p_a(\xi_{2a})$ modulo $2t_3$ coincides with the polynomial obtained as the product of all coordinate hyperplanes. The restriction of ${F_2}_*([W_{1;1,0}]) = W_{a;2,0}$ to $\P(V_a)|_{\mathcal{S}_{0,0,2}} \cong \P^{2a}$ can be seen as the intersection of $4$ coordinate hyperplanes. 
    
    At this point, write $p_a(\xi_{2a})=\alpha+2t_3\beta$ for some $\alpha\in\widetilde{I_2}$ and $\beta\in \Z[t_1,t_2,t_3, \xi_{2a}, \xi_{2b}]/(t_1+t_2+t_3,2t_1t_2t_3)$. Multiplying by $t_1t_2$ and using the relation $2t_1t_2t_3=0$, we get $t_1t_2p_a(\xi_{2a})\in\widetilde{I}_2$. By Lemma~\ref{lemma:consbasisalgebraicversion}, it follows that $p_a(\xi_{2a})\in I_2$.
\end{proof}

\begin{remark}\label{rmk: a=1 case}
    When $a=1$ the map $F_2$ is not defined (same for $G_2$ if $b=1$), and the above proof does not apply.
    On the other end, in this case, only $M_1$ and $M_2$ appear, so the identity
    $$
    \CH^*(\cD_{2a,2b})= \frac{\CH^*_{\PGL_2}(\P(W_a) \times \P(W_b))}{I_2}
    $$
    is clear when $1=a \leq b$.
    
    When $g$ is odd, together with Lemma \ref{lemma: cartesian diagram for RH} and Proposition~\ref{prop: Chow of root gerbes}, this gives another method to compute $\CH^*(\RH_g^1)$. We leave to the reader to check the result agrees with that in Theorem \ref{thm: Chow for g odd n=1} after the substitutions $\xi_2=\gamma, \xi_{2g}=2t-\gamma$ (and the relation $2t=\xi_2+\xi_{2g}$).
\end{remark}

After, Proposition \ref{prop: Chow n>1 inverting 2}, we have only left to show that the images of the pushforwards ${M_r}_*$ lie in $I_2$ working $\otimes \Z_{(2)}$. From now on, unless otherwise specified, we work with $\Z_{(2)}$ coefficients.

\begin{notation}
    Let $a,b,r$ be integers such that $2 \leq r \leq 2a \leq 2b$. We say that the property $P(a,b,r)$ is satisfied if the image of ${M_r}_*$ in the $\PGL_2$-equivariant Chow ring of $(\P(W_a)\times\P(W_b))$ is contained in $I_2$. Given $l$-tuples $(n_1',\ldots,n_l')$ and $(n_1,\ldots,n_l)$, we write $(n_1',\ldots,n_l')\prec(n_1,\ldots,n_l)$ if the first $l$-tuple is smaller than the second with respect to the lexicographic ordering (starting from the left), and similarly for the non-strict inequality $\preceq$.
\end{notation}

We will prove by lexicographic induction on $(a, b, r)$ that the property $P(a, b, r)$ is always satisfied, starting from $(1, 1, 1)$; see Remark~\ref{rmk: a=1 case} for the case $a = 1$. Before starting the induction, we make a further reduction, showing that it suffices to consider integers $r$ that are powers of $2$, and in particular, even.

\begin{lemma}\label{lem: power of 2}
    Suppose that $r>2$ is not a power of 2. Then, $\mathrm{Im}({M_r}_*)$ is contained in the image of ${M_{r_1}}_*$ for some $r_1<r$.
\end{lemma}
\begin{proof}
    Write $r = 2^\ell + u$, with $0 \leq u < 2^\ell$ an integer.  
    If $u > 0$, set $m = 2^\ell$ and let $\phi: \P^m \times \P^{r-m} \to \P^r$ be the multiplication map, which has degree $\binom{r}{m}$. Since $\binom{r}{m}$ is odd, and we are working with $ \Z_{(2)}$ coefficients, the $\PGL_2$-equivariant pushforward $\phi_*$ is surjective. 

    The conclusion follows by observing that $M_r \circ \phi$ factors through $M_m$.
\end{proof}

The next Lemma shows that it is enough to prove that ${M_r}_*(1)\in I_2$ for every even $r$. 

\begin{lemma}\label{lem: enough Mr1}
    Let $1<r \leq a \leq b$ and suppose that ${M_{2l}}_*(1)\in I_2$ for every $l<r$. 
    Then, ${M_{2r}'}_*([W_{r;1,0}]) \in \widetilde{I}_2$ and ${M_{2r}}_*(\xi_{2r}^i) \in I_2$ for all $i>0$.
\end{lemma}
\begin{proof}
    We prove the statement by induction, the case $r=2$ being trivial.
    Consider the following commutative diagram of $\Gm^2$-equivariant maps:
    \[
    \begin{tikzcd}
        \P(V_{r-1})\times_{\mathcal{S}}\P(V_{a-r})\times_{\mathcal{S}}\P(V_{b-r})\arrow[rr,"\psi"]\arrow[d,"\phi"] & & \P(V_{r})\times_{\mathcal{S}}\P(V_{a-r})\times_{\mathcal{S}}\P(V_{b-r})\arrow[d,"M_{2r}'"]\\
        \P(V_{r-1})\times_{\mathcal{S}}\P(V_{a-r+1})\times_{\mathcal{S}}\P(V_{b-r+1})\arrow[rr,"M_{2r-2}'"] & & \P(V_{a})\times_{\mathcal{S}}\P(V_b)
    \end{tikzcd}
    \]
    where $\phi=1\times j_{a-r;1,0}\times j_{b-r;1,0}$ and $\psi=j_{r;1,0}\times1\times1$; see~\S\ref{sec: GL3-counterparts} for the definition of $j_{l;1,0}$. Since by definition ${j_{r;1,0}}_*(1)=[W_{r;1,0}]$, the commutativity of the diagram shows that $M_{2r*}'([W_{r;1,0}])$ is contained in the image of ${M_{2r-2}'}_*$, hence in $\widetilde{I}_2$ by hypothesis. By the projection formula, the same holds for $M_{2r*}'(\xi_r^i[W_{r;1,0}])$, for every $i\geq0$. Using Lemma~\ref{lem: deriving from W} and part $1$ in Proposition \ref{prop:generalizeddecomposition}, we obtain also ${M_{2r}}_*(\xi_{2r}^i) \in I_2$ for all $i>0$. Notice that this is also exploiting the fact that $2r-1$ is odd.
\end{proof}

We are left to show that ${M_{2^\ell}}_*(1)\in I_2\otimes\Z{(2)}$ for all $l\geq0$, with the cases $\ell=0,1$ obvious by definition of $I_2$.

\begin{lemma}
    Let $1\leq r\leq a\leq b$, and suppose that $P(a',b',r')$ holds for every $(1,1,1)\preceq(a',b',r')\prec(a,b,2r)$, with $a'\leq b'$. If $r<a$, then $M_{2r*}'([W_{a-r;1,0}])\in \widetilde{I}_2$; if $r\leq a<b$, then $M_{2r*}'([W_{b-r;1,0}])\in \widetilde{I}_2$.
\end{lemma}
\begin{proof}
    When $a=1$, the assertion is clear by Remark~\ref{rmk: a=1 case}, hence we assume $a\geq2$. Suppose that $r<a$ (the other case is analogous). The commutative diagram of $\Gm^3$-equivariant maps
    \begin{equation}\label{eq: diag Mra-2}
        \begin{tikzcd}
            \P(V_{r})\times_{\mathcal{S}}\P(V_{a-r-1})\times_{\mathcal{S}}\P(V_{b-r})\arrow[rr,"M_{2r}^{'(a-1,b)}"]\arrow[d,"\psi"] &&\P(V_{a-1})\times_{\mathcal{S}}\P(V_{b})\arrow[d,"\varphi"]\\
            \P(V_{r})\times_{\mathcal{S}}\P(V_{a-r})\times_{\mathcal{S}}\P(V_{b-r})\arrow[rr,"M_{2r}^{'(a,b)} "] && \P(V_{a})\times_{\mathcal{S}}\P(V_{b})
        \end{tikzcd}
    \end{equation}
    with $\psi=1\times j_{a-r;1,0}\times1$ and $\varphi=j_{a;1,0}\times1$, gives $M_{2r*}^{'(a,b)}([W_{a-r;1,0}])=\varphi_*(M_{2r*}^{'(a-1,b)}(1))$. Here, we use superscripts to indicate the targets of the various maps. By the inductive hypothesis, we can write
    \[
        M_{2r*}^{'(a-1,b)}(1)={M_{1*}^{'(a-1,b)}}(\alpha_1)+M_{2*}^{'(a-1,b)}(\alpha_2)+F_{1*}^{'(a-1,b)}(\beta_1)+F_{2*}^{'(a-1,b)}(\beta_2)+G_{1*}^{'(a-1,b)}(\gamma_1)+G_{2*}^{'(a-1,b)}(\gamma_2)
    \]
    for some classes $\alpha_i$, $\beta_i$ and $\gamma_i$. Also, the diagram above with $r=1$ shows that $\varphi_*(M_{2*}^{'(a-1,b)}(\alpha_2))=M_{2*}^{'(a,b)}(\psi_*(\alpha_2))$. Similarly, the commutative diagram
    \begin{equation}\label{eq: diag M1a-2}
    \begin{tikzcd}
        \mathcal{Y}_1\times_{\mathcal{S}}\mathcal{Y}_{2a-3}\times_{\mathcal{S}}\mathcal{Y}_{2b-1}\arrow[r,"M_1^{'(a-1,b)}"]\arrow[d,"\tau"] & \P(V_{a-1})\times_{\mathcal{S}}\P(V_{b})\arrow[d,"\varphi"]\\
        \mathcal{Y}_1\times_{\mathcal{S}}\mathcal{Y}_{2a-1}\times_{\mathcal{S}}\mathcal{Y}_{2b-1}\arrow[r,"M_1^{'(a,b)}"] & \P(V_{a})\times_{\mathcal{S}}\P(V_{b})
    \end{tikzcd}   
    \end{equation}
    with $\tau=1\times k_{a-1;1,0}\times1$, gives $\varphi_*({M_1'}_*(\alpha_1))={M_1'}_*(\tau_*(\alpha_1))$. We have completely analogous diagrams for the maps $F_i$ and $G_i$, which yield $\varphi_*(F_{i*}^{'(a-1,b)}(\beta_i))\in\mathrm{Im}(F_{i*})$ and $\varphi_*(G_{i*}^{'(a-1,b)}(\gamma_i))\in\mathrm{Im}(G_{i*})$. It follows that ${M_{2r}'}_*([W_{a;1,0}])=\varphi_*(M_{2r*}^{'(a-1)}(1))\in \widetilde{I}_2$.
\end{proof}

\begin{lemma}\label{lem: induction odd case}
    Let $1\leq r\leq a\leq b$, and suppose that $P(a',b',r')$ holds for every $(1,1,1)\preceq(a',b',r')\prec(a,b,2r)$, with $a'\leq b'$.
    If $a$ is odd and $r<a$ is even, or if $b$ is odd and $r<b$ is even, then $P(a,b,2r)$ is satisfied.
    Moreover, independently from the parity of $a$ and $b$, if $r<a$ then we have ${M_{2r}}_*(\xi_{2a-2r})\in I_2$, and if $r\leq a<b$ then ${M_{2r}}_*(\xi_{2b-2r})\in I_2$.
\end{lemma}
\begin{proof}
    By Remark~\ref{rmk: computation W[m;1,0,0]}, we know that
    \[
    [W_{a-r;1,0}]=\xi_{2a-2r}^2+(a-r)^2c_2+(2a-2r-1)t_1\xi_{2a-2r}+(a-r)(2a-2r-1)t_1^2
    \]
    and the previous Lemma showed that the pushforward of this class along $M_{2r}'$ is in $\widetilde{I}_2$. Then, Proposition~\ref{prop:generalizeddecomposition} implies that $(a-r)(2a-2r-1){M_{2r}}_*(1)\in I_2$ and $(2a-2r-1){M_{2r}}_*(\xi_{2a-2r})\in I_2$. Since $(2a-2r-1)$ is always odd, and we are working with $\Z_{(2)}$ coefficients, we get that ${M_{2r}}_*(\xi_{2a-2r})\in I_2\otimes\Z_{(2)}$. Moreover, if $a$ is odd, then $(a-r)(2a-2r-1)$ is also odd, so also ${M_{2r}}_*(1)\in I_2$. The same argument works for the case $b$ odd.
\end{proof}

Notice that if $a$ or $b$ is odd, and $P(a',b',r')$ is satisfied for every $(1,1,1)\preceq(a',b',r')\prec(a,b,2r)$, with $a'\leq b'$, then $P(a,b,2r)$ holds. Indeed, by Lemma~\ref{lem: enough Mr1} it is enough to show that ${M_{2r}}_*(1)\in I_2$ for $r>1$ a power of 2, which in turn follows from Lemma~\ref{lem: induction odd case}, as $r\not=a$ or $r\not=b$.

We next deal with the case when both $a$ and $b$ are even.

\begin{remark}\label{rmk: reduction generators I2 even case}
    Suppose that both $a$ and $b$ are even. Then, the ideal $I_2\otimes\Z_{(2)}$ is generated by
    \[
        2\xi_{2a},\quad 2\xi_{2b},\quad 2ac_2,\quad 2bc_2,\quad \xi_{2a}^4+c_2\xi_{2a}^2+c_3\xi_{2a},\quad \xi_{2b}^4+c_2\xi_{2b}^2+c_3\xi_{2b},\quad \xi_{2a}\xi_{2b}
    \]
    where the first six classes are obtained from Lemma~\ref{lem: computation i1 i2}, and the last one is ${M_2}_*(1)$ modulo the other relations. Notice that $M_{1*}$, ${M_2}_*(\xi_2)$ and ${M_2}_*(\xi_2^2)$ are superfluous.
\end{remark}

\begin{lemma}\label{lem: induction even case}
    Suppose that $2\leq a\leq b$ are even, and suppose that $P(a',b',r')$ holds for $(1,1,1)\preceq(a',b',r')\prec(a,b,r)$. Then $P(a,b,r)$ is satisfied.
\end{lemma}
\begin{proof}
    We aim to show that ${M_r}_*(1) = 0$ modulo the ideal $I_2 \otimes \mathbb{Z}{(2)}$ computed above. By Lemma~\ref{lem: power of 2}, we can assume that $r > 1$ is a power of $2$, and, in particular, that it is even.

    Thus, we can write ${M_r}_*(1) \in \mathbb{Z}[c_2, c_3, \xi_{2a}, \xi_{2b}]/(2c_3) + I_2$ as
    \begin{equation}\label{eq:m2r1general} p_{a^3}\xi_{2a}^3 + p_{b^3}\xi_{2b}^3 + p_{a^2}\xi_{2a}^2 + p_{b^2}\xi_{2b}^2 + p_a\xi_{2a} + p_b\xi_{2b} + p_1, \end{equation}
    where $p_j \in \mathbb{Z}_{(2)}[c_2, c_3]/(ac_2, bc_2, 2c_3)$ are polynomials with coefficients in ${0, 1}$, except possibly for $p_1$.
    By Lemma~\ref{lem: enough Mr1} and Lemma~\ref{lem: induction odd case}, for every $r$ we have ${M_{r}}_*(\xi_{r})\in I_2$ and for every $r<2a$ we have ${M_{r}}_*(\xi_{2a-r})\in I_2$. It follows that also $\xi_{2a}\cdot{M_{r}}_*(1)={M_{r}}_*(\xi_{r}+\xi_{2a-r})\in I_2$, if $r<2a$. If $r=2a$, then $M_{2a}^*(\xi_{2a})=\xi_{2a}=\xi_{r}$, hence we get also in this case $\xi_{2a}\cdot{M_{r}}_*(1)\in I_2$. Similarly, we have that $\xi_{2b}\cdot{M_{r}}_*(1)\in I_2$ for every $r\leq 2b$ as in the statement. Multiplying the equation~\eqref{eq:m2r1general} by $\xi_{2a}$ (and continuing working modulo $I_2$), we get
    \begin{equation}\label{eq:m2r1timesxia}
        p_{a^2}\xi_{2a}^3+(p_a-p_{a^3}c_2)\xi_{2a}^2+(p_1-c_3p_{a^3})\xi_{2a}\in I_2,
    \end{equation}
    while multiplying by $\xi_b$ we get
    \begin{equation}\label{eq:m2r1timesxib}
        p_{b^2}\xi_{2b}^3+(p_b-p_{b^3}c_2)\xi_{2b}^2+(p_1-c_3p_{b^3})\xi_{2b}\in I_2.
    \end{equation}
    Therefore, we can assume
    \begin{align*}
        p_{a^2}=p_{b^2}=0, && p_a=c_2p_{a^3}, && p_b=c_2p_{b^3}, && p_1=c_3p_{a^3}, && c_3p_{a^3}=c_3p_{b^3}
    \end{align*}
    in $\mathbb{Z}[c_2,c_3,\xi_{2a},\xi_{2b}]/(2c_3)+I_2$. Since the annihilator of $c_3$ in the ring $\mathbb{Z}[c_2, c_3, \xi_{2a}, \xi_{2b}]/(2c_3, I_2)$ is the ideal generated by $2$, and $p_{a^3}, p_{b^3} \in \mathbb{Z}_{(2)}[c_2, c_3]/(ac_2, bc_2, 2c_3)$ are polynomials with coefficients in ${0, 1}$, we can also assume that $p_{a^3} = p_{b^3}$. Finally, we show that
    \begin{equation}\label{eq:onlyextrarelation}
        \alpha:=p(c_2,c_3)(\xi_{2a}^3+c_2\xi_{2a}+\xi_{2b}^3+c_2\xi_{2b}+c_3)
    \end{equation}
    is non-zero in $\mathrm{CH}^*_{\Gm^3}((\mathbb{P}(V_a) \times_{\mathcal{S}} \mathbb{P}(V_b)) \setminus \Delta')$, for any polynomial $p(c_2, c_3)$ of degree $r - 3$ in $c_2$ and $c_3$ with coefficients in ${0, 1}$, except when $p$ is identically $0$.

    We can assume $r > 4$ and that it is a power of $2$, since for $r = 2$, the statement of the lemma is already known to hold, and for $r = 4$, the degree of $p(c_2, c_3)$ is $r - 3 = 1$, which forces it to be $0$, as both $c_2$ and $c_3$ have degrees greater than $1$. 
    
    Consider a point $(q, f, g) \in (\mathbb{P}(V_a) \times_{\mathcal{S}} \mathbb{P}(V_b)) \setminus \Delta'$ with a stabilizer containing the subgroup $G'$ of $\mathbb{G}_m^3$ consisting of triples $(\pm 1, \pm 1, \pm 1)$ whose coordinates multiply to $1$. The group $G'$ can be viewed as a subgroup of $\mu_2^{\times 3}$, isomorphic to $\mu_2^{\times 2}$.  

    For example, let the conic $q$ be given by $x_0^2 + x_1^2 + x_2^2$ and take $f$ and $g$ of the form 
    \[
    f = q^{a/2} + f_1(x_1^2, x_2^2), \quad 
    g = q^{b/2} + g_1(x_1^2, x_2^2),
    \]  
    where $f_1$ and $g_1$ are general polynomials of degree $a/2$ and $b/2$, respectively. Then, the intersection of the zero locus of $fg$ with $q$ is smooth, and the coefficients of $f$ and $g$ relative to $x_0^a$ and $x_0^b$ are non-zero, respectively.
    Then, the stabilizer of $(q,f,g)$ is exactly the group $G'$ described above. This induces the monomorphism of the residual gerbe
    \[
    \begin{tikzcd}
        \mathrm{BG'}\arrow[r,hookrightarrow,"\iota"] & \left[\frac{(\P(V_{a})\times_{\mathcal{S}}\P(V_{b}))\setminus\Delta'}{\Gm^3}\right]
    \end{tikzcd}
    \]
    which in turn defines a pullback map (see~\cite[Lemma 2.4]{MRV06} for the Chow ring of $\mathrm{BG}$)
    \[
    \begin{tikzcd}
        \mathrm{CH}_{\Gm^3}^*((\P(V_{a})\times_{\mathcal{S}}\P(V_{b}))\setminus\Delta')\arrow[r,"\iota^*"] & \mathrm{CH}^*(\mathrm{BG'})\cong\frac{\mathbb{Z}[\mu_1,\mu_2,\mu_3]}{(2\mu_0,2\mu_1,2\mu_2,\mu_0+\mu_1+\mu_2)}.
    \end{tikzcd}
    \]
    Since the composite of $\iota$ with the projection to $\mathrm{B\Gm^3}$ is the map induced by the inclusion $G'\subset \Gm^3$, it is clear that $\iota^*$ sends $\lambda_i$ to $\mu_i$. Moreover, the immersion $\iota$ factors through the immersion of the open subscheme where the coefficient of the conic $q$ relative to $x_2^2$ is non-zero and the coefficients of the polynomial $f$ and $g$ relative to $x_0^{a}$ and $x_0^{b}$ respectively are also non-zero. On such a locus, $\xi_{2a}$ and $\xi_{2b}$ restrict to $at_1$ and $bt_1$, respectively; and, since $a$ and $b$ are even, it follows that they restrict to $0$ on $\mathrm{BG'}$. From this discussion, we get that the restriction of $\alpha$ in Equation \eqref{eq:onlyextrarelation} to $\mathrm{CH}^*(\mathrm{BG'})$ is equal to the restriction of $p(c_2,c_3)c_3$. However, that class is non-zero in $\mathrm{CH}^*(\mathrm{BG'})$ unless $p=0$, as $\CH^*(B\Gm^3)\rightarrow\CH^*(BG')$ has kernel contained in the ideal generated by $2$, while $p$ has coefficients 0 and 1. It follows that $p=0$, and we are done.
\end{proof}

Theorems ~\ref{thm: Chow Dab} and \ref{thm: Chow n>1} are now immediate.

\begin{proof}[Proof of Theorem~\ref{thm: Chow Dab}]
    We show that $P(a,b,r)$ is satisfied by induction on $(1,1,1)\preceq(a,b,r)$ with the lexicographic ordering and $r\leq2a\leq2b$. Notice that in this setting $I_2$ contains also the polynomials $p_a(\xi_{2a})$ and $p_b(\xi_{2a})$ by Lemma \ref{lem: projective polynomials n>1}. The base case is with $a=1$, where it is obvious. Therefore, it is enough to prove the inductive step with $a>1$; from now on assume that $P(a',b',r')$ holds for every $(1,1,1)\preceq(a',b',r')\prec(a,b,r)$ with $a'\leq b'$. By Lemma~\ref{lem: power of 2} we may assume that $r\geq4$ and that it is a power of 2, while by Lemma~\ref{lem: enough Mr1} it is enough to show that ${M_r}_*(1)\in I_2$. If $a$ or $b$ are odd, then Lemma~\ref{lem: induction odd case} concludes, while if they are both even then $P(a,b,r)$ is satisfied by Lemma~\ref{lem: induction even case}. This shows that $P(a,b,r)$ is satisfied for every $r$.
\end{proof}

\begin{proof}[Proof of Theorem \ref{thm: Chow n>1}]

This follows from Theorem \ref{thm: Chow Dab}, Lemma \ref{lemma: cartesian diagram for RH} and the formula for the Chow rings of root stacks in Proposition \ref{prop: Chow of root gerbes}.
    
\end{proof}

\subsection{Geometric interpretation of the generators when $g$ is odd}\label{subsec: interpretation generators n>1}

We saw in Theorem \ref{thm: Chow n>1} and Remark \ref{rmk: a=1 case} that $\CH^*(\RH_g^n)$ for odd $g$ is generated by the classes $c_2, c_3, t, \xi_{2n}, \xi_{2g+2-2n}$. In this section, we provide an interpretation of these classes as Chern classes of certain natural vector bundles on $\RH_g^n$.  

With reference to the diagram in Lemma \ref{lemma: cartesian diagram for RH}, all these classes, except for $\xi_{2n}$ and $\xi_{2g+2-2n}$, are pulled back from $\cH_g$, where their interpretations are already known. For the reader's convenience, we recall them here.  

The class $t$ is the pullback of the first Chern class of the line bundle $\mathcal{L}_g$, while the classes $c_i$ are the Chern classes of the vector bundle $\mathcal{E}_g$. Both $\mathcal{L}_g$ and $\mathcal{E}_g$ are the bundles introduced in \S\ref{subsec: interpretation generators n=1}.  We refer to  \cite[Theorem 7.2]{FV11} or to the discussion after ~\cite[Theorem 7.2]{DL18} for the proof.

The classes $\xi_{2n}$ and $\xi_{2g+2-2n}$ are instead pulled-back from $\mathcal{D}_{2n}$ and $\mathcal{D}_{2g+2-2n}$ respectively.
For all $ m \geq 1$, let $\mathcal{N}_m$ be the line bundle on $\mathcal{D}_{2m}$ defined by $\pi_* \omega_\pi^{\otimes m}(\mathcal{F}_{2m})$ where $\mathcal{F}_{2m} \subseteq \mathcal{P} \xrightarrow{\pi} \mathcal{D}_{2m}$ are the universal divisor and the universal Brauer Severi variety of relative dimension $1$ over $\mathcal{D}_{2m}$. 

\begin{lemma}
    We have 
    $$
    c_1(\mathcal{N}_m)=\xi_{2m}.
    $$
\end{lemma}
\begin{proof}
    Consider the commutative diagram with cartesian squares
    \[
    \begin{tikzcd}
        F_{2m}\arrow[d]\arrow[r,hookrightarrow] & \P^1\times(\P(W_m)\setminus\uD)\arrow[r,"\rho"]\arrow[d] & \P(W_m)\setminus\uD\arrow[d,"\phi"]\arrow[r] & \mathrm{Spec}(k)\arrow[d]\\
        \mathcal{F}_{2m}\arrow[r,hookrightarrow] & \mathcal{P}\arrow[r,"\pi"] & \mathcal{D}_{2m}\cong[(\P(W_m)\setminus\uD)/\PGL_2]\arrow[r] & B\PGL_2
    \end{tikzcd}
    \]
    Note that the pullback of $\mathcal{P} \to \mathcal{D}_{2m}$ along $\P(W_m) \smallsetminus \uD \to \mathcal{D}_{2m}$ is a trivial $\P^1$-fibration. Indeed, as shown above, the map $\P(W_m) \smallsetminus \uD \to B\PGL_2$ that defines it factors through $\mathrm{Spec}(k) \to B\PGL_2$.

    Also, $\cO(F_{2m})=\cO(2m,1)$ and $\omega_{\rho}^{\otimes m}= \cO(-2m,0)$. Thus, $\omega_{\rho}^{\otimes m}(F_{2m})= \cO(0,1)$ and 
    $$
    \phi^* \mathcal{N}_m=\rho_*(\omega_{\rho}^{\otimes m}(F_{2m}))=\cO_{\P(W_m) \smallsetminus \uD}(1).
    $$
    The conclusion follows from the fact that the pullback 
    \[
    \begin{tikzcd}
        \phi^*: \mathrm{Pic}(\mathcal{D}_{2m})= \frac{\Z[\xi_{2m}]}{((4m-2)\xi_{2m})}\arrow[r] & \mathrm{Pic}(\P(W_m) \smallsetminus \uD)=\frac{\Z[\xi_{2m}]}{((4m-2)\xi_{2m})}
    \end{tikzcd}
    \]
    is an isomorphism by~\cite[Theorem 1.2]{EH22} and the projection bundle formula.
\end{proof}

In particular, $\xi_{2n}=c_1(\mathcal{N}_n)$ and $\xi_{2g+2-2n}=c_1(\mathcal{N}_{g+1-n})$.

\bibliographystyle{amsalpha}
\bibliography{library}

$\,$\
\noindent

$\,$\
\noindent
\textsc{Department of Pure Mathematics {\it \&} Mathematical Statistics, 
University of Cambridge, Cambridge, UK}

\textit{e-mail address:} \href{mailto:au270@cam.ac.uk}{ac2758@cam.ac.uk}

$\,$\
\noindent

$\,$\
\noindent
\textsc{Department of Pure Mathematics, Brown University, 151 Thayer Street, Providence, RI 02912, USA}

\textit{e-mail address:} \href{mailto:alberto_landi@brown.edu}{alberto\_landi@brown.edu}

\end{document}